\documentclass[11pt]{article}
\usepackage{color,ulem,cancel}
\usepackage{amssymb,amsfonts,amsthm, mathtools}
\usepackage{todonotes,soul}
\usepackage[english]{babel}
\usepackage[utf8]{inputenc}
\usepackage[margin=22mm]{geometry}
\usepackage{graphicx,overpic}
\usepackage{times,fourier}
\usepackage{tikz}
\usetikzlibrary{shapes,calc,math}
\usepackage{pgf}
\usepackage{pgfplots}
\pgfplotsset{compat=newest}
\usetikzlibrary{decorations.markings}
\tikzset{->-/.style={decoration={ markings, mark=at position #1 with {\arrow{>}}},postaction={decorate}}}
\usepackage{float}
\usepackage[utf8]{inputenc}
\usepackage[absolute]{textpos}
\usepackage{ulem}
\usepackage{emptypage}
\usepackage{amsbsy,mathrsfs}
\usepackage{subcaption}
\usepackage{xcolor}
\usepackage{multicol}
\usepackage{thmtools, thm-restate}
\usepackage[colorlinks=true,citecolor=blue, hypertexnames=false]{hyperref}

\def\ov{\overline}
\makeatletter
\renewcommand\section{\@startsection {section}{1}{\z@}%
  {-2ex \@plus -1ex \@minus -.2ex}%
  {1ex \@plus.1ex}%
  {\normalfont\bf\sffamily\color{darkblue}}}
\renewcommand\subsection{\@startsection{subsection}{2}{\z@}%
  {-1.75ex\@plus -0.4ex \@minus -.2ex}%
  {0.6ex \@plus .1ex}%
  {\normalfont\small\bf\sffamily}}
\renewcommand\subsubsection{\@startsection{subsubsection}{3}{\z@}%
  {-0.6ex\@plus -0.2ex \@minus -.2ex}%
  {0.4ex \@plus .1ex}%
  {\normalfont\normalsize\it}}
\renewcommand\paragraph{\@startsection{paragraph}{4}{\z@}%
  {0.2ex \@plus0.2ex \@minus0.1ex}{-0.5em}%
  {\normalfont\normalsize\bfseries}}
\def\ps@headings{%
  \let\@oddfoot\@empty
  \let\@evenfoot\@empty

  \def\@evenhead{\small\sffamily\thepage\hfil\slshape\leftmark}%
  \def\@oddhead{\small\sffamily{\slshape\rightmark}\hfil\thepage}%
  \let\@mkboth\markboth
  \def\chaptermark##1{\markboth{{\ifnum \c@secnumdepth >\m@ne
		\if@mainmatter \@chapapp\ \thechapter. \ \fi \fi ##1}}{}}%
  \def\sectionmark##1{\markright {{\ifnum \c@secnumdepth >\z@
		\thesection. \ \fi ##1}}}}
\def\fbf#1{\setbox0=\hbox{$#1$}\kern-0.10\wd0
  \lower0.02em\copy0\kern-\wd0 \lower0.02em\hbox{\kern+0.04em\copy0}\kern-\wd0
  \raise0.00em\copy0\kern-\wd0 \raise0.00em\hbox{\kern-0.04em\box0}}
\def\overl@ss#1#2{\vcenter{\offinterlineskip
        \ialign{$\m@th#1\hfil##\hfil$\crcr#2\crcr<\crcr } }}
\def\gl{\mathrel{\mathpalette\overl@ss>}}
\makeatother

\advance\textwidth -10mm
\advance\textheight -14mm
\advance\hoffset 4mm
\advance\voffset 4mm

\advance\parskip 4pt
\numberwithin{equation}{section}
1

\declaretheorem[name=Theorem, parent=section]{theorem}
\declaretheorem[name=Lemma,sibling=theorem]{lemma}
\declaretheorem[name=Corollary, sibling=theorem]{corollary}

\declaretheorem[name=Remark, sibling=theorem]{remark}
\declaretheorem[name=Fact, sibling=theorem]{Fact}
\declaretheorem[name=Proposition, sibling=theorem]{proposition}
\declaretheorem[name=Assumption, sibling=theorem]{assumption}

\declaretheorem[name=RHP, parent=section]{RHP}

\newcommand{\repeatedproptitle}{}
\newtheorem*{repeatedprop}{\protect\repeatedproptitle}

\makeatletter
\newcommand{\address}[1]{\gdef\@address{#1}}
\gdef\@address{} % default empty

\renewcommand{\maketitle}{%
  \begin{center}
    {\LARGE\bfseries\sffamily \@title\par}
    \vspace{1.4ex}
    {\large \@author\par}
    \vspace{0.6ex}
    {\itshape \@address\par}
    \vspace{0.2ex}
    {\small \@date\par}
  \end{center}
  \vspace{1.4\bigskipamount}
}
\makeatother

\def\be{\begin{equation}}
\def\ee{\end{equation}}
\def\bse{\begin{subequations}}
\def\ese{\end{subequations}}

\definecolor{deeppurple}{rgb}{0.5, 0, 0.7}
\definecolor{deeppurple}{rgb}{0.5, 0, 0.7}
\definecolor{darkblue}{rgb}{0, 0, 0.7}

\definecolor{deeppurple}{rgb}{0.5, 0, 0.7}

\def\half{{\textstyle\frac12}}

\def\dn{\mathop{\rm dn}\nolimits}

\def\Real{\mathbb{R}}
\def\R{\mathbb{R}}
\def\Complex{\mathbb{C}}

\newcommand{\bigo}[1]{\mathcal{O} \left( #1 \right) }

\def\Z{\mathbb{Z}}
\def\i{\mathrm{i}}

\def\Re{\mathop{\rm Re}\nolimits}
\def\Im{\mathop{\rm Im}\nolimits}
\def\Res{\mathop{\rm Res}\limits}

\def\d{\mathrm{d}}

\let\t=\theta
\def\e{\mathop{\rm e}\nolimits}
\def\sn{\mathrm{sn}}
\def\cn{\mathrm{cn}}
\def\dn{\mathrm{dn}}
\def\@#1{{\mathbf{#1}}}
\def\_#1{{\mathsf{#1}}}
\let\~=\tilde
\let\==\bar
\let\^=\hat
\let\<=\langle
\let\>=\rangle

\def\min{\mathop{\rm min}\nolimits}
\def\max{\mathop{\rm max}\nolimits}

\def\note[#1]{\marginpar{\color{blue}[#1]}}

\def\C{{\mathbb C}}
\def\R{{\mathbb R}}
\def\N{{\mathbb N}}

\def\Z{{\mathbb Z}}

\def\1{{\bf 1}}

\def\A{\mathcal{A}}
\def\e{\mathrm{e}}

\makeatother
\def \pa{\partial}

\advance\textheight 12mm
\advance\voffset -6mm
\allowdisplaybreaks

\newcommand{\Etl}{\begin{bsmallmatrix} 1 & 0 \\ 0 & 0 \end{bsmallmatrix}}

\newcommand{\Ebl}{\begin{bsmallmatrix} 0 & 0 \\ 0 & 1 \end{bsmallmatrix}}

\newcommand{\od} [3] [ ] {\frac{\mathrm{d}^{#1}#2}{\mathrm{d}{#3}^{#1}} }

\renewcommand*{\d}{\mathrm{d}}

\newcommand{\ii}{\mathrm{i}}

\begin{document}
\pagestyle{plain}

\title{Direct Scattering of the Focusing Nonlinear Schr\"odinger Equation with Step-like Oscillatory Initial Data}
\author{Tamara Grava $^1$$^{,1a}$, Robert Jenkins $^2$, Xiaofan Zhang $^{3,4}$ and Zechuan Zhang $^1$}
\address{$^1$ Scuola Internazionale Superiore di Studi Avanzati (SISSA), Trieste, 34136 - Italy and INFN sezione di Trieste\\
$^{1a}$ School of Mathematics, Fry Building, University of Bristol - UK\\ 
$^2$ 
University of Central Florida, Orlando, FL 32816 - USA\\
$^3$ Hunan University of Science and Technology, Xiangtan, 411201 - China \\
$^4$ China University of Mining and Technology, Xuzhou 221116 - China}
\maketitle

\begin{abstract}
\noindent
 In this manuscript we set up the direct and inverse scattering problems for step-like  traveling wave  solutions of 
the nonlinear Schr\"odinger equation. Specifically,  we consider  initial data $u_0(x)$  of the form  
 \begin{equation*}
     u_0(x)=
     \begin{cases}
     u_e^\ell(x)& \mbox{as $x<0$},
     \\
      u_e^r(x)& \mbox{as $x>0$},
      \end{cases}
     \end{equation*}
where $u_e^\ell(x)$ and $u_e^r(x)$ are elliptic traveling wave solutions of NLS and more in general
 \begin{equation*}
     u_0(x)\to
     \begin{cases}
     u_e^\ell(x)& \mbox{as $x\to-\infty$},
     \\
      u_e^r(x)& \mbox{as $x\to\infty$}.
      \end{cases}
     \end{equation*}
Under the  assumption that $u_0-u_e^\ell \in L^{1,2}(\R^-)$ and  $u_0-u_e^r \in L^{1,2}(\R^+)$,  we formulate the direct scattering problem and establish analytic properties of the scattering data. We then formulate the inverse problem
as a Riemann--Hilbert problem and prove its solvability, thus proving the existence of the  nonlinear Schr\"odinger  solution with the above step-like initial data.
Finally, we observe that our  Riemann--Hilbert formulation is a special case of the one arising for full soliton-gas initial data.
\end{abstract}

\medskip
\tableofcontents

%%%%%%%%%%%%%%%%%%%%%%%%%%%%%%%%%%%%%%%%%%%%%%%%%%%%%%%%%%%%%%%%%%%%%%%%%%%%%%%%%%%%%%%%%%%%%%%%%%%%%%%%%%%%%%%%%%%%%%%
%%%%%%%%%%%%%
\section{Introduction } 
\label{s:intro}

This paper investigates the direct and inverse scattering problems for the cubic focusing nonlinear Schr\"odinger
(NLS) equation with step-like oscillatory initial data, namely with two different    elliptic quasi-periodic
travelling wave solutions  at $x=\pm \infty$.
The focusing NLS equation is given by
\be
\i u_t + \half u_{xx} +  |u|^2 u = 0\,,\quad x\in\R,\quad  t\in\R^+, 
\label{e:nls}
\ee 
where $u = u(x,t)\in\C$ and subscripts denote partial derivatives.
This equation is a fundamental integrable model with applications to water waves \cite{AblowitzSegur81, ZS72}, nonlinear fiber optics \cite{HT73}, plasma physics \cite{Zakharov71}, and Bose-Einstein condensates \cite{PS02}.
It  is  one of the prototypical integrable partial differential equations (PDEs), and it can be expressed as the compatibility condition of two linear equations,  the so-called Lax pair,  introduced by Zakharov and Shabat in 1972 \cite{ZS72},
 which takes the form 
\bse%
\label{e:NLSLP}
\begin{align}
	&W_x = \mathcal{L}(u,z)\,W\,, &&\mathcal{L}(u,z) = -\i z\sigma_3 + U(x,t)\,,
	\label{e:zs}
	\\
	&W_t = \mathcal{B}(u,z)\,W\,,  &&\mathcal{B}(u,z) = -\i z^2\sigma_3 + zU - \frac{1}{2}\i\sigma_3(U^2-U_x)\,,
	%+\frac{i}{2}\omega_0\sigma_3\,,
	\label{e:NLSLP2}
\end{align}
\ese
where $W(x,t;z)\in  \mbox{Mat}(2\times 2,\Complex)$ is the fundamental matrix solution  of the above linear system,  $z\in\Complex$ is the spectral parameter,  and  
\be
U(x,t)=\begin{pmatrix}
    0 & u(x,t)\\
    -\overline{u(x,t)} & 0
\end{pmatrix}, \qquad \sigma_3=\begin{pmatrix*}[r]1&0\\0&-1\end{pmatrix*}.
\ee
  We call  equation  \eqref{e:zs} the Zakharov--Shabat (ZS) spectral problem associated with  the NLS equation. 
The compatibility of the above two equations gives
\[
\mathcal{L}_t-\mathcal {B}_x+[\mathcal{L}, \mathcal{B}]=0,
\]
that is   equivalent to  the NLS equation \eqref{e:nls}.
The NLS equation  has a family of quasi-periodic elliptic travelling wave solutions of the form
\begin{equation}\label{e:qpdef}
  u_e(x,t; v,x_0) \,=\, \e^{-\i \omega_0 t}\,\e^{-\i  \big( -vx + \frac{v^2}2t\big)}\psi(x-vt-x_0),
  \quad x \in \R~, \quad  t \in \R,
\end{equation}
where  $\psi : \R \to \C$  
is a quasi-periodic function  while  $|\psi(x)|$ is a periodic   elliptic function,  $v$ is a real constant that defines the velocity of the wave, $\omega_0\in\R$  and $x_0\in\R$ is a phase shift  that is defined modulo the period of the elliptic wave  (see Remark~\ref{rmk:windinga} ).  The function $\psi$ satisfies the ordinary differential equation (ODE)
\begin{equation}\label{e:snls}
 \frac{1}{2} \psi_{xx}(x) + \omega_0 \psi(x) +  |\psi(x)|^2 \psi(x) \,=\, 0~, \quad
  x \in \R~.
\end{equation}
The general solution  of the ODE  \eqref{e:snls} can be written as  $\psi(x)=\phi(x) \e^{\i\theta(x)},$ where $\phi$ and $\theta $ are expressed via elliptic functions (see \eqref{E:defs}).

When $\psi=\e^{\i p_0x}\psi_0$ with constants $p_0$ and  $\psi_0$, one obtains  the 
 {\it plane wave} solution.
 It is well-known  that  plane waves
are dynamically  unstable  in the focusing case \cite{Zakharov68,GH1}.  The essence of the phenomenon is well-understood:  linearized stability analysis shows that a uniform background is unstable to long wavelength perturbations \cite{ZO}. 

The direct and inverse scattering problem  for  non-zero boundary conditions,  and in particular for  plane waves   with different phases  at $x=\pm \infty$, 
 has been developed by  Biondini and Kova\v{c}i\v{c} in  \cite{BG14}. Demontis, Prinari, van der Mee and Vitale further develop the inverse scattering transform (IST) for asymmetric boundary  plane wave initial  data \cite{DPVV14}. 
  A robust formulation of the inverse scattering transform for the focusing NLS equation with nonzero boundary conditions,
 including arbitrary-order poles and spectral singularities, was developed by Bilman and Miller \cite{BP}. 
  The long-time dynamics of plane wave solutions with long-range perturbation has been considered in \cite{BM17}.   The long-time dynamics of plane wave solutions with different phases as $x\to\pm \infty$ has been obtained in \cite{BS1, BS2, BS3}.

In this manuscript, we  concentrate on the less explored situation where $|\psi|$ is an  elliptic  function,  and study the direct and inverse scattering problem  with initial data asymptotic  to  two  distinct elliptic solutions as $x\to \pm\infty$.
 Egorova et al.\ \cite{Egorova+24} established  the IST  for the Korteweg-de Vries  equation with  step-like quasi-periodic (finite-gap) backgrounds. However,  the  extension to the NLS equation presents new challenges. 
     The elliptic solutions are linearly unstable \cite{Deconinck}, and one of the motivations to develop the direct and inverse scattering  problem  for step-like elliptic wave solutions is to understand the long-time
     behaviour of these solutions. 
     
     We consider initial data of the form
     \begin{equation}
     \label{step_intro}
     u_0(x)=
     \begin{cases}
     u_e^\ell(x)& \mbox{as $x<0$},
     \\
      u_e^r(x)& \mbox{as $x>0$},
      \end{cases}
     \end{equation}
 where $u_e^\ell(x)= u_e^\ell(x,t=0; v^\ell, x_0^\ell)$ and $   u_e^r(x)= u_e^r(x,t= 0;v^r,x_0^r)$ are two distinct  elliptic solutions to the NLS equation of the form \eqref{e:qpdef} with velocities $ v^\ell $  and $v^r$  respectively,  or more in general initial data of the form
     \begin{equation}
     \label{initial_data}
     u_0(x)\to\left\{
     \begin{array}{ll}
     u_e^\ell(x)& \mbox{as $x\to-\infty$}\\
      u_e^r(x)& \mbox{as $x\to\infty$}
      \end{array}
      \right.
     \end{equation}
        where $u_0(x)$ approaches  the  background $u_e^\ell(x)$ and $ u_e^r(x)$ sufficiently fast to be integrable (see below for a precise definition).
         
Our result is the formulation of the direct and inverse scattering problem for such initial data.

\paragraph{Notational conventions.}
We denote the open upper and lower half-planes by $\C^\pm=\{z\in\C:\pm\Im z>0\}$, and their closures by $\ov{\C^\pm}=\C^\pm\cup\R$.
Given a complex number $z$, $\overline{z}$ denotes its complex conjugate. For complex valued functions $f$, the function $f^*$ denotes the Schwarz reflection $f^*(z) = \overline{f(\overline{z})}$. The boundary values of a function $f(z)$ along an oriented contour $\Sigma$ are denoted by $f(z_+)$ and $f(z_-)$ where the subscript $+$ (resp. $-$) denotes the limiting value from the left (resp. right) side of $\Sigma$ with respect to its orientation. 
Given a set $\mathcal{U}$, $\partial \mathcal{U}$ denotes the boundary of that set; in particular for a finite union of simple contours $\Sigma$, $\partial \Sigma$ denotes the set of endpoints of $\Sigma$. 
Throughout the paper we will use the Pauli matrices
\[
	\sigma_1 = \begin{pmatrix} 0 & 1 \\ 1 & 0 \end{pmatrix}, \qquad
	\sigma_2 = \begin{pmatrix*}[r] 0 & -\i \\ \i & 0 \end{pmatrix*}, \qquad
	\sigma_3 = \begin{pmatrix*}[r] 1 & 0 \\ 0 & -1 \end{pmatrix*}.
\]
For a scalar function $f$, the matrix function $f^{\sigma_3}$ is given by $f(z)^{\sigma_3} = \begin{psmallmatrix} f(z) & 0 \\ 0 & f(z)^{-1} \end{psmallmatrix}$.
A function $f$ is in the Sobolev spaces $\mathcal{W}^{k,p}(\R)$ provided that $f$ and its weak derivatives up to order $k$ are all in $L^p(\R)$.  
We also define the weighted $L^p$ spaces
\[
	L^{p,q}(\R) = \left\{ f \in L^p(\R) \, :\, \left<\, \cdot\, \right>^q f \in L^p(\R)  \right\},
\]
where $\left< x \right> := \sqrt{1+x^2}$ denotes the Japanese bracket. 
Function spaces on subsets of $\R$ are defined similarly. In  particular   throughout the  manuscript 
we use    the subsets    $\R^\pm$ that represent any  half lines.

\subsection{Statement of the result.}  
Given an elliptic travelling wave $u_e(x) = u_e(x,0; v , x_0)$ of the form \eqref{e:qpdef},  the spectrum of the corresponding operator $\mathcal{L}$ is the set of values $z\in\C$ such that the fundamental matrix solution $W_e(x,t=0;z)$  of \eqref{e:zs} is bounded  for all $x\in\R$. 
The spectrum consists of the real line $\R$  and two arcs  that we denote by $\Sigma$, namely,
\begin{equation}
\label{spctrum0}
	\Sigma \cup \R:=\left\{z\in \C \,\Big\vert\ |W_e(z,x,t=0)|<\infty,  \,\forall x\in\R\right\}.
\end{equation}

\begin{figure}[t]
	\centering
	
	\begin{subfigure}[b]{0.4\textwidth}
		\centering
		
		\begin{overpic}[
			width=\linewidth,
			%grid,
			%tics=5
			]{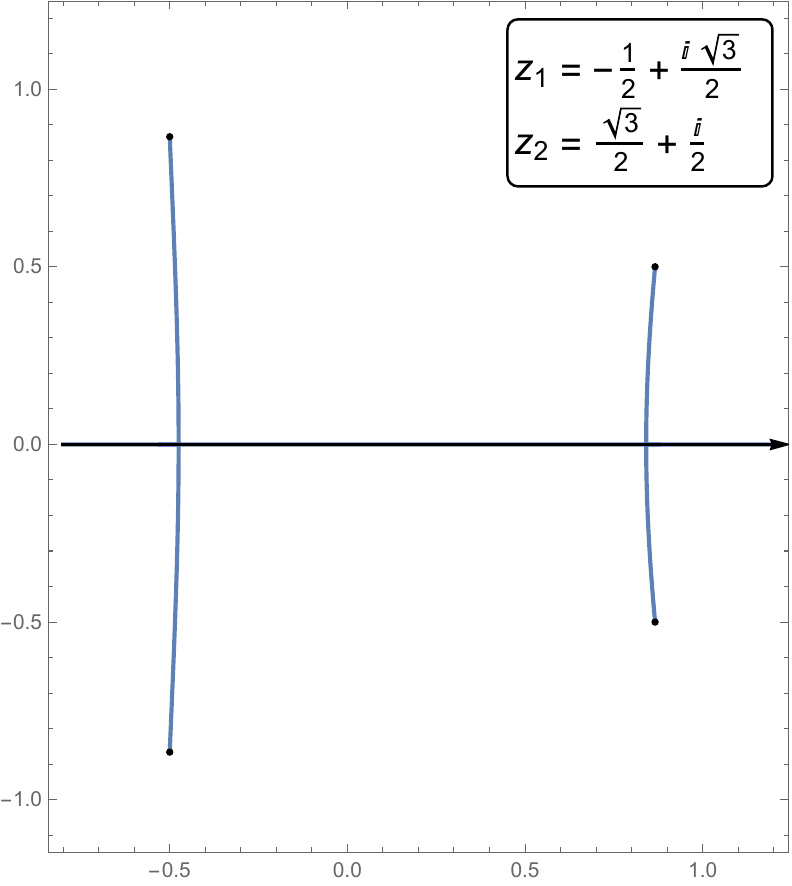}
			\put(21,82){$z_1$}
			\put(76,68.5){$z_2$}
		    \put(21,13){$\bar{z}_1$}
			\put(76,28.5){$\bar{z}_2$}
\end{overpic}
		\caption{}
	\end{subfigure}
	\hspace{2em}
	\begin{subfigure}[b]{0.4\textwidth}
		\centering
		\begin{overpic}[
			width=\linewidth,
			%grid,
			%tics=5
			]{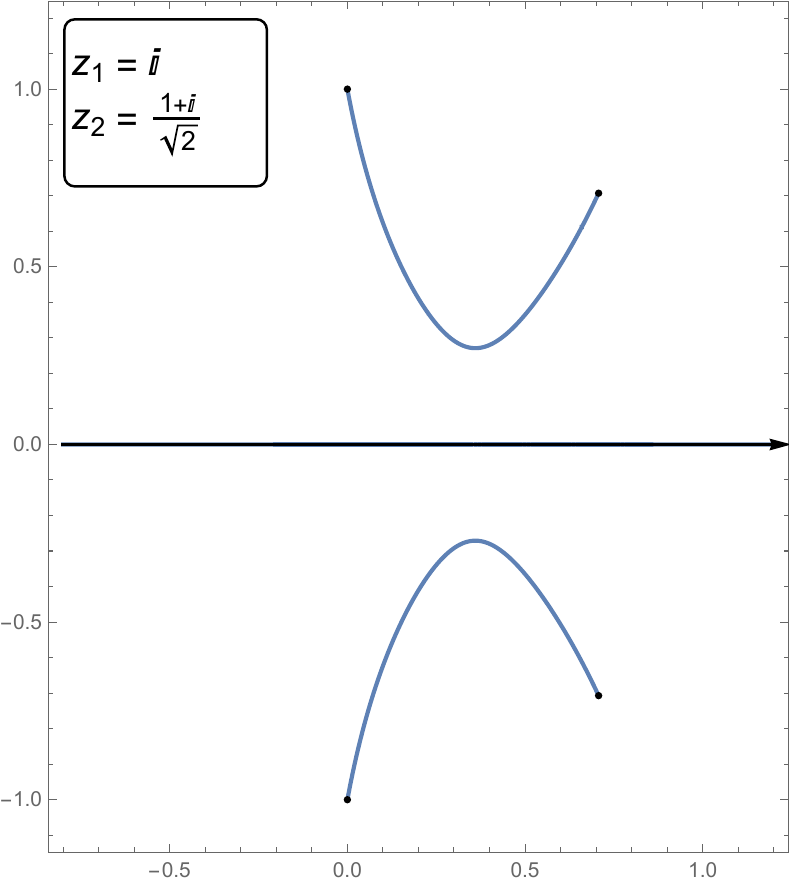}
			\put(33,89){$z_1$}
			\put(70,78){$z_2$}
			\put(33,10){$\bar{z}_1$}
			\put(68,20){$\bar{z}_2$}
		\end{overpic}
		\caption{}
	\end{subfigure}
	
	\caption{Two typical configurations of the spectral curve $\Sigma$: (a) $\Sigma$ intersects the real axis; (b) $\Sigma$ is disjoint from the real axis.}
	\label{fig:spectralcurves}
\end{figure}

The integrability of the NLS equation guarantees that the spectrum $\Sigma \cup \R$ is time independent.  {The symmetry of ZS spectral problem implies that, as a set, $\Sigma$ is invariant under complex conjugation; as a contour we choose the orientation such that  $\Sigma$  satisfies  $\overline{\Sigma}= - \Sigma$.}  In this manuscript we study the direct and inverse scattering in the case where $\Sigma$ does not intersect the real axis, (see Figure~\ref{fig:spectralcurves} (b)).   When $\Sigma \cap \R\neq \emptyset$ (see  Figure~\ref{fig:spectralcurves} (a)),   the  analysis and the formulation of the inverse problem as  a Riemann-Hilbert problem require verifying the local compatibility of the jump matrices at the intersection points and accounting for the corresponding local behavior of the scattering data near these points. This case is not treated in the present manuscript; 
it has been developed in our related work \cite{GJZZ0}.

   We denote by $W^\ell_e(x;z)$ and $W^r_e(x;z)$ the fundamental matrix solutions of the ZS  spectral problem \eqref{e:zs}, associated with  the  elliptic wave `background' solutions  $u_e^\ell(x)$ and $u_e^r(x)$ respectively.
Their corresponding spectra are $\Sigma^\ell\cup\R$ and $\Sigma^r\cup \R$. 
We introduce the notation (see Figure \ref{f:Sigma})
\begin{align}
	& \Sigma^{r,\pm} :=\Sigma^r\cap\C^\pm, 
 	&&\Sigma^{\ell,\pm}:=\Sigma^\ell\cap\C^\pm,
	& \Sigma^\pm := \Sigma^{\ell,\pm} \cup \Sigma^{r,\pm},
	&&\Sigma_{\mathrm{tot}} := \Sigma^\ell \cup \Sigma^r. 
\end{align}
Similarly we write 
\begin{gather}
\partial \Sigma^{s,+}=\{z^s_1,z^s_2\},  \qquad \partial \Sigma^{s,-}=\{\overline{z^s_1},\overline{z^s_2}\},\qquad  \quad s\in\{\ell, r\}.
\end{gather}
 We label the endpoints by taking $\Im(z_1^s) \leq \Im (z_2^s)$; if $\Im(z_1^s) = \Im(z_2^s)$, then we take $\Re(z_1^s) < \Re (z^s_2)$. The contour $\Sigma^{s,+}$ is oriented from $z_1^s$ to $z_2^s$. 
For illustration, Figure~\ref{fig:dn-partial} shows a numerical example of step-like elliptic $\dn$ backgrounds with partially overlapping spectral bands.

\begin{figure}[t]
	\centering
	
	\begin{subfigure}[b]{0.33\textwidth}
		\centering
		\begin{overpic}[
			width=\linewidth,
			%grid,
			%tics=10
			]{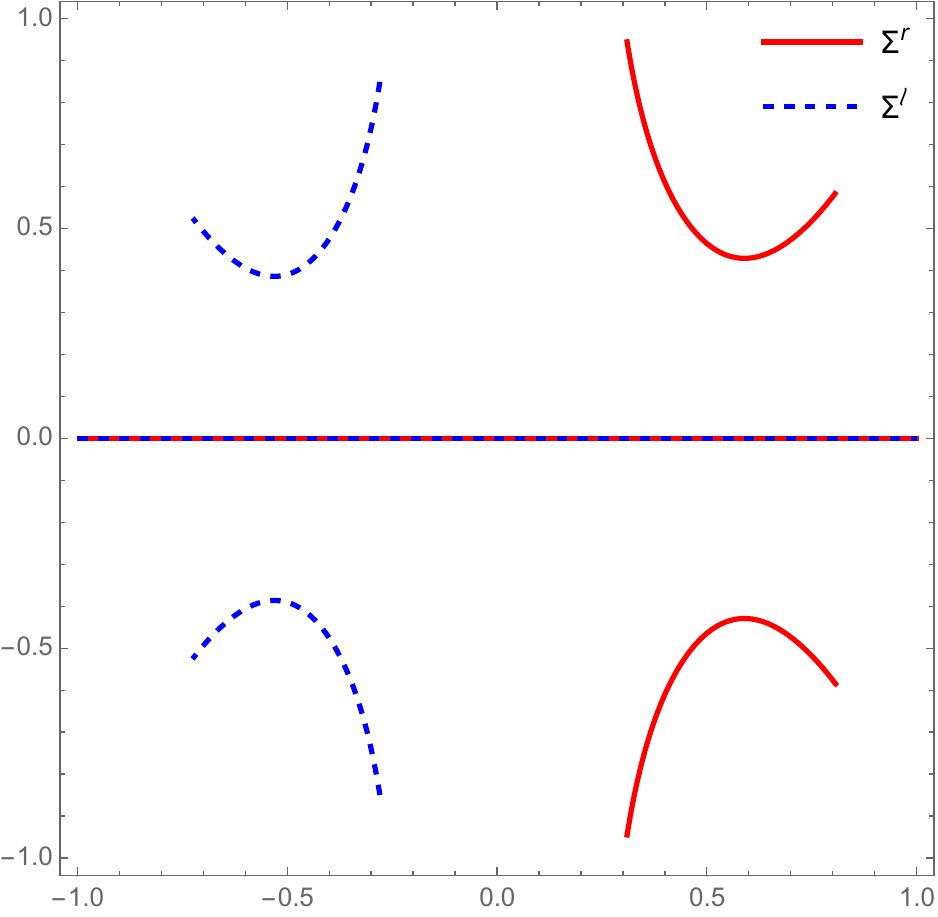}
			\put(13,74){\footnotesize $z_1^\ell$}
			\put(42,88){\footnotesize $z_2^\ell$}
			\put(60,91){\footnotesize $z_2^r$}
			\put(89,70){\footnotesize $z_1^r$}
			\put(19,73){\color{blue} \footnotesize $\bullet$}
			\put(39,88){\color{blue} \footnotesize $\bullet$}
			\put(65.5,92){\color{red} \footnotesize $\bullet$}
			\put(88,76){\color{red} \footnotesize $\bullet$}
			
			\put(13,28){\footnotesize $\bar{z}_1^\ell$}
			\put(43,11){\footnotesize $\bar{z}_2^\ell$}
			\put(60,9){\footnotesize $\bar{z}_2^r$}
			\put(89,27){\footnotesize $\bar{z}_1^r$}
			\put(18,25){\color{blue} \footnotesize $\bullet$}
			\put(39,11){\color{blue} \footnotesize $\bullet$}
			\put(65.5,6){\color{red} \footnotesize $\bullet$}
			\put(88,23){\color{red} \footnotesize $\bullet$}
		\end{overpic}
		\caption{}
	\end{subfigure}
	\begin{subfigure}[b]{0.33\textwidth}
		\centering
		\begin{overpic}[
			width=\linewidth,
			%grid,
			%tics=10
			]{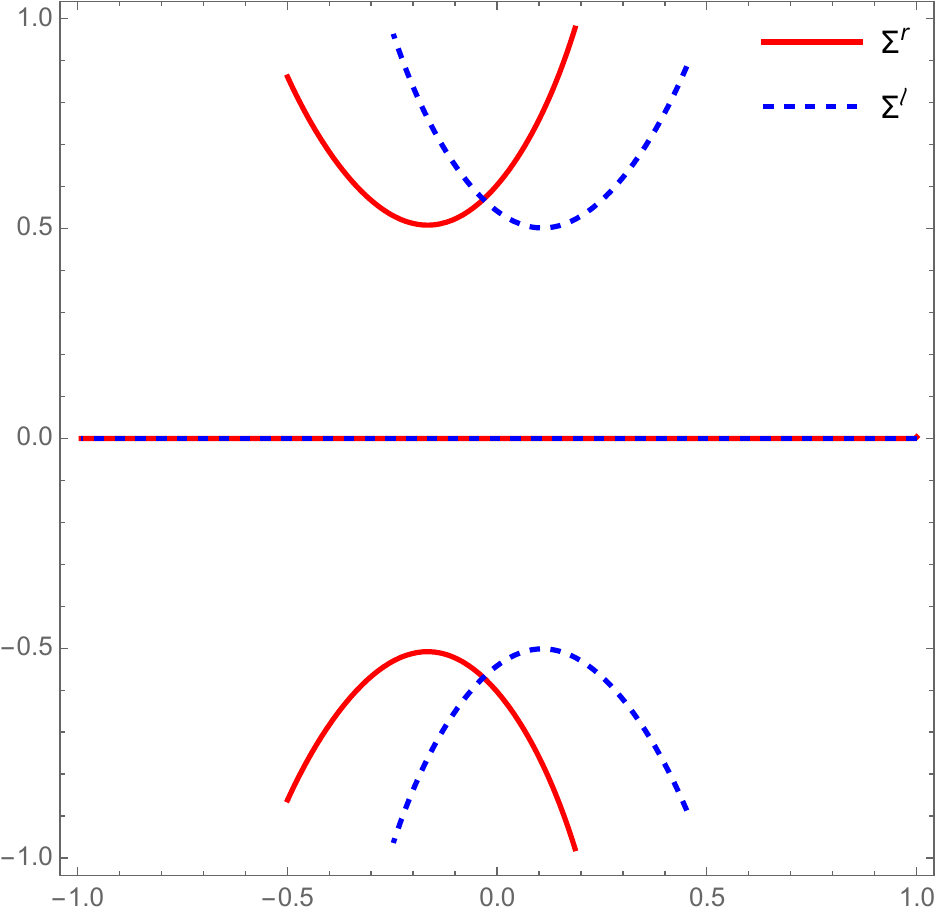}
			\put(34,90){\footnotesize $z_1^\ell$}
			\put(75,91){\footnotesize $z_2^\ell$}
			\put(63,92){\footnotesize $z_2^r$}
			\put(23,90){\footnotesize $z_1^r$}
			\put(40.5,93){\color{blue} \footnotesize $\bullet$}
			\put(72.5,90){\color{blue} \footnotesize $\bullet$}
			\put(29,89){\color{red} \footnotesize $\bullet$}
			\put(60,93){\color{red} \footnotesize $\bullet$}
			
			\put(35,8){\footnotesize $\bar{z}_1^\ell$}
			\put(75,10){\footnotesize $\bar{z}_2^\ell$}
			\put(24,11){\footnotesize $\bar{z}_1^r$}
			\put(63,8){\footnotesize $\bar{z}_2^r$}
			\put(42,8){\color{blue} \footnotesize $\bullet$}
			\put(72,10){\color{blue} \footnotesize $\bullet$}
			\put(29.5,11){\color{red} \footnotesize $\bullet$}
			\put(60.5,5){\color{red} \footnotesize $\bullet$}
		\end{overpic}
		\caption{}
	\end{subfigure}
	\begin{subfigure}[b]{0.32\textwidth}
		\centering
		\begin{overpic}
			[
			width=\linewidth,
			%grid,
			%tics=10
			]{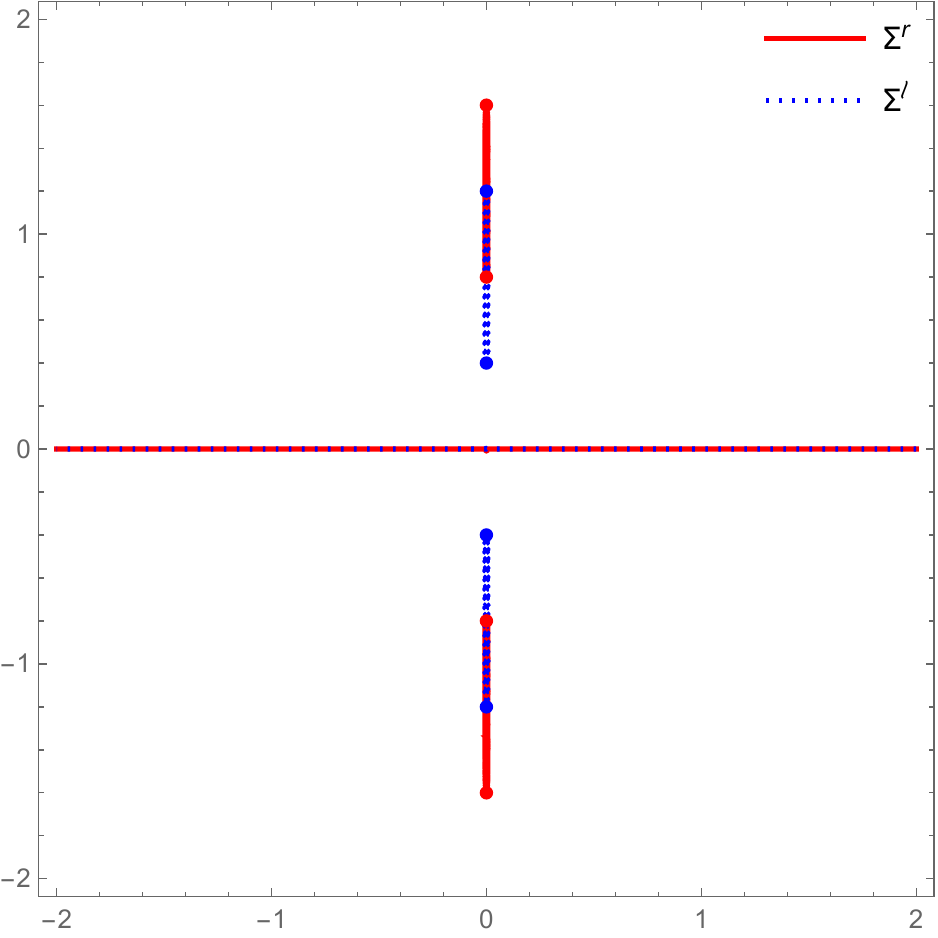}
			\put(44,59){\footnotesize $z_1^\ell$}
			\put(44,80){\footnotesize $z_2^\ell$}
			\put(54,90){\footnotesize $z_2^r$}
			\put(54,70){\footnotesize $z_1^r$}
			
			\put(44,43){\footnotesize $\bar{z}_1^\ell$}
			\put(44,25){\footnotesize $\bar{z}_2^\ell$}
			\put(54,34){\footnotesize $\bar{z}_1^r$}
			\put(54,18){\footnotesize $\bar{z}_2^r$}
		\end{overpic}
		\caption{}
		\label{f:vertical}
	\end{subfigure}
	
	\caption{Relative configurations of $\Sigma^\ell$ and $\Sigma^r$: (a) disjoint; (b) intersecting at a single point; (c) overlapping along a segment.}
	\label{f:Sigma}
\end{figure}

\color{black}
     We then define  the Jost functions  $W^s(x;z)$,  $s\in\{\ell,r\}$,  as the solutions  of   \eqref{e:zs} for the initial data $u_0$ in \eqref{initial_data}  such that 
     \begin{equation}
     \label{eq:Jost_intro} 
   W^s(x;z)\to W^s_e(x;z) , \quad  \mbox{as  $x\to\infty_s$ and $z\in \R\cup\Sigma^s$}, \;s\in\{\ell,r\}.
   \end{equation}
    Here we set $\infty_\ell=-\infty$  and $\infty_r=+\infty$. It follows that  $W^s(x;z)$ satisfies the integral   equation
 \[
	W^s(x;z) =W_e^s(x;z) \left[ I +  \int_{\infty_s}^x  W_e^s(y;z)^{-1} \Delta U^s(y) W^s(y;z) \d y \right],
\]
where  $\Delta U^s(x) = \begin{psmallmatrix} 0 & u_0(x) - u^s_e(x) \\ -\left( \overline{u_0(x)} - \overline{u^s_e(x)} \right) & 0 \end{psmallmatrix}$   . 

 It can be deduced from Proposition~\ref{propm} that a unique solution of the above integral equation exists for $z \in \R \cup \Sigma^s$. Moreover, letting $W^s_1(x;z)$ and $W^s_2(x;z)$ be the first and second columns of $W^s(x;z)$, respectively,  it follows that:
\begin{itemize}
\item $W^\ell_1(x;z)$ is analytic for $z\in \mathbb{C}^+ \backslash \Sigma^{\ell,+}$  and continuous up to the boundary and $W^r_2(x;z)$ is analytic   for $z\in \mathbb{C}^+ \backslash \Sigma^{r,+}$  and continuous up to the boundary;
\item $W^\ell_2(x;z)$ is analytic for $z\in \mathbb{C}^- \backslash \Sigma^{\ell,-}$  and continuous up to the boundary and $W^r_1(x;z)$ is analytic for $z\in \mathbb{C}^- \backslash \Sigma^{r,-}$  and continuous up to the boundary.
\end{itemize}

\begin{figure}[!htbp]
	\centering
%	\captionsetup{font=small}
%	\setlength{\abovecaptionskip}{2pt}
%	\setlength{\belowcaptionskip}{0pt}
%	
%	% ---------- first row ----------
	\begin{subfigure}[t]{0.3\textwidth}
		\centering
		\includegraphics[width=5cm,
		height=5cm]{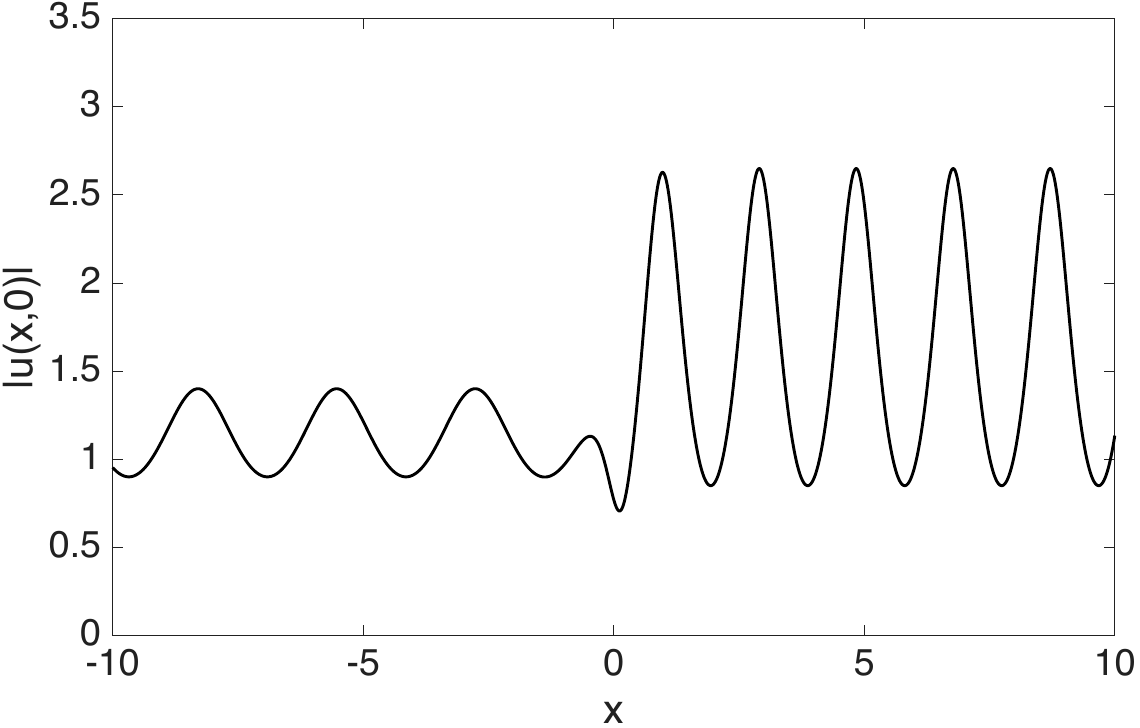}
		\caption{}
		\label{fig:initial-dn-partial}
	\end{subfigure}
	\hfill
	\begin{subfigure}[t]{0.32\textwidth}
		\centering
		\begin{overpic}[
			width=5cm,
			height=5cm
			]{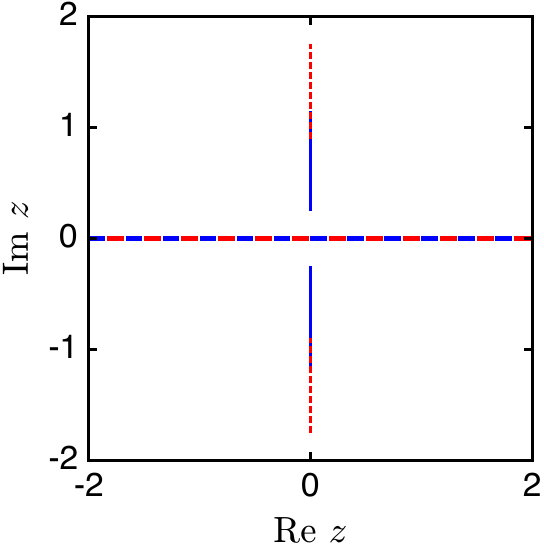}
			\put(59,60){\scriptsize $z_1^\ell$}
			\put(59,80){\scriptsize $z_2^\ell$}
			\put(59,72){\scriptsize $z_1^r$}
			\put(59,91){\scriptsize $z_2^r$}
			\put(56.2,60){\scriptsize $\bullet$}
			\put(56.2,79.25){\scriptsize $\bullet$}
			\put(56.2,73){\scriptsize $\bullet$}
			\put(56.2,91){\scriptsize $\bullet$}
			
			\put(59,50){\scriptsize $\ov z_1^\ell$}
			\put(59,30){\scriptsize $\ov z_2^\ell$}
			\put(59,39){\scriptsize $\ov z_1^r$}
			\put(59,20){\scriptsize $\ov z_2^r$}
			\put(56.2,50){\scriptsize $\bullet$}
			\put(56.2,32){\scriptsize $\bullet$}
			\put(56.2,37){\scriptsize $\bullet$}
			\put(56.2,20){\scriptsize $\bullet$}
		\end{overpic}
		\caption{}
		\label{fig:spectral-dn-partial}
	\end{subfigure}
	\hfill
%	
%	\par\vspace{0.2em}
%	
%	% ---------- second row ----------
	\begin{subfigure}[t]{0.32\textwidth}
		\centering
		\includegraphics[width=6cm,
		height=5cm]{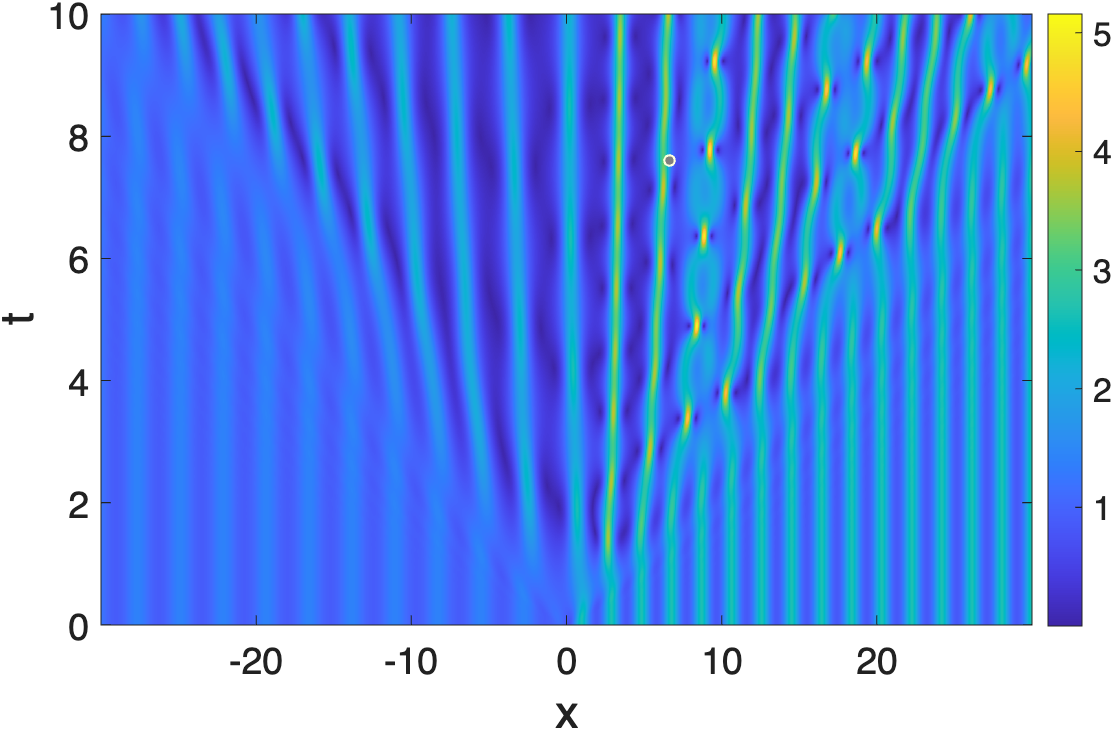}
		\caption{}
		\label{fig:density-dn-partial}
	\end{subfigure}
	
	\caption{Numerical computations for step-like initial data connecting two
		elliptic $\dn$ backgrounds with partially overlapping spectral bands. The
		upper endpoints are $z_1^\ell=0.25\i$ and $z_2^\ell=1.15\i$, and
		$z_1^r=0.9\i$ and $z_2^r=1.75\i$. Figure~(a) shows the initial data.
		Figure~(b) shows the corresponding spectral bands, which lie on the
		imaginary axis and stay away from the real axis. Figure~(c) shows the
		space--time density plot of $|u(x,t)|$.}
	\label{fig:dn-partial}
\end{figure}

For $z\in\mathbb{R}\cup(\Sigma^r\cap\Sigma^\ell)$, the matrix solutions $W^\ell(x;z)$ and $W^r(x;z)$  are both bounded and unimodular, hence  they are linearly related, i.e.,
\bse\label{sr}
\begin{gather}
W^\ell(x; z)=W^r(x; z) S(z),\quad z \in \mathbb{R}, \\
\label{e:bandscat}
W^\ell(x ; z_\pm)=W^r(x;  z_\pm) S(z_\pm),\quad z\in\Sigma^r\cap\Sigma^\ell,
\end{gather}
\ese
%\st{where the subscript $+$ (resp. $-$) denotes the limiting value from the positive (resp. negative) side of the oriented $\Sigma^r\cap\Sigma^\ell.$}
The scattering matrix $S(z)$ takes the form 
\begin{align}\label{e:S_intro}
S(z)=\begin{pmatrix}a(z) & -b^*(z) \\ b(z) & a^*(z)\end{pmatrix}, \quad z \in \mathbb{R},\qquad\qquad~~
S(z_{\pm})=\begin{pmatrix}a(z_{\pm}) & -b_2(z_{\pm}) \\ b_1(z_{\pm}) & a^*(z_{\pm})\end{pmatrix}, \quad z_{\pm} \in \Sigma^{r,+}\cap\Sigma^{\ell,+},
\end{align}
{where  with the notation $z_\pm$ we indicate the boundary values on the two sides of the contour $\Sigma^{r,+}\cap\Sigma^{\ell,+}$.}
\begin{remark}
The functions  $a(z)$,  $b_1(z)$ and $b_2(z)$  in  \eqref{e:S_intro} have to be understood  as  boundary values from the left and right of the contours $\Sigma^{\ell,\pm}$ and $\Sigma^{r,\pm}$.
When the formulas  hold  on both  boundaries we omit  to specify the boundary values   $b_j(z_\pm)$ ,  and we simply use  $b_j(z)$ , $j=1,2$ .\end{remark}

However, a matrix scattering relation like \eqref{e:bandscat} is not available on $\Sigma^r\setminus\Sigma^\ell$ because only $W^\ell_1(x;z)$ (resp. $W^\ell_2(x;z)$) admits
%\todo{It is not  only  the problem  analytic extension but also  boundedenes}
analytic extension    for  $z \in \C^+\backslash\Sigma^{\ell,+}$  (resp. $z \in \C^-\backslash\Sigma^{\ell,-}$ ) and it  is bounded  for  all $x\in\R$. These  columns satisfy: 
\bse
\label{W1ell}
\begin{align}
&W_1^\ell(x;z)=\phantom{-}a(z) W_1^r(x;z)+b_1(z)W_2^r(x;z)\,, &&z\in \Sigma^{r,+} \setminus\Sigma^{\ell,+},  \label{e:W1-}\\
&W^\ell_2(x;z)=-b_1^*(z) W^r_1(x;z)+a^*(z)W^r_2(x;z)\,, &&z\in \Sigma^{r,-}\setminus\Sigma^{\ell,-}\, .
\end{align}
\ese
Similarly, since only $W^r_2(x;z)$ (resp.\ $W^r_1(x;z)$)
has an analytic extension  for   $z\in \C^+\backslash\Sigma^{r,+}$  (resp. $z\in \C^-\backslash\Sigma^{r,-}$) and it  is bounded  for all $x\in\R$, we have
\bse
\label{W2ell}
\begin{align}
&W_2^r(x;z)=b_2(z)W_1^\ell(x;z)+a(z) W_2^\ell(x;z) \,, &&z\in \Sigma^{\ell,+} \setminus\Sigma^{r,+},\label{W2+}\\
&W^r_1(x;z)=a^*(z)W^\ell_1(x;z)-b_2^*(z)W_2^\ell(x;z)\,, &&z\in \Sigma^{\ell,-}\setminus\Sigma^{r,-}\,.
\end{align}
\ese
The l.h.s.\ of \eqref{W1ell} and \eqref{W2ell}  are, respectively,  analytic functions on  $ \C^\pm\backslash\Sigma^{\ell,\pm}$  and $ \C^\pm\backslash\Sigma^{r,\pm}$.  It follows the linear combinations on the r.h.s.\ of these equations are also analytic, so we omit writing boundary values even though individual terms on the r.h.s.\ may not be analytic. 
Given the analytic structure of the Jost functions, one can naively pose the inverse scattering transform as a Riemann-Hilbert problem normalized in the typical way, c.f.\ \eqref{WB}.  In the present setting, this normalization is not optimal. It results in a problem, RHP~\ref{RHP:M}  that  involves both boundary values of two reflection coefficients defined on each band, the reflection coefficient on $\R$, and the scattering data $a(z)$.
To arrive at a simpler inverse problem we introduce a natural  factorization of the scattering function
\begin{equation}
\label{factorization_a}
	a(z) = a_1(z) a_2(z)
\end{equation}
where $a_1(z)$ is analytic in $\C^+ \setminus \Sigma^{\ell,+}$ and 
	$a_2(z)$ is analytic in $\C^+ \setminus  \Sigma^{r,+}$ with the precise expression  given in  \eqref{e:defa12}. By renormalizing the problem in this  novel way, cf.\ \eqref{renormalization} and  \eqref{e:defa12}, the jump matrices   of the Riemann-Hilbert problem are expressed in terms of three reflection coefficients: a single reflection coefficient on $\R$ and two defined on $\Sigma$.  The resulting Riemann-Hilbert problem,  takes a simpler, more symmetric form  in comparison to  the naive approach.

Our first result concerns the direct scattering problem.
We denote the spectral endpoints by $\partial\Sigma$.
Figure~\ref{f:Sigma} illustrates the possible relative
configurations of $\Sigma^\ell$ and $\Sigma^r$.
We focus on nontrivial spectral overlap away from $\R$.
For distinct endpoint sets, Proposition~\ref{prop:verticaloverlap}
shows that such an overlap forces both backgrounds to be
Galilean transforms of $\dn$ waves with the same velocity.
A common Galilean transformation therefore places their
spectral bands on the imaginary axis, as in
Figure~\ref{f:Sigma}(c).

    The inverse scattering problem consists in reconstructing $u(x,t)$ from the associated scattering data corresponding to   step-like oscillatory initial data $u_0$  of the form \eqref{initial_data}.
  \begin{assumption} \label{assume}
For simplicity, we assume that $a(z)\neq 0$  for $z\in \overline{\C^+}$. This is equivalent to saying that the initial datum $u_0$    contains no solitonic or breather component.
\end{assumption}
This assumption has been introduced to streamline the presentation since no technical difficulties will occur by introducing a discrete spectrum in $\C^+\backslash\{\Sigma^{r,+}\cup\Sigma^{\ell,+}\}$.
     Further  from the factorization of $a(z) = a_1(z) a_2(z)$,  we introduce three reflection coefficients: 
     \bse\label{e:r12}
\begin{gather}
	r_1(z)= \frac{ a_2(z_-)}{a_1(z_-)} \frac{\e^{-2\i x_0^\ell E^\ell-2\i x_0^r z}}{a_1(z_+)a_2(z_-)-\i b_2(z_-)\e^{-2\i x_0^\ell E^\ell}}\,, \qquad z\in\Sigma^{\ell,+},\label{e:defr1}\\
{	r_2(z)= \frac{a_1(z_+)}{a_2(z_+)}\frac{\e^{2\i x_0^rE^r+2\i x_0^r z}}{a_1(z_+)a_2(z_-)+\i b_1(z_+)\e^{2\i x_0^r E^r}}}, \qquad z\in\Sigma^{r,+}, \label{e:defr2} \\
\rho(z)=\frac{b(z)}{a_1(z)a_2^*(z)}\e^{-2\i x_0^r z}, \qquad z\in \R, \label{e:rho}
\end{gather}
\ese
 where $*$  denotes the Schwarz reflection defined by   $f^*(z)=\overline{f(\overline{z})}$, the quantities $E^r$ and $E^\ell$ are defined in \eqref{e:defEN} and the constants $x_0^{r}$ and $x_0^{\ell}$ are the shifts in the elliptic solution 
 \eqref{e:qpdef}.
     The next theorem characterizes the analytical properties of the scattering data.
         { \begin{theorem}\label{theo_DS}
     If the initial data  $u_0-u_e^\ell\in L^{1,2}(\R^-)$ and $u_0-u_e^r\in L^{1,2}(\R^+)$ then
     \begin{itemize}
     \item $r_1$ is continuous on  $\Sigma^\ell$ with square-root vanishing at the end points of   $\Sigma^\ell$.
     \item $r_2$ is continuous on  $\Sigma^r$ with square-root vanishing at the end points of   $\Sigma^r$.
     \item $\rho$ is continuous on $\R$. Moreover if   also  $u_0-u_e^\ell\in \mathcal{W}^{k,1}(\R^-)$ and $u_0-u_e^r\in \mathcal{W}^{k,1}(\R^+)$, $k\geq 1$,   then $\rho\in L^{2,n}(\R)$  with $n<k-\frac{1}{2}$ and $ \rho\in L^{1,m}(\R)$  with $m<k-1$.  Finally if $u_0$ is the exact step  as in \eqref{e:stepic}, then $\rho$ is analytic on $\R$.
     \end{itemize}
     \end{theorem}      
     The proof of this theorem is provided at the end of  Section~\ref{Symm_RHP}. }

   The factorization   of $a$   introduced in \eqref{factorization_a} permits to derive a Riemann-Hilbert problem that  has a trivial   jump on $\Sigma^{\ell,+}$ when $r_1=0$ and on $\Sigma^{r,+}$ when $r_2(z)=0$.
Indeed     with these coefficients we define a  matrix $V(z;x,t)$, $z\in\R\cup\Sigma^r\cup\Sigma^\ell$\,,
       \be\label{e:Vtilde}
		V(z;x,t)=\begin{dcases}
			\begin{pmatrix}
				\frac{1-r_1(z)r_2(z)}{1+r_1(z)r_2(z)} & \frac{2 \mathrm{i} r_2(z)}{1+r_1(z)r_2(z)}\e^{-2\theta(z;x,t)} 
				\\
			\frac{2\mathrm{i} r_1(z)}{1+r_1(z)r_2(z)}\e^{2\theta(z;x,t)} & 	\frac{1-r_1(z)r_2(z)}{1+r_1(z)r_2(z)} 
			\end{pmatrix}, & z\in\Sigma^{r,+}\cap\Sigma^{\ell,+},
			\\
			\begin{pmatrix}
				1 & 0 \vspace{0.5ex}\\
				2\mathrm{i} r_1(z)\e^{2\theta(z;x,t)} & 1
			\end{pmatrix}, & z\in\Sigma^{\ell,+}\setminus \Sigma^{r,+},
			\\
			\begin{pmatrix}
				1 &  2\mathrm{i} r_2(z)\e^{-2\theta(z;x,t)} \vspace{0.5ex}\\
				0 & 1
			\end{pmatrix}, & z\in\Sigma^{r,+}\setminus\Sigma^{\ell,+},
			\\
			\begin{pmatrix}
				1+|\rho|^2 & \rho^*(z)\e^{-2\theta(z;x,t)} \vspace{0.5ex}\\
				\rho(z)\e^{2\theta(z;x,t)} & 1
			\end{pmatrix}, & z\in \Real,  \\
			\sigma_2 \overline{V(\, \overline z;x,t)} \sigma_2 & z \in \Sigma^{-}
			\end{dcases}
		\ee
where 
\begin{equation}
	\theta(z; x, t)=\i x z + \i t z^2.
\end{equation}	
The inverse problem is formulated as a Riemann--Hilbert problem.
\begin{RHP}\label{RHP:Phi}
	Find a $2\times2$ matrix-valued function $\Phi(z;x,t)$ such that:
	\begin{enumerate}
		\item $\Phi(z;x,t)$ is analytic in $\Complex\setminus (\Real\cup\Sigma^r\cup\Sigma^\ell)$\,.
		\item $\Phi(z;x,t)=I+\mathcal{O}(z^{-1})$, as $z\to\infty$\,.
		\item $\overline{\Phi(\overline{z};x,t)}=\sigma_{2} \Phi(z;x,t)\sigma_{2}$\,.
		\item $\Phi(z;x,t)$ satisfies the jump condition $\Phi(z_+;x,t)=\Phi(z_-;x,t)V(z)$\,,  for  $ z\in\R\cup\Sigma^\ell\cup\Sigma^r$, where $V$ is defined in \eqref{e:Vtilde}.
		\item  { The matrix $\Phi(z;x,t)$ has at most quartic root singularities at the  boundaries of   $\Sigma^\ell$   and   $\Sigma^r$.}
		      
	\end{enumerate}
\end{RHP}
\begin{remark}
 The  jump matrix  $V(z;x,t)$  coincides with the  jump matrix of the full  soliton gas Riemann--Hilbert problem for the NLS equation (preprint  \cite{GJZZ0}    that follows from the work in preparation for  the  modified Korteweg-de Vries equation 
\cite{GJM}).  When $r_1=0$ or $r_2=0$ the Riemann--Hilbert problem  \ref{RHP:Phi}  corresponds to initial data asymptotic to zero as $x\to \pm \infty$  respectively while
as $x\to \mp \infty$ the initial data is asymptotic to an elliptic wave.  When $\rho=0$ and   $r_2=0$, such a Riemann--Hilbert problem   has appeared in \cite{BGO}  for the NLS equation,  in \cite{GGJM} for the Korteweg-de Vries equation
and in \cite{GGJMM} for the modified KdV equation.

\color{black}
When $\rho=0$ and $r_1\neq 0, \;r_2\neq 0$, the Riemann--Hilbert problem \ref{RHP:Phi}   can be obtained from a soliton gas limit as in \cite{GJM,GJZZ},   where it  is referred to as  the "full soliton gas". When also  $\rho\neq 0$, the above Riemann--Hilbert problem can be considered as a full soliton gas Riemann--Hilbert problem on a background.  
 We observe that  the reflection coefficients $r_1(z)$ and $r_2(z)$ for the full gas problem can be taken from a larger functional space than those obtained in the direct scattering problem (see Theorem~\ref{theo_DS}). In particular the vanishing condition of $r_1(z)$ and $r_2(z)$ at the end points of $\Sigma^\ell$ and $\Sigma^r$ respectively, is not needed in the soliton gas case.

\end{remark}

The next theorem solves the  inverse problem.
\begin{restatable}{theorem}{InverseProblem}
%\begin{theorem}
\label{theorem2}
Given functions $\rho\in L^{2} (\Real)$, $r_1\in L^2(\Sigma^\ell)$, and $r_2\in L^2(\Sigma^r)$, the Riemann--Hilbert problem \ref{RHP:Phi} is uniquely solvable for all $(x,t)\in\Real^2$. The solution of focusing NLS equation  $u(x,t)$ is obtained from the solution  $\Phi(z;x,t)$ of the Riemann--Hilbert problem \ref{RHP:Phi} by
	\be\label{e:recons}
u(x,t)=2\mathrm{i}\lim_{z\to\infty} z \Phi_{12}(z;x,t).
	\ee
Moreover,  if $\rho\in L^{2,2} (\Real)\cap L^{1,2}(\Real)$,   the function $u(x,t)$ defined in \eqref{e:recons} is a   classical solution in $C^2(\R)\times C^1(\R)$.
%\end{theorem}
\end{restatable}

Combining Theorem~\ref{theo_DS} and Theorem~\ref{theorem2} we arrive at the following result.
\begin{corollary}
If the initial data \eqref{initial_data} satisfies the condition $u_0-u_e^\ell\in \mathcal{W}^{4,1}(\R^-)$ and $u_0-u_e^r\in \mathcal{W}^{4,1}(\R^+)$  
 then $u(x,t)$ is a classical solution of the NLS equation  for $t>0$.
\end{corollary}
For the  exact step initial data  $u_0$   as in \eqref{step_intro},   we get explicitly in Section~\ref{perturbative} the scattering data.
The analyticity of the scattering data enables us to get the following.
\begin{lemma}
If the initial data $u_0$ is the exact step as in  \eqref{step_intro}, then $u(x,t)$ is a classical solution for $t>0$
\end{lemma}
We omit the proof of the above lemma, that uses contour deformation and follows the arguments of \cite{steplikemkdv}.
%\todo{change $u\to u_0$ and add a statement in the corollary}
This manuscript is organized as follows: in section~\ref{elliptic} we derive the elliptic travelling wave $u_e(x,t)$  of the NLS equations and the associated solution $W_e(x,t;z)$ of the Lax pair \eqref{e:NLSLP}. We then formulate a Riemann--Hilbert problem for $W_e(x,t;z)$.
In section~\ref{perturbative} we derive the Jost solution for the step-like oscillatory initial data  \eqref{initial_data} and we   prove Theorem~\ref{theo_DS} and  Theorem~\ref{theorem1}.  Furthermore, an explicit example illustrating step-oscillatory initial data of the form \eqref{initial_data} is provided at the end of
Section~\ref{perturbative}. 
In section~\ref{inverse},  we formulate  the inverse problem,  prove its solvability and  prove Theorem~\ref{theorem2}.
Detailed technical aspects of the proofs are deferred to the Appendix. 

      \section{Elliptic travelling waves\label{elliptic}} 
      In this section we review the theory of elliptic travelling waves for the NLS equation and the corresponding spectral curve.
      We need to solve equation \eqref{e:snls}, namely
\[
 \frac{1}{2} \psi_{xx}(x) + \omega_0 \psi(x) +  |\psi|^2 \psi(x) \,=\, 0~, \quad
  x \in \R~.
\]
The general solution can be written as  $\psi(x)=\phi(x) \e^{\i\theta(x)},$ where $\phi$ and $\theta $ are obtained from the  integrals of the differential identities
\[
\dfrac{\phi \d\phi}{\sqrt{-(\phi^6+2\omega_0 \phi^4+\delta_1\phi^2+\delta^2)}}=\d x,\quad\theta'= \dfrac{\delta}{\phi^2},
\]
where $\delta$ and $\delta_1$ are integration constants.
The solution  is   parametrized by  the  three real zeros $e_1<e_2<e_3$ of the polynomial  $f^3+2\omega_0 f^2+\delta_1f+\delta^2=0$:
\begin{subequations}
\label{E:defs}
\begin{eqnarray}
\phi^2(x)&=&e_3-(e_3-e_2) \,\mbox{sn}^2(\sqrt{e_3-e_1}x;k),\quad k =\frac{e_3-e_2}{e_3-e_1},
\label{E:phi_def} \\
\theta(x)&=& \displaystyle{\delta\int_0^x\phi^{-2}(\xi)\d\xi}+\theta_0 \label{E:theta_def} \\
\delta^2&=&-e_1e_2e_3,\quad \omega_0=-\frac{1}{2}(e_1+e_2+e_3),\;\; \delta_1=e_1e_2+e_1e_3+e_2e_3.\quad  \label{E:c_def}
\end{eqnarray}
\end{subequations}
Here  $\theta_0$ is a free constant,  $k \in [0,1]$ is the elliptic parameter of the 
Jacobi elliptic sine function, $\sn(x;k)$ (see e.g. \cite{Lawden}).     
The function $\sn^2(x;k)$ is periodic if $k\in[0,1)$, with 
period given by $L = 2\mathrm{K}$, where $\mathrm{K}=\mathrm{K}(k)$ is  the  complete elliptic integral of the first kind defined by 
\begin{equation}
\mathrm{K}(k)=\int_{0}^{\pi/2}\left(1-k\sin^2 \xi \right)^{-1/2} \,\d\xi.
\end{equation}
When $k=0$, $\sn^2(x;0) = \sin^2(x)$ with $L =  \pi$.  As 
$k$ approaches 1, $\sn(x,k)$ approaches $\tanh(x)$ and $L$ approaches 
infinity~\cite{Lawden}. Although $\phi(x)$ inherits the periodicity of $\sn(x,k)$, the solution 
$\psi(x)$ is
typically {\em not} $L$-periodic in the  space variable  $x$  because the periods 
of $\e^{\i\theta}$ and $\phi$ are typically incommensurate. 
 Note that $\delta^2\geq 0$, which implies that $e_1\leq 0<e_2<e_3$ or $e_1<0\leq e_2<e_3$.  When $e_1=0$ or $e_2=0$, we have  $\theta_0=0$,  and in this case the travelling wave is $x$-periodic:
\begin{equation}
\label{eq:dn_cn}
\begin{split}
&e_1=0,\quad  \psi(x) \,=\,\phi(x) = \sqrt{e_3}\mbox{dn}(\sqrt{e_3}x;k)\,,\\
&e_2=0,\quad  \psi(x) \,=\,\phi(x) = \sqrt{e_3}\mbox{cn}(\sqrt{e_3-e_1}x;k),
\end{split}
\end{equation}
where $\mbox{dn}(x;k)$ and $\mbox{cn}(x;k)$ are the Jacobi elliptic functions.
When we consider a solution of \eqref{e:qpdef}  with velocity $v=0$, namely 
$$ u_e(x, t ; v=0)= \e^{-\i\omega_0 t}\psi(x),$$  then the fundamental matrix solution   of the ZS  spectral problem is of the form 
\begin{equation}
\label{def_Phi}
\hat{\chi}(x,t;z)=\e^{-\frac{\i}{2}\omega_0t\sigma _3}\tilde\chi(x;z)\e^{-\i R t \sigma_3},
\end{equation}
for a $2\times 2$ matrix $\tilde\chi(x;z)$ with time-independent entries
so that one obtains the equation 
\[
\left( \mathcal{B}(\psi;z)+\frac{\i}{2}\omega_0\sigma_3 \right)\tilde\chi=-\i R\tilde\chi\sigma_3\,,
\]
where $\mathcal{B}(\psi,z)$ is the ZS linear operator with respect to the potential $\psi(x)$ that satisfies equation \eqref{e:snls}.
 The above equation 
has a non-trivial solution only when 
\begin{equation}\label{R}
	\det \left( \mathcal{B}+\frac{\i}{2}\omega_0\sigma_3\right)=:R^2=z^4-z^2\omega_0-\delta z+\frac{\omega_0^2-\delta_1}{4},
\end{equation}
 where  the constants of  integration  $\delta$ and $\delta_1$  defined in \eqref{E:c_def}, are given by
\[
\delta=\frac{\i}{2}(\psi\overline{\psi}_x-\psi_x\overline{\psi}),\quad \delta_1=-|\psi|^4-2\omega_0|\psi|^2-|\psi_x|^2.
\]
The branch points of the spectral curve are the zeros of the quartic polynomial on the right-hand side of \eqref{R}.
The    equation  \eqref{R} gives the spectral curve of the elliptic solution and it is  a Riemann surface of genus one, which  we write in the form 
\be
\label{RS0}
\mathcal{S}=\left\{(z,R)\in\mathbb{C}^2: R^2=(z-z_1)(z-\overline{z_1})(z-z_2)(z-\overline{z_2})\right\}\,,
\ee
where $z_1$ and $z_2$ are related to $e_1,e_2,$ and $e_3$ defined in \eqref{E:c_def}  by the relation
 \[
 z_2= \frac{1}{2}\sqrt{-e_1}+ \frac{\i}{2}\left(\sqrt{e_3}+\sqrt{e_2}\right),\quad   z_1=-\frac{1}{2}\sqrt{-e_1}+ \frac{\i}{2}\left(\sqrt{e_3}-\sqrt{e_2}\right),
 \]
 if $\delta<0$,  while 
  \[
 z_2= \frac{1}{2}\sqrt{-e_1}+ \frac{\i}{2}\left(\sqrt{e_3}-\sqrt{e_2}\right),\quad   z_1=-\frac{1}{2}\sqrt{-e_1}+ \frac{\i}{2}\left(\sqrt{e_3}+\sqrt{e_2}\right)\,,
 \]
 if $\delta>0$.
  Below we assume, for simplicity, that we are in the former case, namely $\Re(z_2)>\Re(z_1)$ and $\Im(z_2)>\Im(z_1).$

From the   stationary elliptic solution $u_e(x,t; v=0)= \e^{-\i \omega_0 t}\psi(x),$ with zero velocity, the solution with velocity $v$ is obtained
using Galilean invariance,   namely 
\begin{equation}\label{travelling}
  u_e(x,t; v) \,=\, \e^{-\i \omega_0 t}\,\e^{\i  \big(vx - \frac{v^2}2t\big)}\psi(x-vt),
  \quad (x,t) \in \R^2.
\end{equation}
               The corresponding eigenfunction takes  the form
        \[
        W_e( x,t;z)=\e^{\frac{\i}{2}(vx-\frac{v^2}{2}t)\sigma_3}\hat{\chi}(x-vt,t; z+\frac{v}{2})\,,
               \]
        where $\hat{\chi}(x,t;z)$ has been defined in \eqref{def_Phi}.
The above equation shows that the  fundamental matrix solution of the ZS spectral problem  for  the elliptic travelling wave  with velocity $v$ is obtained from the zero-velocity fundamental matrix solution by the spectral shift $z\to z+\frac{v}{2}$.
Since we develop inverse scattering for initial data that are asymptotic,  as $x\to\pm \infty$, to two different elliptic quasi-periodic waves, Galilean invariance allows us to set only one of the two velocities to zero. For generality, we therefore consider initial data asymptotic to two elliptic periodic travelling waves whose velocities are not assumed to vanish.

In the next sub-section we describe the general elliptic wave  $u_e(x,t)$ with non-zero velocity  and the associated fundamental matrix solution  $W_e(x,t;z)$ of the Lax pair   obtained by solving a Riemann--Hilbert problem. 
\subsection{Genus one  Riemann--Hilbert problem}
\begin{figure}[t]
	\centering
\begin{overpic}[
	width=0.35\linewidth,
	%grid,
	%tics=10
	]{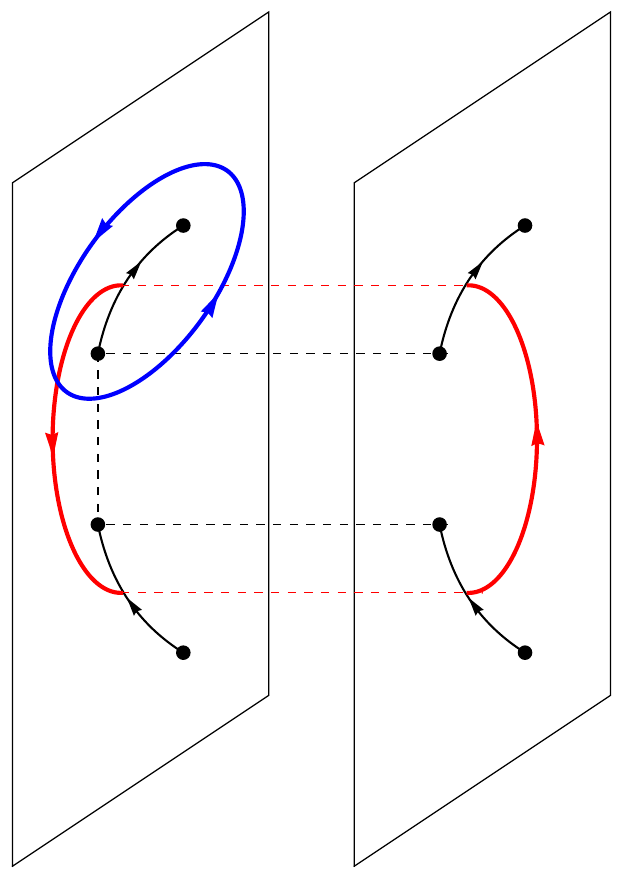}
		
		%---------- branch points ----------
		\put(23,25){$\bar z_2$}
		\put(13,39){$\bar z_1$}
		
		\put(21,72){$z_2$}
		\put(13,58){$z_1$}
		
		\put(60,72){$z_2$}
		\put(52,58){$z_1$}
		
		\put(61,25){$\bar z_2$}
		\put(52,39){$\bar z_1$}
		
		%---------- infinities ----------
		\put(24,92){\footnotesize $\infty^+$}
		\put(63,92){\footnotesize $\infty^-$}
		
		%---------- cycles ----------
		\put(7,76){$\mathcal B$}
		\put(1,50){$\mathcal A$}
		
		%---------- cuts ----------
		\put(16,65){\footnotesize $\Sigma^+$}
		\put(16,33){\footnotesize $\Sigma^-$}
		\put(13,48){\footnotesize $\Sigma_0$}
		
	\end{overpic}
	
	\caption{The homology basis for the Riemann surface
		$\mathcal{S}$ associated with
		$R^2=(z-z_1)(z-\overline{z_1})
		(z-z_2)(z-\overline{z_2})$.}
	\label{surf1}
\end{figure}

%\todo{plot $\Sigma_0$}

We consider the two-sheeted, genus-one Riemann surface
\be
\label{RS_X}
\mathcal{S}=\left\{(z,R)\in\mathbb{C}^2: R^2=(z-z_1)(z-\overline{z_1})(z-z_2)(z-\overline{z_2})\right\}\,.
\ee
where  $$z_j=\xi_j+\i\eta_j,\;\; j=1,2,\qquad 0<\eta_1\leq\eta_2.$$ 
We define
\[
v=-\Re(z_1+z_2)=  -(\xi_1+\xi_2)  \,,
\]
because this quantity is the velocity of the travelling wave.
 The first sheet of $\mathcal{S}$ is identified by the fact that $R(z)\to z^2$ as $z\to\infty$. Denote by $\infty^+$ ($\infty^-$) the pre-image of $z = \infty$ on the first (second) sheet of $\mathcal{S}$.

The function $\sqrt{(z-z_1)(z-\overline{z_1})(z-z_2)(z-\overline{z_2})}$  is a multi-valued function on $\C$   with oriented branch cuts $\Sigma^+$ and $\Sigma^{-}$ as in Figure \ref{surf1}.
Fix a canonical homology basis on $\mathcal{S}$ by choosing $\mathcal{B}$ to encircle $\Sigma^+$ anti-clockwise  on the first sheet, and $\mathcal{A}$ to pass from the positive side of $\Sigma^+$ to $\Sigma^{-}$ on sheet 1 and from the negative side of $\Sigma^{-}$ to $\Sigma^+$ on sheet 2, see Figure \ref{surf1}.
Notice that the surface $\mathcal{S}$ has an anti-holomorphic involution $\sigma: \mathcal{S}\to \mathcal{S}$ given by
\[
\sigma(z,R)=(\overline{z},\overline{R}),
\]
and the canonical homology basis satisfies
\begin{equation}
	\label{SymAB}
	\sigma(\mathcal{A})=-\mathcal{A},\quad \sigma(\mathcal{B})=\mathcal{B}.
\end{equation}
 Denote by $\d p$ and $\d q$ the quasi-momentum differentials and quasi-energy differentials respectively. These are defined to be the unique meromorphic differentials of the second kind characterized by the following asymptotic and normalization conditions:
\bse
\begin{gather}
 \d p=\pm[1+\mathcal{O}(z^{-2})]\d z,\quad \mbox{as $z\to\infty^\pm$}, \;\quad  \mbox{and \;\;\;$ \oint\limits_{\mathcal{A}}\d p=0$} \,,\\
 \d q=\pm[2z+\mathcal{O}(z^{-2})]\d z,\quad \mbox{as $z\to\infty^\pm$}, \;\quad \mbox{and \;\;\;$\oint\limits_{\mathcal{A}}\d q=0$} \,.
\end{gather}
\ese
Moreover, $\d p$ and $\d q$ admit the explicit representations
\bse\label{p}
\begin{gather}
\d p=\frac{z^2-\Re(z_1+z_2)z+c_0}{R}\d z,\\\quad
%=[1+\mathcal{O}(z^{-2})]\d z,\\
\label{p2}
\d q=2\frac{z^3-\Re(z_1+z_2)z^2+(\Re(z_1)\Re(z_2)+\frac{\
\Im(z_1)^2+\Im(z_2)^2 }{2})z+c_1}{R}\d z\,,
%=[2z+\mathcal{O}(z^{-2})]\d z,
\end{gather}
\ese
with constants $c_0$ and $c_1$ determined by the $\mathcal{A}$-normalization
\[
\oint_{\mathcal{A}}\d p=0,\qquad \oint_{\mathcal{A}}\d q=0.\quad
\]
In particular the first integral gives
\begin{align}
c_0=-\dfrac{ \oint_{\mathcal{A}}\frac{z^2-(\xi_1+\xi_2)z}{R}\d z}{\oint_{\mathcal{A}}\frac{\d z}{R}}=\dfrac{1}{2}(|z_1|^2+|z_2|^2)-\frac{1}{2}|z_1-\overline{z_2}|^2\dfrac{\mathrm{E}(m)}{\mathrm{K}(m)}\,,\label{c0}
\end{align}
 where $\mathrm{K}(m)$ and $\mathrm{E}(m)$ denote the complete elliptic integrals of the first and second kinds, respectively, defined by
\be
\mathrm{K}(m)=\int_0^1\frac{\d t}{\sqrt{(1-t^2)(1-mt^2)}},\quad
\mathrm{E}(m)=\int_0^1\sqrt{\frac{1-mt^2}{1-t^2}}\,\d t, \quad m=1-\left|\frac{z_1-z_2}{z_1-\overline{z_2}}\right|^2\,.\label{m}
\ee

Using the Riemann bilinear relations one obtains  the normalizing constant $c_1$  in \eqref{p2}  as
\begin{equation}
\label{c1}
 c_1=\dfrac{c_0}{2}\Re(z_1+z_2)-\frac{1}{2}(\Re(z_1)|z_2|^2+\Re(z_2)|z_1|^2), 
\end{equation}
so that one can rewrite $\d q$ in the form 
\begin{equation}
\label{dqR}
\d q=-v \d p+\d R(z).
\end{equation}
The quasi-momentum differential $\d p$ and quasi-energy  differential $\d q$  have $\mathcal{B}$-periods that are computable from the Riemann bilinear identities

\be
\label{Omega12}
\Omega_1:=\oint_{\mathcal B}\d p=4\pi\i c=-4\pi|c|,     \qquad \Omega_2:=\oint_{\mathcal B}\d q =-v\Omega_1=4\pi v|c|.
\ee
where 
\be\label{e:defc}
c=\left(\oint_{\mathcal A}\frac{\d z}{R(z)}\right)^{-1}
=\frac{\i|z_1-\overline{z_2}|}{4\mathrm K(m)},
\qquad
m=1-\left|\frac{z_1-z_2}{z_1-\overline{z_2}}\right|^2.
\ee

Note that  $\Omega_1$ and $\Omega_2$ are real.

Next we fix a base point $z_2$ on the first sheet of the Riemann surface and define the Abelian integrals   
$p(z)$ and $q(z)$   by
\be\label{e:defAbelian}
p(z)=\int_{z_2}^z \d p,\quad
q(z)=\int_{z_2}^z \d q.
\ee
The expansion for $z\to\infty^+$ gives \cite{Belokolos1994} 
\be\label{e:pqasymp}
p(z)=z+E+ \frac{\pi_1}{z}+ o(z^{-1}),\quad q(z)=z^2 +N +o(1),\quad z\to\infty^+,
\ee
where 
\begin{equation}\label{e:defEN}
\begin{split}
&E:=\lim_{z\to\infty}(p(z)-z)=\lim_{z\to\infty}\int_{z_2}^z[\d p(\lambda)-\d\lambda]-z_2
=\ov z_{2}-\Re(z_{1}+z_{2})+(\ov z_{1}-\ov z_{2})\,\frac{\Pi(n,m)}{\mathrm{K}(m)}\,, \\
&N:=\lim_{z\to\infty}q(z)-z^2=-vE+\lim_{z\to\infty}(R(z)-z^2-vz)=-vE+\dfrac{v^2}{4}-\dfrac{\omega_0}{2},\\
&\pi_1=\Re(z_1)\Re(z_2)
+\frac{(\Im z_1)^2+(\Im z_2)^2}{2}-c_0
=\frac12\left[
|z_1-\overline{z_2}|^2
\frac{\mathrm E(m)}{\mathrm K(m)}
-(\Re(z_1-z_2))^2
\right],\\
& \omega_0=\frac{3}{2}v^2-(z_1(z_2+\overline{z_1}+\overline{z_2})+z_2(\overline{z_1}+\overline{z_2})+\overline{z_1}\overline{z_2}),
\end{split}
\end{equation}
where $\Pi(n,m)$ is the complete elliptic integral of the third kind, defined as
\[
\Pi(n,m)=\int_0^1\frac{\mathrm dt}{(1-n t^2)\sqrt{(1-t^2)(1-m t^2)}}, \qquad n=\frac{z_1-\ov {z_1}}{z_1-\ov {z_2}},
\]  
with $m$ defined in \eqref{m}. 
 We note that, despite the similar notation, the constant $E$ in \eqref{e:defEN} is distinct from the complete elliptic integral $\mathrm{E}(m)$ in \eqref{m}. The fact that $E$ is real follows from the  symmetry  of the curve and the homology basis. The calculations of the constants defined above, including $c$ in \eqref{e:defc}, $c_0$ in \eqref{c0}, and $E$ and $N$ in \eqref{e:defEN}, are given in Appendix~\ref{a:calculationcE}.

Now we are ready to give a more precise definition of the spectrum $\Sigma$ and $\Sigma^+$ (and $\Sigma^{-}$). Indeed we will show in Proposition~\ref{prop:W0}  that the fundamental solution of the ZS linear spectral problem \eqref{e:zs} for the  elliptic  initial data $u_e$ as in \eqref{initial_data} is given by the matrix 
\[
W_e(x,t;z)=O(z;x,t)\e^{-\i\left((x-x_0)(p(z)-E)+t(q(z)-N)\right)\sigma_3}\,,
\]
where $O(z;x,t)$ is a bounded matrix for $z\in\C\backslash\{ z_1,z_2,\ov{z_1},\ov{z_2}\}$ and for all $x\in\R$.
It follows that $W_e(x,t;z)$ is bounded for all $x\in\R$,  whenever   $\Im(p(z))=0$, namely 
\begin{equation}
    \label{spectrum1}
    \Sigma\cup\R=\left\{z\in\C\,|\,\Im(p(z))=0\right\}.
\end{equation}
By construction $p(z_2)=0$, and $\Im (p(z_1))=0$ since $p(z_1)$ is the real half $\mathcal{B}$-period of $\d p$ (see \eqref{Omega12}). The condition $\int_{\mathcal{A}}dp=0$ and Schwarz symmetries of $\d p$ and $\mathcal{A}$ imply that $\Im(p(\ov{z_1}))=0$, $\Im(p(\ov{z_2}))=0$, and $\Im(p(z))=0$ for any $z \in \R$. 
Thus the level set $ \Im(p(z)) = 0$ always contains the points $z_1,z_2, \ov{z_1}$, $\ov{z_2}$, and the real line. 

Let us look more in detail at the level set $\Im p(z) = 0$ off the real line. 
Each analytic arc of the level set $\Im p(z) = 0$ admits a
parametrization $z = z(\tau ),$ where $\tau\in\mathbb{R},$ with $|z'(\tau)| = 1.$  The function $p(z(\tau ))$ equals
\begin{equation*}
 p(z(\tau)) = p(z(\tau_0))+\int^{\tau}_{\tau_0} \d p(z(\sigma)) z'(\sigma) \d \sigma.
\end{equation*} 
Since $\Im p(z(\tau)) = 0,$ we have 
\begin{equation*}
 z'(\tau) = \dfrac{\ov{ \d p(z(\tau))}}{|\d p(z(\tau))|}.
\end{equation*}
This immediately gives us the direction of the level line for every point $z\in\Sigma$ and there is exactly one line passing through every regular point where $\d p(z)\neq 0$.
In addition to regular points, there are singular points, where $\d p(z)$ is either 0 or infinite.  
Let us first consider the latter case. At each branch  point of the elliptic curve  $\mathcal{S}$, the abelian integral  $p(z)$ has a zero of order $1/2$. For example, near $z=z_2$
\begin{equation*}
 \Im(p(z))\simeq 2\Im\left(\dfrac{z_2^2-\Re(z_1+z_2)z_2+c_0}{\sqrt{(z_2-z_1)(z_2-\ov{z_1})(z_2-\ov{z_2})} }(z-z_2)^{\frac{1}{2}}\right)=\Im\left(a_0\e^{\i \varphi_0}(z-z_2)^{\frac{1}{2}}\right)
\end{equation*}
for some constants $a_0>0$ and $\varphi_0 \in \Real$.
So arcs of $\Im p(z)=0$ emerge from $z=z_2$ at angles $-2\varphi_0+2\pi n$, with $n\in\Z$. It follows that only one arc emerges from $z_2$. The same considerations hold for all the other branch points. The other singularity of $p$ is at infinity where its local behavior is given by \eqref{e:pqasymp}. It follows that exactly two arcs of $\Im(p(z))$ emerge from the point at infinity, but these are precisely the two sides of the real line already accounted for. 

Regarding the two zeros of $\d p$, there are two distinct possibilities to consider: either the zeros are real or complex conjugates. Let $\nu_1$ and $\nu_2$  be the zeros  of $\d p $. If they are real, then $\Im(p(\nu_j))=0$, $j=1,2$  and  for $z$ near  $\nu_1$ we have  
\begin{equation*}
 \Im(p(z))\simeq \Im\left(\dfrac{\nu_1-\nu_2 }{2|\nu_1-z_1||\nu_1-z_2|}  (z-\nu_1)^2 \right)=\Im\left(a_0(z-\nu_1)^2\right)\,,
 \end{equation*}
where $a_0 \in \Real$. It follows that $\Im(p(z))=0 $ along arcs that emerge from $\nu_1$ at angles $\arg(z-\nu_1)=\frac{n\pi}{2}$, with $n\in\Z$. Specifically, four arcs emerge from $\nu_1$, two along the real axis and two which leave $\nu_1$ orthogonal to the real line. The same consideration applies to the point $\nu_2$.
So with real roots of $\d p$, the only possible topological configuration of the set $\Sigma$ consists of two arcs $\Sigma^+$ and $\Sigma^{-}$ where $\Sigma^+$  starts at $z_1$ passes through $\nu_1$  and   ends at $\ov{z_1}$; similarly, 
$\Sigma^{-}$ starts at $z_2$ passes through $\nu_2$  and   ends at  $\ov{z_2}$  (see Figure~\ref{fig:spectralcurves}  (a)).
When the two zeros of $\d p$ are complex, namely $\nu_2=\ov{\nu_1}$,
it is simple to argue that these two zeros cannot belong to the level set $\Im(p(z))=0$.  Indeed on the contrary, there would be four lines coming out of $\nu_1$, and assuming that $\Im(\nu_1)>0$ two of these lines can be connected to $z_1$ and $z_2$, while the other two  lines 
have to form a loop. In this case  $\Im(p(z))$   would be a harmonic function in a simply connected domain with zero boundary value, therefore it would be identically zero in the domain within the loop. But this is impossible.  We conclude that  $\Sigma$ consists of two paths:  the path $\Sigma^+$  that connects $z_1$ to $z_2$  and the path   $\Sigma^{-}$  that connects $\ov{z_1}$ to $\ov{z_2}$ (see Figure~\ref{fig:spectralcurves} (b)). For a general spectral configuration of finite gap NLS potential see \cite{BT2026}.

\begin{remark}
The existence of two real zeros of the differential $\d p$ defined by \eqref{p} implies that the spectrum of the ZS operator for potential $u_e$ given by \eqref{e:qpdef} crosses the real line as in Figure~\ref{fig:spectralcurves}$(a)$.  
The differential $\d p$ has two real zeros if 
\begin{gather}
\label{real_spectrum}
%c_0\leq 0 \;\Longrightarrow  
1-m + s - 2\dfrac{\mathrm{E}(m)}{\mathrm{K}(m)} < 0, 
\shortintertext{where}
s = \frac{ \Im (z_1 -\ov{z_2})^2}{|z_1 - \ov{z_2}|^2} = \sin^2(\arg(z_1 - \ov{z_2}))\in[0,1], \nonumber
\end{gather}
and $m$ is given by \eqref{m}.
If the left hand side of \eqref{real_spectrum} is positive, the spectrum of the ZS operator has no real intersections as in Figure~\ref{fig:spectralcurves}$(b)$.
The set of values $(m,s)$ where \eqref{real_spectrum} is satisfied is shown in blue in Figure~\ref{f:spectral_parameters}.
\end{remark}
\begin{figure}[htb]
	\centering
	\begin{overpic}[width=.4\textwidth]{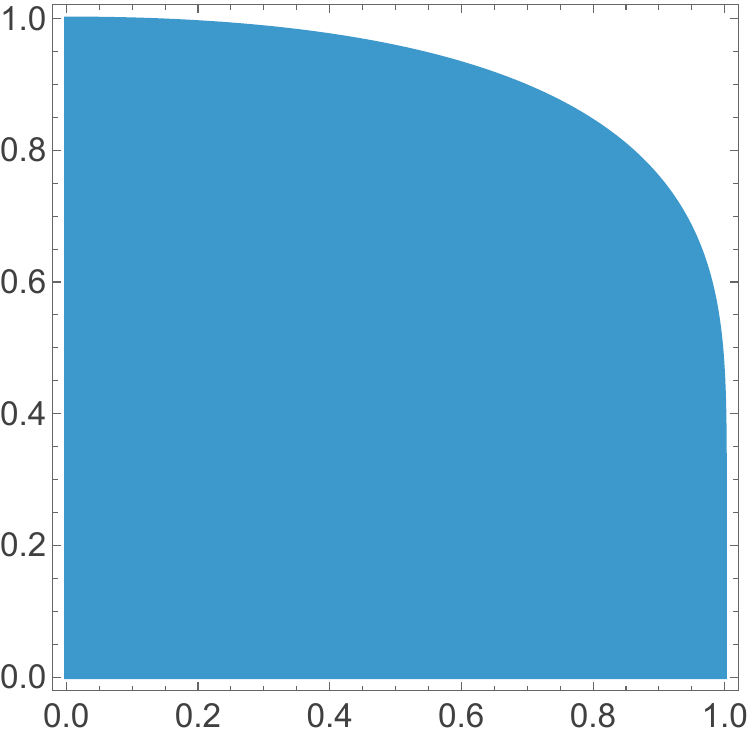}
	\put(51,-3){$m$}
	\put(-3,50){$s$}
	\end{overpic}
	\caption{ When $(m,s)$ is in the white region the spectrum \eqref{spectrum1} of the ZS operator has two arcs $\Sigma^+$ and $\Sigma^{-}$ that do not intersect the real axis $\R$ (Figure~\ref{fig:spectralcurves}(b)). 
    When $(m,s)$ is in the blue region the arcs $\Sigma^+$ and $\Sigma^{-}$ intersect the real axis $\R$ (Figure~\ref{fig:spectralcurves}(a)). }
	\label{f:spectral_parameters}
\end{figure}

\begin{remark}\label{prop:dn_vertical}
	Let $u_e(x,t)$ be a genus-one travelling wave  considered above, with nonreal
	bands $\Sigma^\pm$ disjoint from $\R$.
	Then $\Sigma^\pm$ are vertical segments if and only if $u_e(x,t)$
	is the $\dn$ wave,  up to Galilean transformation and  spatial and constant phase shifts.
\end{remark}

In this manuscript, we consider the case in which the spectrum $\Sigma$ consists of two arcs connecting $z_1$ to $z_2$ and $\ov{z_1}$ to $\ov{z_2}$ (the white region in Figure~\ref{f:spectral_parameters}). We also define the segment $\Sigma_0$\ connecting $z_1$ to $\overline{z_1}$,  (see Figure~\ref{surf1}). 
In addition, 
the quasi-momentum $p(z)$, satisfies the following inequalities.
\begin{lemma}
\label{dp_inequality}
The following inequalities are satisfied
 \be\label{e:imP}
\begin{cases}
    \Im (p(z)-E)>0\,, & z\in \mathbb{C}^+\setminus\Sigma^+\,,\\
    \Im (p(z)-E)<0 \,, & z\in \mathbb{C}^{-}\setminus\Sigma^{-}\,.\\
\end{cases}
\ee
\end{lemma}
\begin{proof}
In order to prove  the inequalities \eqref{e:imP}
we recall that $p(z)=\int_{z_2}^z  \d p(\lambda)$, where $\d p$ is defined in \eqref{p} and $E$ is real.  We have also defined  $\Sigma\cup\mathbb{R}$ as the level set where $\Im (p(z))=0$.  By the Schwarz symmetry of the differentials,
$\d p(\bar{z})=\overline{\d p(z)}$, hence the boundary values of $p(z)$ on $\mathbb{R}\cup\Sigma$ are real. It follows from the above that $\Im(p(z))$ is harmonic  on  $\mathbb{C}\setminus\{\Sigma^+\cup\Sigma^{-}\}$.  In addition, from the normalization at infinity we have \eqref{e:pqasymp}, so $\Im(p(z)-E)=\Im z+o(1)$. In particular, for $|z|$ large in the upper half-plane, the quantity  $\Im(p(z)-E)$ is positive, and for $|z|$ large in the lower half-plane, the quantity  $\Im(p(z)-E)$ is negative. Set 
$D^+:=\C^+\setminus\Sigma^+$ and $D^-:=\C^-\setminus\Sigma^{-}$. The function $\Im (p(z)-E)$ is harmonic on each domain $D^\pm$ and continuous up to the boundary, where it vanishes. Since $D^+$ (resp. $D^-$) is connected and contains a neighborhood of infinity in the upper (resp. lower) half-plane, the maximum principle together with the sign at infinity yields \eqref{e:imP}.
\end{proof}

The next goal of this section  is to obtain the solution of the Lax pair \eqref{e:NLSLP} for the elliptic solution  $u_e$ starting from the inverse problem formulated as a Riemann--Hilbert problem.

Let
 \begin{equation}
\label{Omega}
\Omega=\Omega_0+x\Omega_1+t\Omega_2= \Omega_1 ( x - x_0 + \Re( z_1+z_2) t )\,,
\end{equation}
where $\Omega_1$  and $\Omega_2$   have  been defined in \eqref{Omega12}\,,
 $ \Omega_0$ is a real constant and 
 \begin{equation}
     \label{x0}
x_0=-\frac{\Omega_0}{\Omega_1},
 \end{equation}
 where we observe   the invariance of the  Riemann-Hilbert problem~\ref{RHP:O} defined below   under the shift $x_0\to x_0+\dfrac{2\pi  n}{\Omega_1}$  $n\in\Z$.
 %We define  the  following  Riemann--Hilbert problem.
\begin{RHP}\label{RHP:O}
Find a $2\times2$ matrix-valued function $O(z) = O(z;x,t)$ which satisfies the following conditions:
\begin{enumerate}
    \item $O(z;x,t)$ is analytic in $\mathbb{C}\setminus\left\{ \Sigma_0 \cup\Sigma^+\cup \Sigma^{-}\right\}$,  where $  \Sigma_0 =[\overline{z_1},z_1]$  (see Figure~\ref{surf1}). 
    \item $O(z;x,t)=I+\mathcal{O}(z^{-1})$,\quad $z\to\infty$.
    \item $O(z;x,t)$ satisfies the jump conditions $O_+(z;x,t)=O_-(z;x,t)V^{(O)}(z;x,t)$, where
    \be\label{e:Ojump}
    V^{(O)}(z;x,t)=
    \begin{cases}
    \begin{pmatrix}
    	0 & \mathrm{i}\,\e^{2\mathrm{i}\,(xE+tN)}\\
\mathrm{i}\,\e^{-2\mathrm{i}\,(xE+tN)} & 0
    \end{pmatrix}\,, & z\in \Sigma^+\cup\Sigma^{-}\,,\\
    \e^{\mathrm{i} \,\Omega \sigma_3}\,, & z\in{\Sigma_0}\,.
     \end{cases}
\ee
\item $O(z;x,t)$ admits quartic root singularities at $z\in\{z_1, \overline{z_1},z_2,\overline{z_2}\}$.
\item $O(z;x,t)$ satisfies Schwarz symmetry: $\sigma_{2} O^*(z)\sigma_{2}=O(z)$, where $O^*(z;x,t)=\overline{O(\ov{z};x,t)}$.
\end{enumerate}
\end{RHP}
To construct the solution to RHP \ref{RHP:O}, we introduce the holomorphic differential $\omega$  given explicitly by
\begin{align}
\omega=\frac{c\d z}{\,R(z)}\,,
\end{align}
where the normalization constant $c$ has been defined in \eqref{e:defc} and it is purely imaginary.
We also define the period ratio
\be
\tau=\oint_\mathcal{B}\omega=\i\dfrac{\mathrm{K}(m')}{\mathrm{K}(m)},\quad m'=1-m\,,
%=\dfrac{\oint_{\mathcal{B}}\omega_0}{\oint_{\mathcal{A}}\omega_0},
\ee
and $m$ has been defined in \eqref{e:defc}.

  Next we introduce 
 the Jacobi theta function, \begin{align}
\theta_3(z;\tau)=\sum_{n\in\mathbb{Z}}\e^{2\i\pi nz+\i\pi n^2\tau},\quad z\in\mathbb{C},
\end{align}
that  satisfies the periodicity relations
\begin{align}
\theta_3(z+h+k\tau;\tau)=\e^{-\i\pi k^2\tau-2\i\pi kz}\theta_3(z;\tau),\quad h,k\in\mathbb{Z}.
\end{align}
We also recall that $\theta_3$ vanishes at $\frac{\tau}{2}+\frac12$.
Finally, we define the Abel integral
\begin{align}
A(z)=\int_{z_2}^{z}\omega,
\end{align}
where the path of integration avoids all the paths   $\Sigma_0 \cup\Sigma^+\cup \Sigma^{-}$.
The Abel map satisfies the following evaluation at specific points:
\begin{align}
\quad A_{+}(z_1)=\frac{\tau}{2},\quad A_{+}(\overline{z_1})=\frac{1}{2}+\frac{\tau}{2},\quad A(\overline{z_2})=\frac{1}{2}.
\end{align}
In addition, the following jump relations hold across the cuts:
\begin{align}
 & A(z_+)+A(z_-)=0,\quad z\in\Sigma^+\,,\\
 & A(z_+)-A(z_-)=\tau,\quad z\in  \Sigma_0 \,,\\
 & A(z_+)+A(z_-)=1,\quad z\in\Sigma^{-}\,,
\end{align}
where here and below we denote by $A(z_\pm)$ the boundary values of $A$ for $z$ approaching the oriented curves $\Sigma_\pm$ and $\Sigma_0$ from the left and the right.
\begin{Fact}
There is a unique solution to RHP \ref{RHP:O}, which can be expressed explicitly in the form:
\begin{align}\label{defO}
O(z;x,t) = 
\e^{\i(xE + t N)\sigma_3}
\begin{bmatrix}
\frac{\gamma+\gamma^{-1}}{2} H_{11}(z) & \frac{\gamma-\gamma^{-1}}{2} H_{12}(z) \\
\frac{\gamma-\gamma^{-1}}{2} H_{21}(z) & \frac{\gamma+\gamma^{-1}}{2} H_{22}(z)
\end{bmatrix}
\e^{-\i(xE + t N)\sigma_3}\, ,
\end{align}
where $\gamma = \gamma(z)$ is analytic in $\C\setminus\Sigma$ and defined by
\be \label{e:def.gamma}
\gamma(z)=\left(\frac{(z-z_2)(z-\overline{z_1})}{(z-z_1)(z-\overline{z_2})}\right)^{1/4},
\ee
normalized such that $\gamma(z)\to1$ as $z\to\infty$, and satisfying the jump condition
\be
\gamma_+(z)=\mathrm{i}\gamma_-(z),\qquad z\in\Sigma^+\cup\Sigma^{-}\,.
\ee
The matrix entries $H_{i j}$ $(i,j=1,2)$  are expressed as
\bse\label{e:Hij}
\begin{gather}
H_{11}(z) = \frac{\theta_3(0)\theta_3(A(z)-A(\infty)-\frac{\Omega}{2\pi})}{\theta_3(\frac{\Omega}{2\pi})\theta_3(A(z)-A(\infty))}, \quad
H_{12}(z) = \frac{\theta_3(0) \theta_3(A(z)+A(\infty)+\frac{\Omega}{2\pi})}{\theta_3(\frac{\Omega}{2\pi}) \theta_3(A(z)+A(\infty))}
%\e^{2\mathrm{i}(xE+tN)},
\\
H_{21}(z) = \frac{\theta_3(0)\theta_3(A(z)+A(\infty)-\frac{\Omega}{2\pi})}{\theta_3(\frac{\Omega}{2\pi})\theta_3(A(z)+A(\infty))}
%\e^{-2\mathrm{i}(xE+tN)}
, \quad
H_{22}(z) = \frac{\theta_3(0)\theta_3(A(z)-A(\infty)+\frac{\Omega}{2\pi})}{\theta_3(\frac{\Omega}{2\pi})\theta_3(A(z)-A(\infty))}.
\end{gather}
\ese
\end{Fact}
\begin{proposition}\label{prop:W0}
The  matrix  
\be\label{defWbg}
W_e(x,t;z)=O(z;x,t)\e^{-\mathrm{i}\left((x-x_0)(p(z)-E)+t(q(z)-N)\right)\sigma_3}\,,
\ee
with $O(z;x,t)$ as in \eqref{defO},  $p(z)$ and $q(z)$ as in \eqref{e:defAbelian},
solves the Lax pair \eqref{e:NLSLP} with potential 
\be
\label{elliptic_general}
u_e(x,t; v,x_0)=\Im(z_2-z_1) \frac{\theta_3(0) \theta_3(2A(\infty)+\frac{\Omega}{2\pi})}{\theta_3(\frac{\Omega}{2\pi}) \theta_3(2A(\infty))}\e^{2\mathrm{i}(xE+tN)},
\ee
with $\Omega$ as in \eqref{Omega}, $E $ and $N$ as in \eqref{e:defEN},  and $v=-\Re(z_1+z_2)$.

When $\Im(z_1)=\Im(z_2)$, the above formula for the potential needs to be modified to the form 
\be
 \label{ellitpic_cn0}
 u_e(x,t,v,x_0)=\Re(z_1-z_2)\Im(z_2)\frac{\theta_3(0) \theta_3(\frac{\Omega}{2\pi}+\frac{\tau+1}{2})}{\theta_3(\frac{\Omega}{2\pi})\theta_3'(\frac{\tau+1}{2})c}\e^{2\mathrm{i}(xE+tN)}\,,
 \ee
 with the constant $c$ as in \eqref{e:defc}. 
 Further we have 
 \be
 \label{O_expansion}
O(z;x,t)=I+\frac{1}{2\mathrm{i} z}\begin{pmatrix}
	\int_{x_0+vt}^x \left[|u_e(s,t)|^2-2\pi_1\right]\d s &  u_e(x,t)\\
	u_e^*(x,t) & -\int_{x_0+vt}^x \left[|u_e(s,t)|^2-2\pi_1\right]\d s 
\end{pmatrix}+\mathcal{O}(z^{-2}),
\ee
where $\pi_1$ is the coefficient of $z^{-1}$ in the
expansion \eqref{e:pqasymp} of the quasi-momentum $p(z)$,
and is given in \eqref{e:defEN}.

 \end{proposition} 
\begin{proof}
By relating the jump conditions of the Riemann--Hilbert problem \ref{RHP:O} for $O(z)$ to those in $W_e$, we obtain
\begin{align}\label{e:jumpW0}
	\begin{split}
W_e(z_+)&=W_e(z_-)\e^{\i\left((x-x_0)(p(z_-)-E)+t(q(z_-)-N)\right)\sigma_3}V^{(O)} \e^{-\i\left((x-x_0)(p(z_+)-E)+t(q(z_+)-N)\right)\sigma_3}\\
&=W_e(z_-)
\begin{cases}
 \i \e^{2\i x_0 E\sigma_3}\sigma_1, & z \in \Sigma^+\cup\Sigma^{-}, \vspace{8pt}\\
I, & z \in  \Sigma_0,
\end{cases}
\end{split}
\end{align}
where $x_0=-\frac{\Omega_0}{\Omega_1}$ as in \eqref{x0} and $\sigma_1=\begin{psmallmatrix}0&1\\1&0\end{psmallmatrix}$. The first jump holds because $p(z_{+})+p(z_{-})=0$ and $q(z_{+})+q(z_{-})=0$ as $z\in\Sigma^+ \cup \Sigma^{-}$, and the second one holds because $p(z_+)-p(z_-)=\Omega_1$ and $q(z_+)-q(z_-)=\Omega_2$  as $z\in \Sigma_0$.
However, these jumps are independent of $x$ and $t$, and therefore $(W_e)_x$ and $(W_e)_t$   satisfy the same jump conditions as $W_e$. 
It follows that the products $(W_e)_{x}(W_e)^{-1}$ and $(W_e)_{t}(W_e)^{-1}$
are analytic across $\Sigma_0 \cup\Sigma^{+}\cup\Sigma^{-}$, and  can admit at worst isolated singularities at the four branch points. However, the growth condition at each endpoint guarantees the local growth rate of the product is at most $\frac{1}{2}$-root blow up; it follows that the branch point singularities are removable   and the products $(W_e)_{x}(W_e)^{-1}$ and $(W_e)_{t}(W_e)^{-1}$  are entire functions of $z$. They  have  the  asymptotic expansions
\bse
\begin{align}\label{W0xW0-1}
	\begin{split}
		(W_e)_{x}(W_e)^{-1} & =O_{x} O^{-1}+O\left(-\i (p(z)-E) \sigma_{3}\right) O^{-1} \\
		& =\mathcal{O}\left(z^{-1}\right)+\left(I+\frac{O^{(1)}}{z}+\frac{O^{(2)}}{z^2}\ldots\right)\left(-\i \sigma_{3} (z+\dfrac{\pi_1}{z}+O\left(z^{-2}\right)\right)\left(I-\frac{O^{(1)}}{z}+\frac{(O^{(1)})^2-O^{(2)}}{z^2}\ldots\right) \\
		& =-\i z \sigma_{3}+\i\left[\sigma_{3}, O^{(1)}\right]+\mathcal{O}(z^{-1}),
	\end{split}
\end{align}
and
\be\label{W0tW0-1}
(W_e)_{t}(W_e)^{-1}=-\i z^2\sigma_3+\i z[\sigma_3,O^{(1)}]+\i\left([O^{(1)},\sigma_3]O^{(1)}+[\sigma_3,O^{(2)}]\right)+\mathcal{O}(z^{-1}),
\ee
\ese
where we used the asymptotic expansion of $O(z;x,t)$ and $O(z;x,t)^{-1}$:
\[
O(z;x,t)=I+\frac{O^{(1)}}{z}+\frac{O^{(2)}}{z^2}+\mathcal{O}(z^{-3}), \quad O(z;x,t)^{-1}=I-\frac{O^{(1)}}{z}+\frac{(O^{(1)})^2-O^{(2)}}{z^2}+\mathcal{O}(z^{-3}).
\]
Since $(W_e)_{x} (W_e)^{-1}$ and  $(W_e)_{t}(W_e)^{-1}$ are entire, the asymptotic tail $\mathcal{O}(z^{-1})$ is zero. Therefore  $(W_e)_{x} (W_e)^{-1}$  and $(W_e)_{t}(W_e)^{-1}$  are polynomials.
Multiplying both sides of \eqref{W0xW0-1} and \eqref{W0tW0-1} by $W_e$ from the right finally shows that $W_e$ defined in \eqref{defWbg} satisfies the linear equations 
\bse
\begin{gather}
\label{W0x}
	(W_e)_{x}=\left\{-\i z\sigma_{3}+\begin{pmatrix}
		0 & 2 \i O^{(1)}_{12}(x, t) \\
		-2 \i O^{(1)}_{21}(x, t) & 0
	\end{pmatrix}\right\} W_e\,,\\
	\label{W0t}
	(W_e)_{t}=\left\{-\i z^2\sigma_{3}
	+\begin{pmatrix}
		0 & 2\i z O_{12}^{(1)}\\
		-2\i z O_{21}^{(1)} & 0
	\end{pmatrix}
	+\begin{pmatrix}
		-2\i O_{12}^{(1)} O_{21}^{(1)} & -2 \i O^{(1)}_{12}O_{22}^{(1)}+2\i O_{12}^{(2)} \\
		2 \i O^{(1)}_{21} O_{11}^{(1)}-2\i O_{21}^{(2)} & 2\i O_{12}^{(1)}O_{21}^{(1)}
	\end{pmatrix}\right\} W_e.
\end{gather}
\ese
Comparing \eqref{e:zs} with \eqref{W0x} 
it also follows that $u_e(x,t)=2\i O_{12}^{(1)}(x,t)$  and $\overline{u_e(x,t)}=2 \i O^{(1)}_{21}(x, t) $.
Expanding the explicit solution \eqref{defO} for large $z$ we obtain the formula \eqref{elliptic_general} for $u_e(x,t)$, the elliptic solution of NLS.

The expression \eqref{elliptic_general} for $u_e(x,t)$ appears to give a zero potential when $\Im(z_2-z_1)=0$, but this is not the case.
When the branch points are of the form  $z_1=\xi_1+\i\eta_2 $   and  $z_2=\xi_2+\i\eta_2 $  where  $\eta_2>0$, we have $2A(\infty)=\frac{\tau+1}{2}$, so that $\theta_3(2A(\infty))=0$.
In this case \eqref{elliptic_general} admits further reduction.  
Indeed, as $z\to\infty$,  we have 
\begin{align*}
	2\i O_{12}(z)
	&=2\i\frac{\theta_3(0) \theta_3(\int_\infty^z \omega+\frac{\tau+1}{2}+\frac{\Omega}{2\pi})}{\theta_3(\frac{\Omega}{2\pi}) }
	\frac{\frac{1}{2}\left(\gamma(z)-\frac{1}{\gamma(z)}\right)}{ \theta_3\left(\frac{\tau+1}{2} +\int_\infty^z \omega\right)}\e^{2\i(xE+tN)}\\
	&=\frac{\theta_3(0) \theta_3(\frac{\tau+1}{2}+\frac{\Omega}{2\pi})}{\theta_3(\frac{\Omega}{2\pi}) }\dfrac{\Re(z_1-z_2)\eta_2}{\theta_3'(\frac{\tau+1}{2})c}\dfrac{1}{z}\e^{2\i(xE+tN)}+\mathcal{O}(z^{-2}),
\end{align*}
where  $c$ is defined in \eqref{e:defc}.  As $u_e(x,t) = 2\i \lim\limits_{z\to \infty}\bigl(zO_{12}\bigr)$, we have \eqref{ellitpic_cn0} when $\Im(z_1) = \Im(z_2)$.

To prove that   \eqref{W0t}  coincides with the time-derivative  linear equation \eqref{e:NLSLP2}  of the Lax  pair,   we need to consider the   coefficient of $z^{-1}$  in the  expansion in \eqref{W0xW0-1} that gives
\[
O^{(1)}_x=-\i[\sigma_3, O^{(2)}]+\i\pi_1\sigma_3+\i[\sigma_3,O^{(1)}]O^{(1)}.
\]
The above equation gives the identity
\begin{gather}
\label{O12x}
	O^{(1)}_{x}=\i\pi_1\sigma_3+2 \i \begin{pmatrix}
		0 & -O^{(2)}_{12}(x, t) \\
		 O^{(2)}_{21}(x, t) & 0
	\end{pmatrix}+2\i  \begin{pmatrix}
		O_{12}^{(1)} O_{21}^{(1)} &O^{(1)}_{12}O_{22}^{(1)}\\
		- O^{(1)}_{21} O_{11}^{(1)}& -O_{12}^{(1)}O_{21}^{(1)}
	\end{pmatrix}\,,
	\end{gather}
that implies
\begin{align}
&(O^{(1)}_{11})_{x}=\i\pi_1+2\i O_{12}^{(1)} O_{21}^{(1)},\quad(O^{(1)}_{22})_{x}=-\i\pi_1-2\i O_{12}^{(1)} O_{21}^{(1)}\,,\\
&(O^{(1)}_{12})_{x}=-2 \i O^{(2)}_{12}+2\i O^{(1)}_{12}O_{22}^{(1)},\quad
(O^{(1)}_{21})_{x}=2 \i O^{(2)}_{21}-2\i O^{(1)}_{21}O_{11}^{(1)}.
\end{align}
The above two relations show that \eqref{W0t} has the form \eqref{e:NLSLP}  and 
\begin{equation}
2\i(O^{(1)}_{11})_{x}=-2\i(O^{(1)}_{22})_{x}=-2\pi_1+|u_e(x,t)|^2.
\end{equation}
In order to obtain \eqref{O_expansion}  we just need to integrate the above relation over $x$ keeping in mind 
that when $x=x_0+vt$  then $O^{(1)}_{11}(x_0+vt,t)=0$ because of the parity of the theta-function.
\end{proof}

 The above  formulas  \eqref{elliptic_general} and   \eqref{ellitpic_cn0}  for the potential $u_e(x,t)$ can be written using Jacobi elliptic functions.
 
 \begin{proposition}
 \label{proposition_elliptic}
 The elliptic potential  \eqref{elliptic_general} and \eqref{ellitpic_cn0} of the NLS solution can be expressed via Jacobi elliptic functions as 
 \be\label{e:|u|^2}
|u_e(x,t)|^2=(\Im(z_1+z_2))^2-4\Im(z_1)\Im(z_2)\sn^2 \left(|z_1-\overline{z_2}|  ( x - x_0 + \Re( z_1+z_2) t )+\mathrm{K}(m);m\right), \quad 
\ee
where  $\sn(z;m)$ is the Jacobi elliptic function and  $\mathrm{K}(m)$ is   the complete elliptic integral of  the first kind with parameter $m$ given by \eqref{m}.
The average of $|u_e(x,t)|^2$ over a period $L=\frac{2\mathrm{K}(m)}{|z_1-\overline{z_2}|}$ gives
\begin{equation}
\label{period_u^2}
\langle |u_e|^2\rangle:=\frac{1}{L}\int_0^L|u_e(x,t)|^2\d x=|z_1-\overline{z_2}|^2\dfrac{\mathrm{E}(m)}{\mathrm{K}(m)}-(\Re(z_1-z_2))^2.
\end{equation}
Further, when $\Re(z_1)=\Re(z_2)$  namely $z_1=-\frac{v}{2}+\mathrm{i} \eta_1$ and  $z_2=-\frac{v}{2}+\mathrm{i}\eta_2$, with $\eta_2>\eta_1$ then $E=\frac{v}{2}$,
$N=-\frac{v^2}{4}-\frac{\omega_0}{2}$ so that 
\begin{equation}
\label{e:dn}
u_e(x,t)=|z_1-\overline{z_2}|\dn \left(|z_1-\overline{z_2}|  ( x - x_0 -v t )+\mathrm{K}(m);m\right)\e^{-\mathrm{i} \omega_0 t}\,\e^{\mathrm{i}  \big(vx - \frac{v^2}2t\big)}. 
\end{equation}
When $\Im(z_1)=\Im(z_2)$  namely $z_1=\xi_1+\mathrm{i}\eta$ and  $z_2=\xi_2+\mathrm{i}\eta$, with $\eta>0$,  $\xi_1<\xi_2$ and $\xi_1+\xi_2=-v$,  then 
\begin{equation}
\label{elliptic_cn}
u_e(x,t)= 2\mathrm{i} \eta \cn\left(|z_1-\overline{z_2}|  ( x - x_0 -v t )+\mathrm{K}(m);m\right)\,\e^{\mathrm{i}(vx-\frac{v^2}{2}t-\omega_0t-\frac{\Omega_0}2)}\,,
\end{equation} 
for the  real constant $\Omega_0$   defined in  \eqref{Omega}.
\end{proposition}
The  Proposition \ref{proposition_elliptic} is proved in  the Appendix \ref{a:u}.
\begin{remark}
It is a straightforward substitution to check that the above formulas  \eqref{e:|u|^2},  \eqref{e:dn} and  \eqref{elliptic_cn} for the travelling wave elliptic solutions of the NLS equation  coincide, up to a Galilean transformation  \eqref{travelling} and a shift,  with the  formulas obtained in \eqref{E:defs} and \eqref{eq:dn_cn}.
\end{remark}
\begin{remark}
The average $\langle |u_e|^2\rangle$ in \eqref{period_u^2}
equals $2\pi_1$, twice the coefficient of $z^{-1}$
in the expansion \eqref{e:pqasymp} of the quasi-momentum. This is quite a general fact, since the quasi-momentum differential $\d p$  is a generating function  for  conserved quantities, \cite{FT} see also \cite{BT2026}.
We can therefore re-express the expansion 
 of $O(z;x,t) $ in Proposition  \ref{prop:W0} in the form 
\be
\label{e:Oasymp}
O(z;x,t)=I+\frac{1}{2\mathrm{i} z}\begin{pmatrix}
	\int_{x_0+vt}^x \left[|u_e(s,t)|^2-\langle |u_e|^2\rangle\right]\d s &  u_e(x,t)\\
	u_e^*(x,t) & -\int_{x_0+vt}^x \left[|u_e(s,t)|^2-\langle |u_e|^2\rangle\right]\d s 
\end{pmatrix}+\mathcal{O}(z^{-2}).
\ee
\end{remark}

The following proposition shows that a nontrivial overlap of two
distinct genus-one spectral bands is necessarily vertical. 

\begin{proposition}\label{prop:verticaloverlap}
	Let $\Sigma^{\ell,+}$ and $\Sigma^{r,+}$ be the regular spectral
	bands considered above, with distinct endpoint sets. If their
	intersection contains a nontrivial open subarc, then there exists
	$\xi\in\R$ such that
	\be
	\Sigma^{s,+}=\xi+\i[\eta_1^s,\eta_2^s],
	\qquad 0<\eta_1^s<\eta_2^s,\quad s\in\{\ell,r\}.
	\ee
	In particular, their intersection is a single nontrivial
	segment, and both backgrounds are Galilean transforms of
	$dn$ waves with the same velocity $v^\ell=v^r=-2\xi$.
\end{proposition}
The proof is in Appendix~\ref{a:verticaloverlap}

\section{Perturbative Argument}\label{perturbative}
Consider the ZS  problem \eqref{e:zs} given initial data $u = u_0$ which satisfies \eqref{initial_data} for two given elliptic traveling wave solutions of \eqref{e:nls}
$u_e^\ell(x) = u_e(x,t=0; v^\ell, x_0^\ell)$ as $x\to -\infty$ and  $u_e^{r}(x) = u_e(x,t=0; v^r, x_0^r)$ as $x\to +\infty$. 
We denote by $W^s_e(x;z)$
the  solution of the ZS problem when the potential is  $u^s_e(x)$  on the whole real line,  with $s$ in the set $\{\ell,r\}$, that is,
\be\label{e:defW0}
W^s_e(x;z):=O^{s}(z;x,0)\e^{-\i (x-x_0^s)(p^s(z)-E^s)\sigma_{3}},\quad \mbox{$s\in\{\ell,r\}$\,,}
\ee
where $p(z)$, $O(z;x,t)$ and $x_0$ are defined in \eqref{e:defAbelian}, \eqref{defO} and \eqref{defWbg}, respectively, and $E$ is determined by the second term of the asymptotic expansions of $p(z)$ introduced in \eqref{e:defEN}. 
For our step-like potential $u_0$ we seek solutions $W^{s}(x;z)$,  with $s$ in the set $\{\ell,r\}$,  of \eqref{e:zs} such that 
\begin{subequations}\label{e:Wpmasy}
\begin{align}
 W^{r}(x;z)&=O^{r}(z;x,0)(I+\mathcal{O}(x^{-1}))\e^{-\i (x-x^r_0)(p^r(z)-E^r)\sigma_{3}}, \quad \mbox{as $x\to  \infty$},\\
 W^{\ell}(x;z)&=O^{\ell}(z;x,0)(I+\mathcal{O}(x^{-1}))\e^{-\i (x-x^\ell_0)(p^\ell(z)-E^\ell)\sigma_{3}}, \quad \mbox{as $x\to-\infty$}.
 \end{align}\end{subequations}
We define the matrix  functions $m^{s}(x;z)$, with $s=\ell $ or $s=r$,   as
\begin{align}\label{e:defm}
W^{s}(x;z)=O^s(z;x,0) m^s(x;z)\e^{-\i (x-x_0^s)(p^s(z)-E^s)\sigma_{3}}.
\end{align}
For any $z \in \mathbb{C}\setminus\partial\Sigma$, $O^s(z;x,0)$ is uniformly bounded in $x$.

\begin{lemma} 
Suppose $W^s(x;z)$, $s\in\{\ell,r\}$ are solutions of the ZS  equation \eqref{e:zs} with potential $u_0(x)$.
Then $m^s(x;z)$, defined in \eqref{e:defm}, satisfies the following ODE
\begin{align}\label{e:ODEm}
\partial_{x} m^s=-\mathrm{i} (p^s(z)-E^s)\left[\sigma_{3}, m^s\right]+O^s(z;x,0)^{-1} \Delta U^s(x) O^s(z;x,0) m^s, 
\end{align}
where $\Delta U^s:=\begin{psmallmatrix}0 & \Delta u^s(x) \\ -\ov{\Delta u^s(x)} & 0\end{psmallmatrix}$, $\Delta u^s(x):= u_0(x)-u_{e}^s(x)$, and $[a,b]=ab-ba$ denotes the  matrix  commutator.
\end{lemma}
\begin{proof}
Since $W^s_e(x;z)=O^s(z;x,0)\e^{-\i(x-x_0^s)(p^s(z)-E^s)\sigma_{3}}$ is a solution of ZS for the elliptic potential $u^s_e(x)$, we rewrite \eqref{e:defm} as
\begin{align*}
W^s(x;z)=O^s(z;x,0) m^s(x;z) \e^{-\i(x-x_0^s)(p^s(z)-E^s)\sigma_{3}}=W_e^s(x;z) \e^{\i (x-x_0^s)(p^s-E^s) \sigma_{3}}m^s(x;z)\e^{-\i(x-x_0^s)(p^s(z)-E^s)\sigma_{3}}.
\end{align*}
Taking the derivative of both sides w.r.t. $x$, we have
\begin{align*}
W_{x}^s &=\left(-\i z\sigma_{3}+U^s_e \right)W^s+W_e^s\e^{\i(x-x_0^s)(p^s(z)-E^s)\sigma_{3}}\left(\i (p^s(z)-E^s) \sigma_{3}\right)m^s\e^{-\i(x-x_0^s)(p^s(z)-E^s)\sigma_{3}}\\
&+W^s_e \e^{\i(x-x_0^s)(p^s(z)-E^s)\sigma_{3}} m_{x}^s \e^{-\i(x-x_0^s)(p^s(z)-E^s)\sigma_{3}}
+W_e^s \e^{\i(x-x_0^s)(p^s(z)-E^s)\sigma_{3}}m^s\left(-\i(p^s(z)-E^s)\sigma_{3}\right)\e^{-\i(x-x_0^s)(p^s(z)-E^s)\sigma_{3}}.
\end{align*}
Plugging this into the ZS  equation \eqref{e:zs} and solving for $m_{x}^s$ then gives \eqref{e:ODEm}.
\end{proof}
Equation \eqref{e:ODEm} gives the Volterra integral equation:
\bse\label{e:Intm}
\begin{align}
m^\ell(x;z)&=I+\int_{-\infty}^{x} \e^{-\i (x-y) (p^\ell(z)-E^\ell) \sigma_{3}} O^\ell(y;z)^{-1} \Delta U^\ell(y)O^\ell(y;z) m^\ell(y;z)\e^{\i (x-y) (p^\ell(z)-E^\ell) \sigma_{3}}  \d y\,,\label{e:Intml}
\\
m^r(x;z)&=I+\int_{\infty}^{x} \e^{-\i (x-y) (p^r(z)-E^r) \sigma_{3}} O^r(y;z)^{-1} \Delta U^r(y)O^r(y;z) m^r(y;z)\e^{\i (x-y) (p^r(z)-E^r) \sigma_{3}}   \d y\,.\label{e:Intmr}
\end{align}
\ese

\begin{figure}[t]
 \centering
 \begin{subfigure}[b]{0.3\textwidth}
  \centering
  \includegraphics[width=\linewidth]{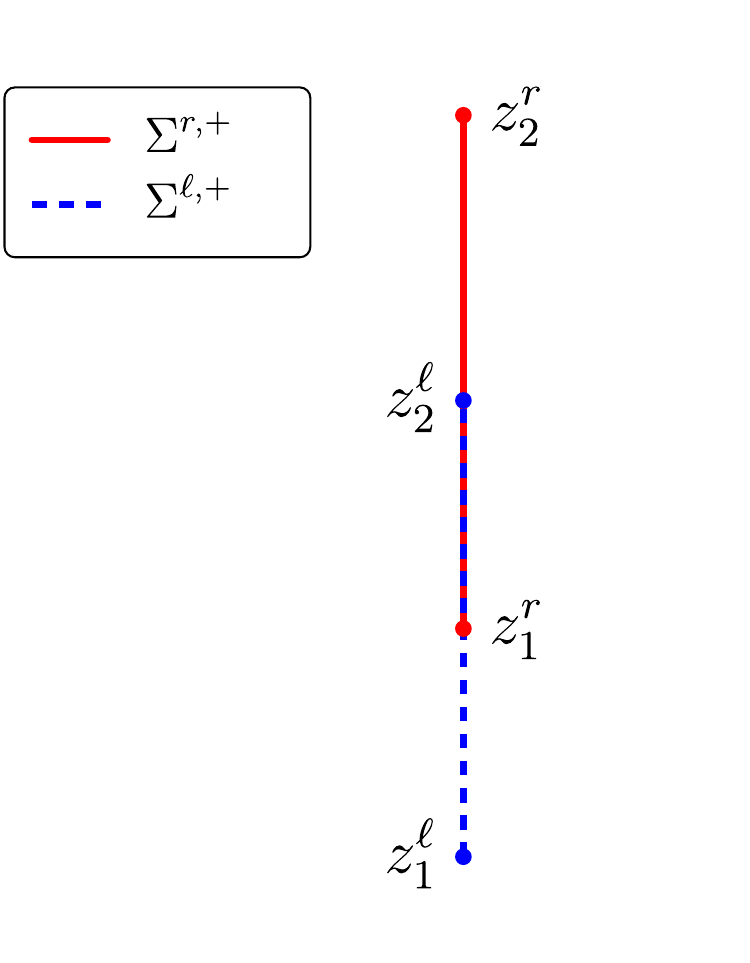}
  \caption{}
 \end{subfigure}\hspace{2.5cm}
 \begin{subfigure}[b]{0.3\textwidth}
  \centering
  \includegraphics[width=\linewidth]{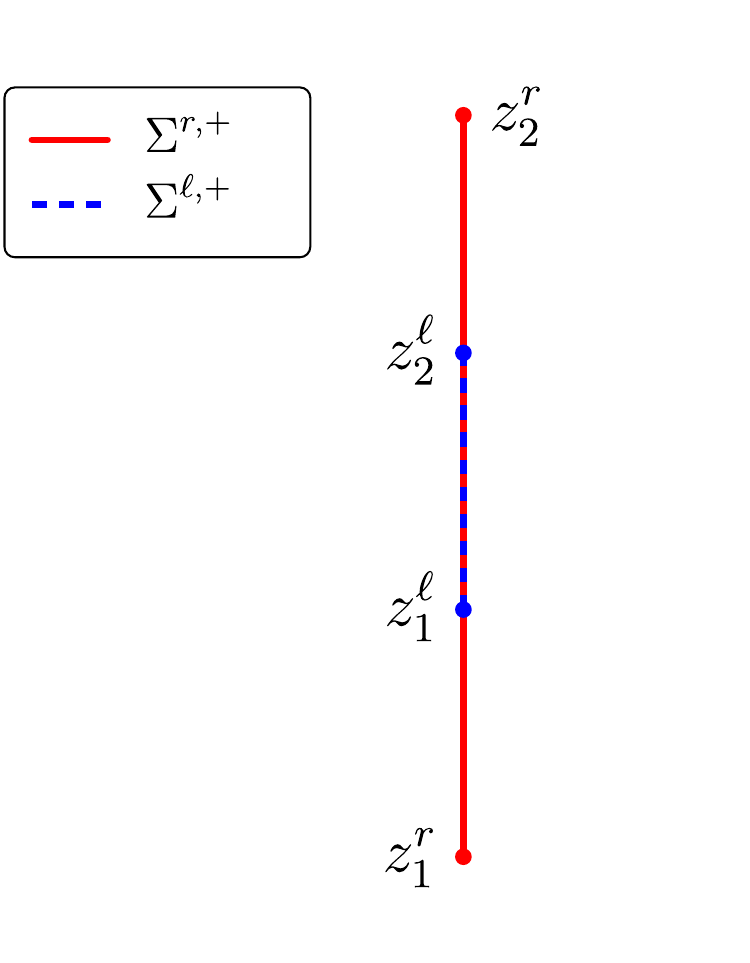}
  \caption{}
 \end{subfigure}\hfill
 \caption{ Overlapping upper spectral bands with pairwise distinct
 	endpoints.
 	(a) The bands overlap on $[z_1^r,z_2^\ell]$;
 	(b) $\Sigma^{\ell,+}$ is strictly contained in
 	$\Sigma^{r,+}$.
 	Interchanging $\ell$ and $r$ gives the other two configurations.}

 \label{fig:cuts}
\end{figure}
Hereafter, we define 
\be
\Sigma^\pm:=\Sigma^{r,\pm} \cup \Sigma^{\ell,\pm},\quad
\Sigma^r:=\Sigma^{r,+}\cup\Sigma^{r,-}
,\quad \Sigma^\ell:=\Sigma^{\ell,+}\cup\Sigma^\ell_-
\,,\quad \Sigma_{\mathrm{tot}}:=\Sigma^r\cup\Sigma^\ell\,,
\ee
where the superscript ``$+$'' (resp. ``$-$'') denotes the part of the spectrum in the upper (resp. lower) half-plane; see Figure~\ref{fig:cuts}.

The next proposition gathers standard facts needed in the Riemann--Hilbert analysis. The proof appears in Appendix \ref{a:propm}.
\begin{proposition}\label{propm}
Suppose $u_0-u_e^\ell \in L^1(\Real^-)$ and $u_0-u_e^r \in L^1(\Real^+)$. Then $W^s(x;z)$, $s \in \{\ell,r\}$, have the following properties:
\begin{enumerate}
\item \label{analyticity}
For every $z_\pm\in \R \cup \operatorname{int}(\Sigma^{s,\pm})$\footnote{Here $z_\pm \in \operatorname{int}(\Sigma^{s,\pm})$ denotes the left/right boundary values on the interior of $\Sigma^s$.}, and each $x_*\in\Real$, there exist unique solutions $W^\ell(\,\cdot\,;z)\in L^\infty(-\infty,x_*)$ and $W^r(\,\cdot\,;z)\in L^\infty(x_*,\infty)$ of \eqref{e:zs} given by \eqref{e:defm}, where $m^s(x;z)$, $s\in\{\ell,r\}$ are solutions of the integral equations \eqref{e:Intm}. 
Moreover, for each fixed $x\in \R$, the columns $W^s_j(x;z)$, $j=1,2,$ are continuous for $z_\pm\in \R \cup \operatorname{int}(\Sigma^{s,\pm})$ and extend continuously to analytic functions in the following domains:
\begin{multicols}{2}
\begin{enumerate}
\item $W_{1}^\ell(x ; z)$ is analytic for $z \in \mathbb{C}^+ \backslash \Sigma^{\ell,+}$.
\item $W_{2}^\ell(x; z)$ is analytic for $z\in \mathbb{C}^- \backslash \Sigma^{\ell,-}$.  
\item $W_{1}^r(x; z)$ is analytic for $z\in \mathbb{C}^- \backslash \Sigma^{r,-}$. 
\item $W_{2}^r(x; z)$ is analytic for $z\in \mathbb{C}^+ \backslash \Sigma^{r,+}$.
\end{enumerate}
\end{multicols}

\item \label{jumps}
For $s\in\{\ell,r\}$ and $z \in \Sigma^s$, the boundary values $W^s(x;z_\pm)$ satisfy
\begin{equation}\label{jumpW}
    W^s(x;z_+)= W^s(x;z_-)\, \e^{2\mathrm{i} x_0^s E^s \sigma_3} (\mathrm{i} \sigma_1),  \qquad z \in \Sigma^s.
\end{equation}

\item \label{asymptotics}
Given $n\geq 2$, suppose that $u_0-u_e^\ell(x)\in \mathcal{W}^{n,1}(\Real^-)$ and $u_0-u_e^r(x)\in \mathcal{W}^{n,1}(\Real^+)$.
 Then, for each fixed $x$, the normalized columns in
\eqref{e:asymptoticsW} admit expansions through order
$z^{-(n-1)}$, with remainder $\mathcal{O}(z^{-n})$,
in the corresponding closed half-planes. To leading order,
\bse\label{e:asymptoticsW}
\begin{align}
&\begin{aligned}
    W^\ell_1(x;z) \e^{\mathrm{i}  (x- x_0^\ell)z} 
    &=\begin{pmatrix} 1\\0 \end{pmatrix} + 
    \frac{1}{2\mathrm{i} z} \begin{bmatrix}
	  A^\ell(x)   &%\vspace{0.8ex}\\
	  \overline{u_0(x)} 
    \end{bmatrix}^\intercal
+ \mathcal{O}\left( z^{-2} \right), \\
W^r_2(x;z) \e^{-\mathrm{i}  (x- x_0^r)z} 
&=\begin{pmatrix} 0\\1 \end{pmatrix} + 
    \frac{1}{2\mathrm{i} z} \begin{bmatrix}
	    u_0(x) & %\vspace{0.8ex}\\
	   -A^r(x)
    \end{bmatrix}^\intercal
+ \mathcal{O}\left( z^{-2} \right)
\end{aligned}
\qquad \overline{\C^+} \ni z \to \infty, \\%   
&\begin{aligned}
    W^r_1(x;z) \e^{\mathrm{i}  (x- x_0^r)z} &=  
    \begin{pmatrix} 1\\0 \end{pmatrix} + 
    \frac{1}{2\mathrm{i} z} \begin{bmatrix}
	  A^r(x)   &%\vspace{0.8ex}\\
	  \overline{u_0(x)} 
    \end{bmatrix}^\intercal
+ \mathcal{O}\left( z^{-2} \right), \\
W^\ell_2(x;z) \e^{-\mathrm{i}  (x- x_0^\ell)z} 
&=\begin{pmatrix} 0\\1 \end{pmatrix} +
    \frac{1}{2\mathrm{i} z} \begin{bmatrix}
	    u_0(x) & %\vspace{0.8ex}\\
	   -A^\ell(x)
    \end{bmatrix}^\intercal
+ \mathcal{O}\left( z^{-2} \right)
\end{aligned}
\qquad \overline{\C^-} \ni z \to \infty.
\end{align}
\ese
Here, $A^s(x)$ are functions satisfying $\frac{\mathrm{d}}{\mathrm{d} x} A^s(x) = |u_0(x)|^2$, $s\in \{ \ell,r \}$.

\item \label{singularity}
Suppose $u_0-u_e^\ell \in L^{1,1}(\R^-)$ and $u_0 -u_e^r \in L^{1,1}(\R^+)$. Fix $x_*\in\R$. Then there exists a  constant $C>0$ (independent of $z$) such that
% \todo{This is true also for $W(x;z)$}
\bse\label{e:singW}
\begin{gather}
	| W^\ell(x;z) | \leq \sup_{x< x_*} | {W}^\ell_e(x;z) | \cdot \exp \left( C (|x_*|+1) \| \Delta u^\ell \|_{L^{1,1}(-\infty,x_*]} \right),  \qquad z \to z_j^\ell (j=1,2), \quad x < x_*.\label{e:singWl}\\
		| W^r(x;z) | \leq \sup_{x> x_*} | {W}^r_e(x;z) | \cdot \exp \left( C (|x_*|+1) \| \Delta u^r \|_{L^{1,1}(x_*,+\infty)} \right),  \qquad z \to z_j^r (j=1,2), \quad x > x_*.\label{e:singWr}
\end{gather}
\ese
The limits in \eqref{e:singW} are taken along either side
of the corresponding spectral band.
Each solution $W^s(x;z)$, $s \in \{\ell,r\}$, has at most fourth-root singularities at $z_j^s, \ j=1,2$, such that 
\begin{equation}\label{e:W.endpoints}
\begin{aligned}
      &\gamma^{s}(z_\pm)^{-1}  W^s(x;z_\pm) = \e^{\i xE^s\sigma_3} \widehat{W}^s(x;z_1^s) + o(1)
      && z_\pm \to z_1^s , \quad z_\pm \in \Sigma^s, \\
      &\gamma^{s}(z_\pm)  W^s(x;z_\pm) = \e^{\i xE^s\sigma_3} \widehat{W}^s(x;z_2^s) + o(1)
      && z_\pm \to z_2^s , \quad z_\pm \in \Sigma^s.
\end{aligned}
\end{equation}
Here, $\gamma^s$ is the function \eqref{e:def.gamma} with branch points in $\partial \Sigma^s$ and  $\widehat{W}^s(x;z_j^s)$ is the endpoint limit of the function defined in \eqref{e:defWhat}, characterized by the integral equation \eqref{e:endpoint.integral.eq}. For each fixed $x$, the fourth-root bounds and normalized
endpoint limits also hold for $W_1^\ell(x;z)$ and $W_2^r(x;z)$ as $z$
approaches the corresponding upper endpoint within
$\C^+\setminus\Sigma^{\ell,+}$ and
$\C^+\setminus\Sigma^{r,+}$, respectively.

 If, in addition, $u_0-u_e^\ell\in L^{1,2}(\R^-)$ and
$u_0-u_e^r\in L^{1,2}(\R^+)$, then the normalized boundary
values in \eqref{e:W.endpoints} are locally H\"older continuous
with exponent $1/2$ up to the corresponding endpoints.  More precisely, for $s\in\{\ell,r\}$ and sufficiently small
$\delta>0$, let $B_\delta(p):=\{z\in\C:|z-p|<\delta\}$. Then
\begin{equation}\label{e:jostendpointholder}
	\begin{aligned}
		\left|
		\frac{W^s(x;z_\pm)}{\gamma^s(z_\pm)}
		-
		\frac{W^s(x;w_\pm)}{\gamma^s(w_\pm)}
		\right|
		&\le C|z-w|^{1/2},
		&& z,w\in\Sigma^{s,+}\cap\overline{B_\delta(z_1^s)},
		\\
		\left|
		\gamma^s(z_\pm)W^s(x;z_\pm)
		-
		\gamma^s(w_\pm)W^s(x;w_\pm)
		\right|
		&\le C|z-w|^{1/2},
		&& z,w\in\Sigma^{s,+}\cap\overline{B_\delta(z_2^s)}.
	\end{aligned}
\end{equation}
The boundary values at $z$ and $w$ are taken from the same
side of $\Sigma^{s,+}$, and the normalized expressions at
the endpoints are understood through their limits in \eqref{e:W.endpoints}.
The constant $C$ is independent of $z$ and $w$.
The corresponding statements at the conjugate endpoints
follow by Schwarz symmetry.

\end{enumerate}
\end{proposition}

\noindent\textbf{Symmetry of solutions and continuous spectral data}\\
For $z\in\mathbb{R}\cup(\Sigma^r\cap\Sigma^\ell)$, $W^\ell(z)$  and  $W^r(z)$ are both fundamental  matrix solutions of the same $2\times2 $ first-order matrix equation \eqref{e:zs} and therefore are linearly dependent. Therefore, there exist unimodular matrices such that  
\bse
\begin{gather}
W^\ell(x; z)=W^r(x; z) S(z),\quad z \in \mathbb{R}, \\
\label{e:S.Sigma}
W^\ell(x ; z_\pm)=W^r(x ; z_\pm) S(z_\pm),\quad z\in\Sigma^r\cap\Sigma^\ell,
\end{gather}
\ese
where the subscript $+$ (resp. $-$) denotes the limiting value on the positive side (resp. negative side) of the contour.

Because the focusing ZS  equation has the Schwarz symmetry $\sigma_{2} X^*(z)\sigma_{2}=X(z)$, the same symmetry must be satisfied by the solution, namely %olds true for the finite gap solution $W_0(x;z)$, the solutions $W^s(x;z)$ of perturbed problem and the scattering matrix $S(z)$:
\begin{align}
\sigma_{2} W^{*}_e(x; z) \sigma_{2}=W_e(x; z),\quad \sigma_{2} W^{s,*}(x; z) \sigma_{2}=W^s(x; z), \quad \sigma_{2}S^{*}(z) \sigma_{2}=S(z),
\end{align}
where $\sigma_{2}=\begin{psmallmatrix} 0& -\i \\ \i & \phantom{-}0\end{psmallmatrix}$ is the second Pauli matrix. Thus, the scattering matrix $S(z)$ has the following form: 
\begin{subequations}\label{e:S}
\begin{align}
S(z)&=\begin{pmatrix}a(z) & -b^*(z) \\ b(z) & a^*(z)\end{pmatrix}, ~~\qquad z \in \mathbb{R},\\
S(z_{\pm})&=\begin{pmatrix}a(z_{\pm}) & -b_2(z_{\pm}) \\ b_1(z_{\pm}) & a^*(z_{\pm})\end{pmatrix}, \quad z_{\pm} \in \Sigma^{r,+}\cap\Sigma^{\ell,+}.
\end{align}
\end{subequations}

\begin{remark}\label{rmk:windinga}
In the overlapping case, we fix $x_0^r$ and choose a
representative of $x_0^\ell$ such that the winding number
of $a(z)$ around $\Sigma^+$ is zero.
Indeed, suppose that the winding number corresponding
to the original representative $x_0^\ell$ is $n\in\Z$.
We replace $x_0^\ell$ by $x_0^\ell+\widetilde{x}_0^\ell$,
where
$\widetilde{x}_0^\ell=-\frac{2\pi n}{\Omega_1^\ell}$
denotes the increment of $x_0^\ell$. By the definition of $\Omega^\ell$ in \eqref{Omega}, this gives
\[
\Omega^\ell_{\mathrm{new}}
=\Omega^\ell-\Omega_1^\ell\widetilde{x}_0^\ell
=\Omega^\ell+2\pi n.
\]
Since $\theta_3(w+n)=\theta_3(w)$ for $n\in\Z$,
the elliptic background $u_e^\ell$ remains unchanged.
The right background is also unchanged, since $x_0^r$
is fixed. Thus the same initial potential $u_0$ is retained.

In the disjoint case, we choose representatives of both
$x_0^\ell$ and $x_0^r$ so that the winding number of $a$
around each of $\Sigma^{\ell,+}$ and $\Sigma^{r,+}$ is zero.
The same periodicity argument shows that these choices
leave both elliptic backgrounds unchanged.
\end{remark}

On $\Sigma^r\setminus\Sigma^\ell$, a matrix scattering relation does not exist because only the first column $W_1^\ell(x;z)$ (resp.\ $W^\ell_2(z)$) extends analytically to $\C^+ \setminus \Sigma^{\ell,+}$ (resp.\ $\C^-\backslash\Sigma^{\ell,-}$ ).
The analytic columns satisfy the scattering relations:
\bse
\begin{gather}
\label{31}
W_1^\ell(z)=a(z) W_1^r(z)+b_1(z)W_2^r(z)\,, \qquad z\in \Sigma^{r,+}\setminus\Sigma^{\ell,+},\\
\label{32}
W^\ell_2(z)=-b_1^*(z) W^r_1(z)+a^*(z)W^r_2(z)\,,\qquad z\in \Sigma^{r,-}\setminus\Sigma^{\ell,-}\,.
\end{gather}
\ese 
Similarly,  only  $W^r_2(z)$ (resp.\ $W^r_1(z)$)
has an analytic extension to $\C^+\backslash\Sigma^{r,+}$  (resp.\ $\C^-\backslash\Sigma^{r,-}$). On $\Sigma^\ell\setminus\Sigma^r$ these columns satisfy:
\bse
\begin{gather}
\label{33}
W_2^r(z)=b_2(z)W_1^\ell(z)+a(z) W_2^\ell(z) \,,\qquad z\in \Sigma^{\ell,+}\setminus\Sigma^{r,+},\\
\label{34}
W^r_1(z)=a^*(z)W^\ell_1(z)-b_2^*(z)W_2^\ell(z)\,,\qquad z\in \Sigma^{\ell,-}\setminus\Sigma^{r,-}\,.
\end{gather}
\ese

\begin{theorem}\label{theorem1}
    Consider initial data $u_0$ of the type \eqref{initial_data} such that $u_0-u_e^\ell \in L^1(\Real^-)$, $u_0-u_e^r\in L^1(\Real^+)$, and let $a(z)$, $b(z)$, $b_1(z)$ and $b_2(z)$ be the scattering data in \eqref{e:S_intro}. Then
    \begin{enumerate}
        \item The scattering coefficients can be expressed in terms of $W^s(x;z)$ as
 \begin{equation} \label{e:scoefs}
      \begin{aligned}
         a(z)&=\det[W_{1}^\ell(x;z), W_{2}^r(x;z)], \;z\in\R,  \qquad~~~ b(z)=\det[W^r_1(x;z),W^\ell_1(x;z)]\,,\;\;z\in\R\\
b_1(z)&=\det[W^r_1(x;z),W^\ell_1(x;z)],\;z\in\Sigma^{r,+}, \qquad b_2(z)=\det[W^r_2(x;z),W^\ell_2(x;z)],\;\;z\in\Sigma^{\ell,+}\,.
        \end{aligned}
        \end{equation}
   {  The functions  $a(z)$ and $b(z) $  are  continuous on $\R$. The function $a(z) $ extends analytically to $\Complex^+\setminus(\Sigma^{\ell,+}\cup\Sigma^{r,+})$ with continuous boundary values $a(z_\pm)$ for $z_\pm\in\mathrm{int}\,(\Sigma^{\ell,+}\cup\Sigma^{r,+})$.      
Moreover, $a(z)$ has at worst quartic root singularities at $\{z_1^\ell, z_2^\ell, z_1^r, z_2^r\}$. 
        The functions  $b_1(z)$ and $b_2(z)$ are continuous for   $z_\pm\in\mathrm{int}\,(\Sigma^{r,+})$  and $z_\pm\in\mathrm{int}\,(\Sigma^{\ell,+})$  respectively   and  have at worst quartic root singularities at $\{z_1^s, z_2^s\}$, $s\in\{r, \ell\}$ respectively.
         }
         \item If   $u_0-u_e^\ell \in \mathcal{W}^{1,1}(\Real^-)$ and $u_0-u_e^r \in \mathcal{W}^{1,1}(\Real^+)$, then for $z\in  \overline{\C^+}$,
where $\overline{\C^+}$ is the closure of  $\C^+$,                
 \be\label{e:scat.asympt}
    a(z)\e^{-\mathrm{i} (x_0^\ell-x_0^r)z}=1+\mathcal{O}(z^{-1}), \qquad |z|\to\infty.
         \ee
         Moreover, if 
         $u_0-u_e^\ell\in\mathcal{W}^{4,1}(\Real^-)$ and
         $u_0-u_e^r\in\mathcal{W}^{4,1}(\Real^+)$, then
         for $z\in\Real$,
         \begin{align}\label{b+asympt}
         	b(z)\e^{\mathrm{i}(x_0^\ell-x_0^r)z}
         	=\mathcal{O}(z^{-4}),\qquad |z|\to\infty.
         \end{align}

         \end{enumerate}
        \end{theorem}

\begin{proof}
Applying Cramer's rule to \eqref{e:S} 
gives the determinantal formulae \eqref{e:scoefs} for the scattering coefficients.  The analytic properties of $a(z)$, $b(z)$, $b_1(z)$ and $b_2(z)$ are inherited from those of $W^\ell(z)$ and $W^r(z)$ established in Proposition \ref{propm}. 
 To prove \eqref{e:scat.asympt}, we evaluate the determinant
 in \eqref{e:scoefs} at $x=0$. Under the $\mathcal W^{1,1}$
 assumption, the first-iterate calculation in
 Appendix~\ref{a:propm}, with one integration by parts, gives
 $\mathcal K e_1=\mathcal O(|z|^{-1})$ in $L^\infty(\R^-)$.
 The uniform Volterra inverse bound then gives
 \[
 m_1^\ell-e_1=(I-\mathcal K)^{-1}\mathcal K e_1 =\mathcal O(|z|^{-1}).
 \]
 The same argument applies to the second column on the right
 half-line. Using \eqref{e:defm} at $x=0$, we obtain
 \[
 W_1^\ell(0;z)\e^{-\i x_0^\ell z}=e_1+\mathcal O(|z|^{-1}),\qquad
 W_2^r(0;z)\e^{\i x_0^r z}=e_2+\mathcal O(|z|^{-1}),
 \]
 uniformly as $|z|\to\infty$ in $\overline{\C^+}$. Hence
 \[
 a(z)\e^{-\i(x_0^\ell-x_0^r)z}=\det\!\left[W_1^\ell(0;z)\e^{-\i x_0^\ell z},
 W_2^r(0;z)\e^{\i x_0^r z} \right]=1+\mathcal O(|z|^{-1}).
 \]

The asymptotic bound for $b(z)$, \eqref{b+asympt}, is proved in Appendix \ref{s:asymptoticsb}. 
\end{proof}

The next proposition gives further insight into the  algebraic relations of the scattering coefficients.
\begin{proposition}\label{S}
    Let $u_0-u_e^\ell \in L^1(\Real^-)$, $u_0 -u_e^r \in L^1(\Real^+)$, and $a(z)$, $b(z)$, $b_1(z)$ and $b_2(z)$ be the scattering data in \eqref{e:S}, then
    \begin{enumerate}
      
        \item For $z \in \Sigma^r\cap \Sigma^\ell$, the scattering data satisfy the jump relations
    \begin{equation}\label{jumpS}
        \begin{bmatrix}
           a(z_+) & -b_2(z_+) \\ b_1(z_+) & a^*(z_+)
        \end{bmatrix}
        =
        \begin{bmatrix}
           a^*(z_-)\e^{-2\mathrm{i}(x_0^\ell E^\ell - x_0^r E^r)} & 
           b_1(z_-)\e^{2\mathrm{i}(x_0^\ell E^\ell + x_0^r E^r)} \\ 
           -b_2(z_-)\e^{-2\mathrm{i}(x_0^\ell E^\ell + x_0^r E^r)} & 
           a(z_-)\e^{2\mathrm{i}(x_0^\ell E^\ell - x_0^r E^r)}
        \end{bmatrix}.
    \end{equation}
        \item For $z\in\Sigma_{\mathrm{tot}}\setminus(\Sigma^r\cap\Sigma^\ell)$, the scattering data satisfy the jump relations
     \bse\label{a+symmetry}
         \begin{align}
        &a(z_\pm)=\mp \mathrm{i} \e^{-2 \mathrm{i} x_0^\ell E^\ell} b_2(z_\mp), 
        &&z\in \Sigma^{\ell,+}\setminus\Sigma^{r,+},
        &a(z_\pm)=\mp\mathrm{i} \e^{2\mathrm{i} x_0^rE^r} b_1(z_\mp), 
        &&z\in\Sigma^{r,+}\setminus\Sigma^{\ell,+}, \\
        &a^*(z_\pm)=\pm \mathrm{i} \e^{2\mathrm{i} x_0^\ell E^\ell}b_2^*(z_\mp),
        &&z\in \Sigma^{\ell,-}\setminus\Sigma^{r,-}\,,
        &a^*(z_\pm)=\pm \mathrm{i} \e^{-2 \mathrm{i} x_0^rE^r} b_1^*(z_\mp),
        &&z\in \Sigma^{r,-}\setminus\Sigma^{\ell,-}\,.
        \end{align}
        \ese
        The above facts imply the following relations, which we record here for later use:
      \bse
      \label{e:ajump}
        \begin{align}
        	&\frac{a(z_+)}{a(z_-)}=	-\frac{b_2(z_-)}{b_2(z_+)}, &&z\in \Sigma^{\ell,+}\setminus\Sigma^{r,+}\,,
            &\frac{a(z_+)}{a(z_-)}=	-\frac{b_1(z_-)}{b_1(z_+)}, &&z\in \Sigma^{r,+}\setminus\Sigma^{\ell,+}\,, 
        	\\
            &\frac{a^*(z_+)}{a^*(z_-)}=-\frac{b_2^*(z_-)}{b_2^*(z_+)}, 
            && z\in\Sigma^{\ell,-}\setminus\Sigma^{r,-}\,,
        	&\frac{a^*(z_+)}{a^*(z_-)}=	-\frac{b_1^*(z_-)}{b_1^*(z_+)}, 
            &&z\in \Sigma^{r,-}\setminus\Sigma^{\ell,-}\,.
        	% \\
        	% |a(z)|^2=\frac{1}{1-|r|^2}, \quad z\in \R\,.
        \end{align}
      \ese  

    \item When $z_1^r$ and/or $z_2^r$ are not in $\Sigma^{\ell,+}$ then
\bse \label{e:b.at.endpoints}
    \begin{align}
    	 &\lim_{\substack{z\to z_1^r \\ z\in \Sigma^{r,+} \setminus \Sigma^{\ell,+}}} \frac{b_1(z_-)}{b_1(z_+)} = - \i, 
	 && \lim_{\substack{z\to z_2^r \\ z\in \Sigma^{r,+} \setminus \Sigma^{\ell,+}}} \frac{b_1(z_-)}{b_1(z_+)} = \i,
\shortintertext{and when $z_1^\ell$ and/or $z_2^\ell$ are not in $\Sigma^{r,+}$ then}
	 &\lim_{\substack{z\to z_1^\ell \\ z\in \Sigma^{\ell,+} \setminus \Sigma^{r,+}}} \frac{b_2(z_-)}{b_2(z_+)} 
	 = -\i,  
	 &&\lim_{\substack{z\to z_2^\ell \\ z\in \Sigma^{\ell,+} \setminus \Sigma^{r,+}}} \frac{b_2(z_-)}{b_2(z_+)} 
	 = \i.   
   \end{align}
\ese
    
    \end{enumerate}
\end{proposition}
\begin{proof}
For $z \in \Sigma^\ell \cap \Sigma^r$, we have the scattering relation \eqref{e:S.Sigma}. Applying the jump relations \eqref{jumpW} to that relation gives :
\begin{align}
\begin{split}
W^\ell(x ; z_{+})&=W^\ell(x ; z_{-})(\i \e^{2\i x_0^\ell E^\ell\sigma_3}\sigma_{1})=W^r(x;z_{-}) S(z_{-})(\i \e^{2\i x_0^\ell E^\ell\sigma_3}\sigma_{1})\\
&=W^r(x; z_{+}) (-\i \sigma_1\e^{-2\i x_0^r E^r \sigma_3}) S(z_{-})  (\i \e^{2\i x_0^\ell E^\ell\sigma_3}\sigma_{1})\,,
\end{split}
\end{align}
which, recalling \eqref{e:S}, is equivalent to \eqref{jumpS}.
For $z\in\Sigma_\mathrm{tot}\setminus(\Sigma^r\cap\Sigma^\ell)$,  the full scattering matrix is not available, nevertheless, we can still derive relations among the scattering data.  Using \eqref{e:scoefs}, we obtain
\bse
\begin{align*}
&a(z_+)=\det[W_1^\ell(x;z_+), W_2^r(x;z_+)]=\det[W_1^\ell(x;z_-), \i \e^{2\i x_0^r E^r} W_1^r(x;z_-)]=-\i \e^{2\i x_0^r E^r}  b_1(z_-)\,,&& z\in\Sigma^{r,+}\setminus\Sigma^{\ell,+},\\
&a^*(z_+)=\det[\ov{W_1^\ell(x;\ov{z_+})},\ov{W_2^r(x;\ov{z_+})}]=\det[ \ov{W_1^\ell(x;\ov{z_-})}, -\i \e^{-2\i x_0^r E^r} \ov{W_1^r(x;\ov{z_-})}]=\i \e^{-2\i x_0^r E^r}  b_1^*(z_-)\,,&& z\in\Sigma^{r,-}\setminus\Sigma^{\ell,-}\,,\\
&a(z_+)=\det[W_1^\ell(x;z_+), W_2^r(x;z_+)]=\det[\i \e^{-2\i x_0^\ell E^\ell} W_2^\ell(x;z_-),  W_2^r(x;z_-)]=-\i \e^{-2\i x_0^\ell E^\ell}  b_2(z_-)\,, && z\in\Sigma^{\ell,+}\setminus\Sigma^{r,+},\\
&a^*(z_+)=\det[\ov{W_1^\ell(x;\ov{z_+})},\ov{W_2^r(x;\ov{z_+})}]=\det[-\i \e^{2\i x_0^\ell E^\ell} \ov{W_2^\ell(x;\ov{z_-})},  \ov{W_2^r(x;\ov{z_+})}]=\i \e^{2\i x_0^\ell E^\ell} b_2^*(z_-)\,,&& z\in\Sigma^{\ell,-}\setminus\Sigma^{r,-}\,.
\end{align*}
\ese
This completes the proof of \eqref{a+symmetry} from which \eqref{e:ajump} follows immediately. 
 Finally, \eqref{e:b.at.endpoints} follows from \eqref{e:W.endpoints}. For example if $z_2^\ell \in \Sigma^{\ell,+} \setminus \Sigma^{r,+}$, then $b_2(z_-)/b_2(z_+)$ is continuous with a well-defined limit as $z \to z_2^\ell$ since the ratio eliminates the singular growth of $b_2(z)$ at $z_2^\ell$:
\be \label{e:b2jump}
\begin{aligned}
 \lim_{\substack{s\to z_2^\ell\\ s\in \Sigma^{\ell,+}\setminus\Sigma^{r,+}}}
\frac{b_2(s_-)}{b_2(s_+)}
&=
\lim_{\substack{s\to z_2^\ell\\ s\in \Sigma^{\ell,+}\setminus\Sigma^{r,+}}}
\frac{\gamma^\ell(s_-)^{-1}}{\,\gamma^\ell(s_+)^{-1}}
\frac{\det\!\bigl[\,W_2^r(x;s),\,\gamma^\ell(s_-)W_2^\ell(x;s_-)\,\bigr]}{\det\!\bigl[\,W_2^r(x;s),\,\gamma^\ell(s_+)W_2^\ell(x;s_+)\,\bigr]}\\
&=\i\,
\frac{\det\!\bigl[\,W_2^r(x;z_2^\ell),\,\e^{\i x E^\ell\sigma_3}\widehat{W}_2^\ell(x;z_2^\ell)\,\bigr]}
{\det\!\bigl[\,W_2^r(x;z_2^\ell),\,\e^{\i xE^\ell\sigma_3}\widehat{W}_2^\ell(x;z_2^\ell)\,\bigr]}
=\i.
\end{aligned}
\ee
The other cases are proved similarly. 
\end{proof}

To set up a Riemann--Hilbert problem, we introduce the sectionally meromorphic matrix $M(z;x)$:
\be \label{WB}
M(z;x)=
\begin{cases}
\left[\frac{W_{1}^\ell(x; z)}{a(z)}, \;W_{2}^r(x;z)\right] \e^{\i (x-x_0^r)z\sigma_3}, & z\in\Complex^+\setminus (\Sigma^{r,+}\cup\Sigma^{\ell,+})\,,\vspace{5pt}\\
\left[W^r_1(x;z), \frac{W^\ell_2(x;z)}{a^*(z)}\right] \e^{\i(x-x_0^r)z\sigma_3}, & z\in\Complex^-\setminus(\Sigma^{r,-}\cup\Sigma^{\ell,-})\, ,
\end{cases}
\ee
and define the reflection coefficients 
\begin{equation}\label{eq:reflections}
   R(z) = \frac{b(z)}{a(z)},  \quad z \in \R;	\qquad \qquad 
   R_1(z) := \frac{\ii b_1(z) }{a(z)} \e^{2\ii x_0^r E^r}, \quad z \in \Sigma^{r,+}; \qquad \qquad  
   R_2(z) := \frac{ b_2(z)}{a(z) }, \quad z \in \Sigma^{\ell,+}.
\end{equation}
Then $M(z;x)$ satisfies the following Riemann--Hilbert problem.
\begin{RHP}\label{RHP:M}
 Find a $2\times2$ matrix-valued function $M(z;x)$ which satisfies the following conditions:
 \begin{enumerate}
 	\item $M(z;x)$ is analytic  in $\Complex\setminus (\Real\cup\Sigma^r\cup\Sigma^\ell)$.
    \item $M(z;x)=I+\mathcal{O}(z^{-1})$, as $z\to\infty$.
    \item\label{jump_eqM}  For  $z\in \Real\cup\Sigma^r\cup\Sigma^\ell$  $M(z_+;x)=M(z_-;x)J(z;x)$  where
    % JUMP  EXPRESSED IN TERMS OF REFLECTION COEFFICIENTS R_1 and R_2 (doesn't quite work)
    \be\label{jumpM}
   J(z;x)= \begin{cases}
    	\begin{pmatrix}
    		R_1(z_+) & \ii \e^{2\ii x_0^r E^r} \e^{-2\ii(x-x_0^r) z} \\
    		\ii (1 + R_1(z_+) R_1(z_-)) \e^{-2\ii x_0^r E^r} \e^{2\ii(x-x_0^r) z} & - R_1(z_-)
    	\end{pmatrix}, & z\in\Sigma^{r,+}\cap\Sigma^{\ell,+},
    	\vspace{6pt}
    	\\
    	\begin{pmatrix}
    		1 & 0 \\
    		\frac{1}{R_2(z_+) a(z_+)^2}  \e^{2\mathrm{i}(x-x_0^r)z } & 1
    	\end{pmatrix}, & z\in\Sigma^{\ell,+}\setminus \Sigma^{r,+},
    	\vspace{6pt}
    	\\
    	\begin{pmatrix}
    		 R_1(z_+)  & \ii \e^{2\ii x_0^rE^r} \e^{-2\mathrm{i}(x-x_0^r)z}  \\
    		0 & \frac{1}{R_1(z_+)}
    	\end{pmatrix}, & z\in\Sigma^{r,+}\setminus\Sigma^{\ell,+},
    	\vspace{6pt}
    	\\
    	\begin{pmatrix}
    		1+|R(z)|^2 & R^*(z) \e^{-2\mathrm{i}(x-x_0^r)z}\\
    		R(z) \e^{2\mathrm{i}(x-x_0^r)z} &1
    	\end{pmatrix}, & z\in \Real,
    	\vspace{6pt}
    	\\
	\sigma_2 \overline{J(\, \overline z;x,t)} \sigma_2 & z \in \Sigma^{-}\,.
    \end{cases}
  \ee

     \item 
     $M(z;x)$ satisfies Schwarz symmetry:
     \be
     M( z;x)=\sigma_{2} \overline{M(\,\overline{z};x)} \sigma_{2}.
     \ee
     \item 
     $M(z;x)$ admits  at most  quartic root   singularities at $z\in\partial\Sigma^\ell\cup\partial\Sigma^r$.
\end{enumerate}
\end{RHP}

From the definition \eqref{WB}  the only non obvious condition  of the above Riemann-Hilbert problem is the jump \eqref{jumpM} that we promptly prove.
When $z\in \Sigma^{r,+}\cap\Sigma^{\ell,+}$, the scattering relation \eqref{sr}, the jump condition of $W^s(x;z)$ \eqref{jumpW}, and the boundary relations \eqref{jumpS} all apply. Using these, we have  for the first and second columns of $M$ 
\begin{align*}
M_1(z_+;x)
&=\frac{W_1^\ell(x;z_+)}{a(z_+)}\e^{\i(x-x_0^r)z} 
=\frac{\i W_2^\ell(x;z_-)} {a(z_+)}\e^{\i(x-x_0^r)z-2\i x_0^\ell E^\ell} \\
&=\left(-\frac{\i b_2(z_-)\e^{-2\i x_0^\ell E^\ell}}{a(z_+)} \frac{W_1^\ell(x;z_-)}{a(z_-)}+\frac{\i\e^{-2\i x_0^\ell E^\ell}}{a(z_+)a(z_-)}W_2^r(x;z_-)\right)\e^{\i(x-x_0^r)z} \\
&= {\frac{\ii b_1(z_+)}{a(z_+)}\e^{2\ii x_0^r E^r}M_{1}(z_-;x) +  \frac{\ii \e^{-2\i x_0^r E^r} }{a(z_+)a^*(z_+)} \e^{2\ii(x-x_0^r)z} M_2(z_-; x)\, ,} \\
&= {R_1(z_+) M_{1}(z_-;x) + \i \left( 1 + \frac{ b_1(z_+) b_2(z_+)}{a(z_+) a^*(z_+)} \right) \e^{2\ii ((x-x_0^r)z-x_0^r E^r)}M_{2}(z_-;x)\, ,} \\
&={R_1(z_+) M_{1}(z_-;x) + \i \left( 1 - \frac{ b_1(z_+) b_1(z_-)}{a(z_+) a(z_-)} \e^{4\i x_0^r E^r} \right) \e^{2\i ((x-x_0^r)z-x_0^r E^r)} M_{2}(z_-;x) } \\
&= {R_1(z_+) M_{1}(z_-;x) + \i \left( 1 + R_1(z_+) R_1(z_-) \right)\e^{2\i\left((x-x_0^r)z-x_0^rE^r\right)}M_{2}(z_-;x)}\, , \\
M_2(z_+;x) 
 &=W_2^r(x;z_+)\e^{-\i(x-x_0^r)z} 
=\i  W_1^r(x;z_-)\e^{-\i(x-x_0^r)z+2\i x_0^r E^r} \\
&=\left(\frac{\i W_1^\ell(x;z_-)}{a(z_-)}-\frac{\i b_1(z_-) W_2^r(x;z_-)}{a(z_-)}\right)\e^{-\i(x-x_0^r)z+2\i x_0^r E^r}=\i \e^{-2\i(x-x_0^r)z+2\i x_0^r E^r}M_1(z_-;x)- { R_1(z_-)} M_2(z_-;x),
\end{align*}
which is equivalent to the jump \eqref{jumpM} on $\Sigma^{\ell,+} \cap  \Sigma^{r,+}$.

On $\Sigma^{\ell,+}\setminus \Sigma^{r,+}$, 
plugging \eqref{W2+} into the definition of $M(z;x)$ in \eqref{WB}, we can rewrite $M_1(z_+;x)$ in terms of $M_2(z_-;x)$ and $M_1(z_-;x)$:
\begin{align*}
M_1(z_+;x)=\frac{W_1^\ell(x;z_+)}{a(z_+)}\e^{\i(x-x_0^r)z}
=\frac{\i \e^{2\i (x-x_0^r)z-2\i x_0^\ell E^\ell}}{a(z_-)a(z_+)}M_2(z_-;x)-\frac{\i b_2(z_-)}{a(z_+)}\e^{-2\i x_0^\ell E^\ell}M_1(z_-;x),
\end{align*}
using the symmetry \eqref{a+symmetry}, we then have the jump in \eqref{jumpM}.
Analogously, on $z\in\Sigma^{r,+}\setminus\Sigma^{\ell,+}$, from \eqref{e:W1-}, we have
\begin{align*}
\begin{split}
M_1(z_+;x)
  &=\frac{W_1^\ell(x;z_+) \e^{\i(x-x_0^r) z}}{a(z_+)} 
   {
    = \frac{i b_1(z_+) e^{2\i x_0^r E^r}}{a(z_-)} \frac{W_1^\ell(x;z_-)}{a(z_+)}\e^{\i(x-x_0^r)z} 
    = R_1(z_+) M_1(z_-;x)} ;\\
M_2(z_+;x)
%=W_2^r(x;z_+)\e^{-\i(x-x^r)z_+}
&=\i W_1^r(x;z_-)\e^{-\i(x-x_0^r)z+2\i x_0^r E^r}=\i \left[\frac{W_1^\ell(x;z_-)}{a(z_-)} -\frac{b_1(z_-)}{a(z_-)}W_2^r(x;z_-)\right]\e^{-\i(x-x_0^r)z+2\i x_0^r E^r} \\
&=\i M_1(z_-;x)\e^{-2\i(x-x_0^r)z+2\i x_0^r E^r} +{ \frac{ a(z_+) }{\i b_1(z_+)\e^{2\i x_0^r E^r}} } M_2(z_-;x)\,, \\
&={\i M_1(z_-;x)\e^{-2\i(x-x_0^r)z+2\i x_0^r E^r} +\frac{1}{R_1(z_+)}  M_2(z_-;x)\,, } 
\end{split}
\end{align*}
which completes the verification.
%\end{proof}
\medskip

\subsection{A symmetric Riemann-Hilbert problem  \label{Symm_RHP}} 
The Riemann--Hilbert problem \ref{RHP:M} is sufficient to pose an inverse problem, but it is not optimal because  one needs to know the values of several independent scattering functions along the continuous spectrum in order to pose the  inverse problem. By introducing a different normalization of the problem we can simultaneously reduce the number of scattering functions needed to pose the inverse problem and symmetrize the jump functions resulting in a simpler Riemann-Hilbert problem.
% To accomplish this renormalization it is necessary to introduce a factorization of the transmission coefficient $a(z)^{-1}$. 
The first step  is to introduce a new representation for the analytic continuation of  $a(z)$ to $\C^+ \setminus \Sigma^{+}$ in terms of the other scattering coefficients  $b(z)$, $b_1(z)$ and $b_2(z)$ defined on  the continuous spectrum.  

To state the result we define the function
\begin{equation}
\label{Lambda}
\Lambda(z) =  \sqrt{ \ \ \quad\prod_{\mathclap{\,\,\,\,p \in \partial( \Sigma^{\ell,+} \cap  \Sigma^{r,+}) }} (z-p)(z-\overline p) } , \qquad z \in \C \setminus (\Sigma^r \cap \Sigma^\ell) 
\end{equation}
analytic for $ z \in \C \setminus (\Sigma^r \cap \Sigma^\ell) 
$ and normalized such that %$w(z) = 1 + \bigo{z^{-1}}$ and 
$z^{-2} \Lambda(z) = 1 + \bigo{z^{-1}}$ as $z \to \infty$. Observe also that $\Lambda^*(z) = \Lambda(z)$. % while $w^*(z) = w(z)^{-1}$. 
To simplify the presentation   we also  introduce the following function on the continuous spectrum:
\begin{align}
\label{nu}
	\nu(z)=\begin{cases}
		\log\left(-\frac{b_2(z_-)}{b_2(z_+)}\right), & z\in\Sigma^{\ell,+}\setminus\Sigma^{r,+},\\
		\log\left(-\frac{b_1(z_-)}{b_1(z_+)}\right), & z\in\Sigma^{r,+}\setminus\Sigma^{\ell,+},\\	
		\log\left(1-b_1(z_+)b_2(z_+)\right), & z\in\Sigma^{\ell,+}\cap\Sigma^{r,+},\\
		\log\left(1-|b(z)|^2\right), & z\in\R,\\
		-\nu^*(z), & z\in \Sigma^{-}
	\end{cases}
\end{align}
that is well defined thanks to the assumptions on $a(z)$ to be nonzero and Proposition~\ref{S}.

\begin{proposition}\label{prop:a.representation}
Let $a(z)$ be the scattering coefficient defined by \eqref{e:scoefs} for a potential $u_0$ satisfying the conditions of Theorem~\ref{theorem1}. 
Then for each $z \in \C^+ \setminus\Sigma^{+}$ %\Sigma^\mathrm{tot}_+
\begin{itemize}
\item If $\ \Sigma^r \cap \Sigma^\ell = \emptyset$, then 
  \begin{align}\label{a.rep.disjoint}
	a(z) = \exp \left[ \frac{1}{2\pi \i } \left( 
	\int_{\Sigma_{tot}\cup\R} \frac{\nu(\xi)}{\xi-z}\d \xi
	\right) \right] \e^{\i (x_0^\ell - x_0^r)z}, \quad z \in \C^+ \setminus\Sigma^+,
\end{align}
with $\nu$ defined in \eqref{nu} and $x_0^r$ and $x_0^\ell$ the shifts of the initial data $u_e^r$ and $u_e^{\ell}$  as in \eqref{e:qpdef}.
  \item If $\ \Sigma^r \cap \Sigma^\ell \neq \emptyset$, then 
\begin{align}\label{a.rep.overlap}
a(z) =\exp \Bigg[\frac{\Lambda(z)}{2\pi \i}\Bigg(
\int_{\Sigma^{\ell}\cap \Sigma^{r}} \frac{\nu(\xi)}{\Lambda(\xi_+)} \frac{\d \xi}{\xi-z}
+
\int_{\Sigma_{tot}\cup\R
\setminus\left( \Sigma^{\ell}\cap \Sigma^{r}\right)} \frac{\nu(\xi)}{\Lambda(\xi)} \frac{\d \xi}{\xi-z}
\Bigg)\Bigg]
\e^{-\i(x_0^\ell E^\ell-x_0^r E^r)}, 
\quad  
z \in \C^+ \setminus\Sigma^+,
\end{align}
with $\Lambda$ defined in \eqref{Lambda}.
\end{itemize}
\end{proposition}
%\todo{We need to put the boundary value of $\Lambda$ in the intersection ZZ: done}
\begin{proof}
	See Appendix~\ref{app:a.representation}. 
\end{proof}

To arrive at our simplified Riemann-Hilbert problem we seek a factorization of $a(z)$ with certain useful properties.

\begin{proposition}\label{prop:a.factorization}
Suppose that Assumption~\ref{assume} holds,
$u_0-u_e^\ell\in L^{1,2}(\R^-)$ and
$u_0-u_e^r\in L^{1,2}(\R^+)$.
Suppose that there exists a factorization
	\begin{equation}
	a(z) = a_1(z) a_2(z),\qquad z \in \C^+ \setminus \Sigma^+,
\end{equation}
such that the functions $a_1(z)$ and $a_2(z)$ have the following properties:

\begin{enumerate}

	\item $a_1(z)$ is analytic in $\C^+ \setminus \Sigma^{\ell,+}$, and 
	$a_2(z)$ is analytic in $\C^+ \setminus  \Sigma^{r,+}$. 
	Both functions have finite continuous boundary values on
	$\Sigma^+$ away from
	$\partial\Sigma^{\ell,+}\cup\partial\Sigma^{r,+}$.
	
		\item The boundary values satisfy the jump relations
	\begin{subequations}\label{e:jumpa12}
		\begin{align}
			&	\frac{a_1(z_+)}{a_1(z_-)} = 
			\frac{a(z_+)}{a(z_-)} = -\frac{b_2(z_-)}{b_2(z_+)}, 
			&& z \in \Sigma^{\ell,+} \setminus \Sigma^{r,+} ,\\
			&	\frac{a_2(z_+)}{a_2(z_-)} = 
			\frac{a(z_+)}{a(z_-)} = -\frac{b_1(z_-)}{b_1(z_+)},
			&& z \in \Sigma^{r,+} \setminus \Sigma^{\ell,+}, \\
			&\frac{a_1(z_+)}{a_1(z_-)}   \frac{a_2(z_-)}{a_2(z_+)} = -\frac{b_1(z_+)}{b_1(z_-)} , 
			&& z \in \Sigma^{\ell,+} \cap \Sigma^{r,+}.
		\end{align}
	\end{subequations}

	\item As $z\to\infty$,
	\begin{equation}\label{e:asym}
			a_1(z) =  \e^{\i (x_0^\ell - x_0^r) z} \left[ 1+ \bigo{z^{-1}} \right], \qquad
			a_2(z) = 1 + \bigo{z^{-1}}.
	\end{equation}

	\item Both functions have continuous boundary values on $\R$, satisfying
	\begin{equation}\label{e:a12mod}
		|a_1(z)| = |a(z)|, \quad |a_2(z)| = 1  ,  \qquad z \in \R. 
	\end{equation}
	
	\item At every
	$p\in\partial\Sigma^{\ell,+}\cup\partial\Sigma^{r,+}$,
	\[
	a_j(z)^{-1}=\mathcal O(1),\qquad j=1,2,
	\]
	as $z\to p$ from each adjacent component of
	$\C^+\setminus\Sigma^+$.
\end{enumerate}
Then the function 
\begin{equation}\label{renormalization}
	\Phi(z; x,0) = \begin{dcases}
		M(z;x) a_2(z) ^{\sigma_3}, & z \in \C^+ ,\\
		M(z;x) a_2^*(z)^{-\sigma_3}, & z \in \C^-,
	\end{dcases}
\end{equation}
satisfies RHP~\ref{RHP:Phi} at $t=0$.
\end{proposition}
\begin{proof}
	See Appendix~\ref{a:proof3.7}.
\end{proof}

To have an explicit representation of $a_1(z)$ and $a_2(z)$ satisfying  the assumptions  of Proposition~\ref{prop:a.factorization} we introduce the function
\begin{align}
\label{tildenu}
\widetilde\nu(z)=\begin{cases}
	\nu(z),
	& z\in\mathbb R\cup\left(\Sigma^{\ell,+}\setminus\Sigma^{r,+}\right),\\[2mm]
	-\nu(z),
	& z\in\Sigma^{r,+}\setminus\Sigma^{\ell,+},\\[2mm]
	\log\!\left(-\dfrac{b_1(z_+)}{b_1(z_-)}\right),
	& z\in\Sigma^{\ell,+}\cap \Sigma^{r,+},\\[3mm]
	-\widetilde\nu^*(z),
	& z\in\Sigma^{\ell,-}\cup\Sigma^{r,-}.
\end{cases}
\end{align}
\begin{proposition}\label{prop:afactor}
	Let $a(z)$ be as in Proposition~\ref{prop:a.representation}. Then the factorization of $a(z)$ introduced  in Proposition~\ref{prop:a.factorization} is achieved as follows:
	\begin{itemize}
	\item  If $\ \Sigma^r \cap \Sigma^\ell = \emptyset$, then
	\begin{align}
	a_1(z) =  \e^{\i(x_0^\ell - x_0^r) z}\exp\left[  \frac{1}{2\pi \i} \left( 
		\int_{\Sigma^\ell\cup\R}  \frac{\nu(\xi)}{\xi-z} \d \xi \right) \right]\,,\qquad 
	a_2(z) = \exp\left[ \frac{1}{2\pi \i} \left( 
	\int_{\Sigma^r}\frac{\nu(\xi)}{\xi-z}\d \xi\right) \right].
\end{align}
	
	\item  
	If $\Sigma^r\cap\Sigma^\ell\neq\emptyset$, the logarithmic
	branches in \eqref{nu} are chosen so that the full exponent
	in \eqref{a.rep.overlap}, including
	$-\i(x_0^\ell E^\ell-x_0^rE^r)$, is
	$\i(x_0^\ell-x_0^r)z+\mathcal O(z^{-1})$
	as $z\to\infty$. We define
	\bse\label{e:defa12}
	\be
\begin{aligned}
	a_1(z)=\e^{\frac{\i}{2}\left((x_0^\ell-x_0^r) z-(x_0^\ell E^\ell-x_0^rE^r)\right)}&\exp\left\{\frac{1}{4 \pi\i}\Bigg[	\int_{\Sigma^{\ell}\cap\Sigma^{r}}\left(
	\frac{\Lambda(z)}{\Lambda(\xi_+)}\nu(\xi)
	+\widetilde\nu(\xi)
	\right)\frac{\d\xi}{\xi-z}\Bigg.\right.
	\\
&\left.\left.	+
	\int_{\R\cup\Sigma_{\mathrm{tot}}\setminus \left(\Sigma^{\ell}\cap\Sigma^{r} \right)}\left(
	\frac{\Lambda(z)}{\Lambda(\xi)}\nu(\xi)
	+\widetilde\nu(\xi)
	\right)\frac{\d\xi}{\xi-z}
	\right]
	\right\}\,,
	\label{e:defa1}
	\end{aligned}
	\ee
	\be
	\begin{aligned}
	a_2(z)=\e^{-\frac{\i}{2}\left((x_0^\ell-x_0^r) z+(x_0^\ell E^\ell-x_0^rE^r)\right)}
	&\exp\left\{\frac{1}{4 \pi\i}\Bigg[	\int_{\Sigma^{\ell}\cap\Sigma^{r}}\left(
	\frac{\Lambda(z)}{\Lambda(\xi_+)}\nu(\xi)
	-\widetilde\nu(\xi)
	\right)\frac{\d\xi}{\xi-z}\Bigg.\right.
	\\
	&\left.\left.	+
	\int_{\R\cup\Sigma_{\mathrm{tot}}\setminus \left(\Sigma^{\ell}\cap\Sigma^{r} \right)}\left(
	\frac{\Lambda(z)}{\Lambda(\xi)}\nu(\xi)
	-\widetilde\nu(\xi)
	\right)\frac{\d\xi}{\xi-z}
	\right]
	\right\}\,.
	\label{e:defa2}
\end{aligned}
\ee
	\ese
	with $\Lambda$ defined in \eqref{Lambda} and $\nu$ and $\tilde{\nu}$ defined in \eqref{nu} and \eqref{tildenu} respectively.
\end{itemize}
\end{proposition}
\begin{proof}
	See Appendix~\ref{a:proof3.8}.
\end{proof}

\begin{proposition}\label{p:singularitya12}
Suppose that Assumption~\ref{assume} holds, $u_0-u_e^\ell\in L^{1,2}(\R^-)$ and $u_0-u_e^r\in L^{1,2}(\R^+)$.
Assume also that
$\partial\Sigma^{\ell,+}\cap\partial\Sigma^{r,+}=\emptyset$.
Let $a_1(z)$ and $a_2(z)$ be defined in Proposition~\ref{prop:afactor}.
Then their boundary values $a_1(z_\pm)$ and $a_2(z_\pm)$, taken from each side of the corresponding contour, are continuous on
$\Sigma^{\ell,+}\setminus \partial\Sigma^{\ell,+}$ and  $\Sigma^{r,+}\setminus \partial\Sigma^{r,+}$, respectively. Moreover, for  $j=1,2$, 
 at the endpoint $p=z_j^\ell$,  
\be
a_1(z)=
	\mathcal{O}((z-p)^{-\frac 14}), 
\quad z\to p,
\qquad\qquad
a_2(z)=
	\mathcal{O}(1),
\quad z\to p,
\ee
while at the endpoint $p=z_j^r$,
\be
a_1(z)=
	\mathcal{O}(1), 
\quad z\to p,\qquad\qquad
a_2(z)=
	\mathcal{O}((z-p)^{-\frac 14}), 
\quad z\to p.
\ee
In addition, $a_j(z)^{-1}=\mathcal O(1)$, $j=1,2$,
as $z$ approaches any upper spectral endpoint from
each adjacent component of $\C^+\setminus\Sigma^+$.
\end{proposition}
\begin{proof}
	See Appendix~\ref{a:proof3.9}.
\end{proof}

We next derive the continuity and endpoint behavior of $r_1$ and $r_2$,  defined in  \eqref{e:defr1} and \eqref{e:defr2} respectively,    using the properties of  $a_1(z)$ and $a_2(z)$   in Proposition~\ref{prop:a.factorization}  and Proposition~\ref{p:singularitya12}, and the local behavior of the scattering coefficients $a$, $b_1$, $b_2$ in Theorem~\ref{theorem1}. 
\begin{lemma}
\label{lemma3.10}
Suppose that Assumption~\ref{assume} holds, $u_0-u_e^\ell\in L^{1,2}(\R^-)$ and $u_0-u_e^r\in L^{1,2}(\R^+)$. Then $r_1(z)$ defined in \eqref{e:defr1} has  square-root zeros at $z_1^\ell$ and $z_2^\ell$, and is continuous on $\Sigma^{\ell,+}$, and $r_2(z)$ defined in \eqref{e:defr2} has square-root zeros at $z_1^r$ and $z_2^r$, and is continuous on $\Sigma^{r,+}$. 
 \end{lemma}

 \begin{proof}
 	 By Theorem~\ref{theorem1} and Proposition~\ref{p:singularitya12},
 	 all the coefficients in \eqref{e:r12} have continuous boundary
 	 values away from the spectral endpoints. 
 	 Thus we first check that the denominators do not vanish.
 	 
 	 On the overlap, the two denominators coincide by \eqref{jumpS}.
 	 The identity \eqref{eq:ov} gives
 	 \[
 	 	\bigl(a_1(z_+)a_2(z_-)+\i b_1(z_+)\e^{2\i x_0^rE^r}\bigr)
 	 	\bigl(a_1(z_+)a_2(z_-)-\i b_1(z_+)\e^{2\i x_0^rE^r}\bigr)
 	 	=
 	 	\frac{a_1(z_+)a_2(z_-)}{a_1(z_-)a_2(z_+)}
 	 	\e^{-2\i (x_0^\ell E^\ell-x_0^rE^r)}\ne0.
 	 \]
 	 Hence the common denominator is nonzero. On the non-overlapping
 	 portions, continuity follows directly from \eqref{e:r1} and
 	 \eqref{e:r2}. It remains to match the limits at the inner endpoints
 	 and determine the behavior at the endpoints of each band.

Let $p\in\partial\Sigma^{r,+}\cap\operatorname{int}\Sigma^{\ell,+}$.
Dividing the numerator and denominator in \eqref{e:defr1}
by $a_2(z_-)$ gives
\[
r_1(z)=
\frac{\e^{-2\i x_0^\ell E^\ell-2\i x_0^r z}}
{a_1(z_-)\left(a_1(z_+)
	-\i\e^{-2\i x_0^\ell E^\ell}
	\dfrac{b_2(z_-)}{a_2(z_-)}\right)}.
\]
The factors $a_1(z_\pm)$ extend continuously and nonvanishingly
through $p$ by Appendix~\ref{a:proof3.9}.
The quotient $b_2(z_-)/a_2(z_-)$ also extends continuously
through $p$. Indeed, for $p=z_2^r$, \eqref{e:scoefs} gives
\[
\frac{b_2(z_-)}{a_2(z_-)}
=
\frac{\det[\gamma^r(z_-)W_2^r(x;z_-),W_2^\ell(x;z_-)]}
{\gamma^r(z_-)a_2(z_-)}.
\]
By Proposition~\ref{propm}, \eqref{e:W.endpoints}, and
Appendix~\ref{a:proof3.9}, the numerator and denominator have
continuous limits through $p$, and the denominator limit is
nonzero. So $r_1$ extends continuously through $p$. For $p=z_1^r$, use $(\gamma^r)^{-1}$ instead.

The same denominator also determines the endpoint behavior of
$r_2$. On the overlap, \eqref{jumpS} and \eqref{e:defr2} give
\[
	r_2(z)
	=\frac{\e^{2\i x_0^r z+2\i x_0^rE^r}}
	{a_2(z_-)a_2(z_+)}
	\frac{a_1(z_+)}
	{a_1(z_+)-\i\e^{-2\i x_0^\ell E^\ell}
		\dfrac{b_2(z_-)}{a_2(z_-)}}
	=\frac{\e^{2\i x_0^r z+2\i x_0^rE^r}}
	{2a_2(z_-)a_2(z_+)}\bigl(1+o(1)\bigr)
	=\mathcal O(|z-p|^{1/2}),
\]
where the last step uses the nonzero normalized quarter-root
limits of $a_2$ in Appendix~\ref{a:proof3.9}.

For $p\in\partial\Sigma^{\ell,+}\cap\operatorname{int}\Sigma^{r,+}$,
the analogous argument using \eqref{a+symmetry} shows that $r_2$
is continuous through $p$ and $r_1$ has a square-root zero at $p$.

At the remaining endpoints, \eqref{e:r1} and \eqref{e:r2}
give the same square-root vanishing directly.
The nonzero normalized endpoint limits also show that the
leading square-root coefficients are nonzero.
Extending $r_1$ by zero at the left endpoints and $r_2$ by zero
at the right endpoints therefore gives
\[
r_1\in C(\Sigma^{\ell,+}),\qquad r_2\in C(\Sigma^{r,+}).
\]
 \end{proof}

Regarding $\rho(z)$ defined in \eqref{e:rho}, using the asymptotics of $b(z)$ \eqref{b-asym} and those of $a_1(z)$ and $a_2(z)$ in Proposition~\ref{prop:a.factorization}, we deduce the following lemma.
\begin{lemma}
\label{lemma3.11}  Suppose that  $a(z)$  is  nonzero   for $z\in \R$.  If $u_0-u_e^\ell\in \mathcal{W}^{k,1}(\R^-)$ and $u_0-u_e^r\in \mathcal{W}^{k,1}(\R^+)$, $k\geq 1$,   then $\rho\in L^{2,n}(\R)$  with $n<k-\frac{1}{2}$ and $ \rho\in L^{1,m}(\R)$  with $m<k-1$.
 \end{lemma}
 \begin{proof}
 Combining Lemma~\ref{lem:iterates}  and   section~\ref{a:asymW}  one gets that  $\rho(z)$ defined in \eqref{e:rho}  is bounded on $\R$ and  satisfies $\rho(z)=\mathcal{O}(z^{-k})$ as $z\to\infty$. This implies the statements.
 \end{proof}
 \paragraph{Proof of Theorem~\ref{theo_DS}.}The proof of theorem  \ref{theo_DS}  follows from   lemma~\ref{lemma3.10}  and lemma~\ref{lemma3.11}. The analyticity of $\rho$ for the exact step initial data, follows from the explicit formula \eqref{rho_step}.

Before turning to the time evolution, we illustrate the preceding
construction with an explicit example. 

\subsection{Example of a  pure-step } 
We choose the overlapping $\dn$ configuration
identified in Proposition~\ref{prop:verticaloverlap}
with the pure-step initial datum
\begin{equation}\label{e:stepic}
	u_0(x)=
	\begin{cases}
		u^\ell_e(x), & x\le 0,\\
		u^r_e(x), & x>0,
	\end{cases}
\end{equation}
where $u^\ell_e(x)$ and $u^r_e(x)$ are the genus-$1$ finite-gap backgrounds with branch points
$\pm \i\,\eta_1^\ell,\pm \i\,\eta_2^\ell$ and $\pm \i\,\eta_1^r,\pm \i\,\eta_2^r$, respectively,
 and $\eta_1^\ell<\eta_1^r<\eta_2^\ell<\eta_2^r$. In the present setting, $E^s=0$ and $N^s=\frac{
		(\eta_1^s)^2+(\eta_2^s)^2}2$.
We also take $x^\ell_0=x^r_0=0$ for simplicity.

By Remark~\ref{prop:dn_vertical}, the corresponding nonreal spectral bands are the vertical segments shown
in Figure~\ref{f:Sigma}(c):
\begin{equation}\label{eq:gamma-def}
	\Sigma^\ell := [\i\,\eta_1^\ell,\i\,\eta_2^\ell]\ \cup\ [-\i\,\eta_2^\ell, -\i\,\eta_1^\ell],\qquad
	\Sigma^r := [\i\,\eta_1^r,\i\,\eta_2^r]\ \cup\ [-\i\,\eta_2^r,-\i\,\eta_1^r].
\end{equation}

Define the scalar fourth roots:
\begin{equation}\label{e:sgamma}
	\gamma^r(z):=\left(\frac{(z-\i\,\eta_2^r)(z+\i\,\eta_1^r)}{(z+\i\,\eta_2^r)(z-\i\,\eta_1^r)}\right)^{1/4},\qquad
	\gamma^\ell(z):=\left(\frac{(z-\i\,\eta_2^\ell)(z+\i\,\eta_1^\ell)}{(z+\i\,\eta_2^\ell)(z-\i\,\eta_1^\ell)}\right)^{1/4},
\end{equation}
with branches chosen so that $\gamma^r(z)\to 1$ and $\gamma^\ell(z)\to 1$ as $z\to\infty$ and with branch cuts along
$\Sigma^r$ and $\Sigma^\ell$, respectively. With this normalization their boundary values satisfy the constant jump
\begin{equation}
	\gamma^s(z_+)=\i\,\gamma^s(z_-),\qquad z\in\Sigma^s,\quad s\in\{\ell,r\}.
\end{equation}
At $x=t=0$, the theta quotients in \eqref{e:Hij} equal one.
Since the Jost solutions coincide with their respective
backgrounds on the two half-lines, \eqref{defO} gives
\[
W^s(0;z)=O^s(z;0,0)
=\frac12
\begin{pmatrix}
	\gamma^s(z)+\gamma^s(z)^{-1}
	&
	\gamma^s(z)-\gamma^s(z)^{-1}
	\\
	\gamma^s(z)-\gamma^s(z)^{-1}
	&
	\gamma^s(z)+\gamma^s(z)^{-1}
\end{pmatrix},
\qquad s\in\{\ell,r\}.
\]
Matching at $x=0$ therefore yields
\[
S(z)=O^r(z;0,0)^{-1}O^\ell(z;0,0).
\]

In this case, $S(z)$ admits the explicit closed form
\[
S(z)=\frac12
\begin{pmatrix}
	\lambda(z)+\lambda(z)^{-1} & \lambda(z)-\lambda(z)^{-1}\\
	\lambda(z)-\lambda(z)^{-1} & \lambda(z)+\lambda(z)^{-1}
\end{pmatrix},
\qquad
\lambda(z)=\frac{\gamma^\ell(z)}{\gamma^r(z)}=\left(
\frac{(z-\i\eta_2^\ell)(z+\i\eta_1^\ell)(z+\i\eta_2^r)(z-\i\eta_1^r)}
{(z+\i\eta_2^\ell)(z-\i\eta_1^\ell)(z-\i\eta_2^r)(z+\i\eta_1^r)}
\right)^{1/4}.
\]
Its restrictions to the corresponding spectral contours give
\be
a(z)
=
\frac{1}{2}
\left(
\lambda(z)+\lambda(z)^{-1}
\right),
\ee
\be
b(z)
=
\frac{1}{2}
\left(
\lambda(z)-\lambda(z)^{-1}
\right),
\qquad z\in\mathbb R,
\ee
\be
b_1(z_\pm)
=
\frac{1}{2}
\left(
\lambda(z_\pm)-\lambda(z_\pm)^{-1}
\right),
\quad z\in\Sigma^{r,+};
\qquad \qquad
b_2(z_\pm)
=
-\frac{1}{2}
\left(
\lambda(z_\pm)-\lambda(z_\pm)^{-1}
\right),
\quad z\in\Sigma^{\ell,+}.
\ee

			The boundary relations are the specialization of
			Proposition~\ref{S} with $x_0^s=E^s=0$.
			The preliminary inverse problem is RHP~\ref{RHP:M},
			with the scattering coefficients given above.

		We now factorize $a(z)=a_1(z)a_2(z)$.
		In this pure-step case, we have
		\be
		\Lambda(z)
		=
		\sqrt{\bigl(z^2+(\eta_1^r)^2\bigr)
			\bigl(z^2+(\eta_2^\ell)^2\bigr)},
		\ee
		with cuts along $\Sigma^\ell\cap\Sigma^r$ and normalization
		$\Lambda(z)=z^2+\mathcal O(1)$ as $z\to\infty$.
		The normalized fourth roots satisfy
		$|\arg\gamma^s(z)|<\pi/4$ away from their cuts.
		Consequently, $\Re\lambda(z)>0$ and $\Re a(z)>0$.
		We therefore choose the analytic logarithm $\log a(z)$
		that tends to zero as $z\to\infty$.
		
		The boundary values of $\lambda$ satisfy
		\be
		\lambda(z_+)=
		\begin{cases}
			\i\lambda(z_-),
			& z\in\Sigma^{\ell,+}\setminus \Sigma^{r,+},\\
			\lambda(z_-),
			& z\in\Sigma^{\ell,+}\cap \Sigma^{r,+},\\
			-\i\lambda(z_-),
			& z\in\Sigma^{r,+}\setminus\Sigma^{\ell,+}.
		\end{cases}
		\ee
		Thus the density $\nu$  defined in \eqref{nu} takes the explicit form
		\be\label{e:ps.nu}
		\nu(z)=
		\begin{cases}
			\displaystyle
			\log\!\left(
			\i\frac{\lambda(z_-)^2-1}{\lambda(z_-)^2+1}
			\right),
			& z\in\Sigma^{\ell,+}\setminus \Sigma^{r,+},\\[3mm]
			\displaystyle
			\log\!\left(
			-\i\frac{\lambda(z_-)^2-1}{\lambda(z_-)^2+1}
			\right),
			& z\in\Sigma^{r,+}\setminus\Sigma^{\ell,+},\\[3mm]
			2\log a(z),
			& z\in\Sigma^{\ell,+}\cap \Sigma^{r,+},\\[1mm]
			\log|a(z)|^2,
			& z\in\mathbb R.
		\end{cases}
		\ee

		On the  overlap in $\C^+$, $b_1(z_+)=b_1(z_-)$.
		We fix the logarithm in the definition of $\widetilde\nu$ by
		\be\label{e:psnutilde}
		\widetilde\nu(z)=
		\begin{cases}
			\nu(z),
			& z\in\mathbb R\cup
			(\Sigma^{\ell,+}\setminus \Sigma^{r,+}),\\[1mm]
			-\nu(z),
			& z\in\Sigma^{r,+}\setminus\Sigma^{\ell,+},\\[1mm]
			-\i\pi,
			& z\in\Sigma^{\ell,+}\cap \Sigma^{r,+}.
		\end{cases}
		\ee
		The values on the lower contours are defined by Schwarz symmetry.
		
		For the pure-step setting, $a(z)$ has no jump across the overlap.
		Using the expression of $a_2(z)$ in~\eqref{e:defa2}, together
		with \eqref{e:psnutilde}, we obtain
		\be
		\log a_2(z)
		=\frac{1}{4\pi\i}
		\left(
		2\int_{\Sigma^r\setminus\Sigma^\ell}
		\nu(\xi)\frac{\d\xi}{\xi-z}
		+\i\pi\int_{\Sigma^\ell\cap\Sigma^r}
		\frac{\d\xi}{\xi-z}
		\right)
		=\frac{1}{2\pi\i}
		\int_{\Sigma^r\setminus\Sigma^\ell}
		\nu(\xi)\frac{\d\xi}{\xi-z}
		+\frac14\log
		\frac{(z-\i\eta_2^\ell)(z+\i\eta_1^r)}
		{(z-\i\eta_1^r)(z+\i\eta_2^\ell)}.
		\ee
		Here the terms on $\mathbb R\cup
		(\Sigma^\ell\setminus\Sigma^r)$ cancel, while those on
		$\Sigma^r\setminus\Sigma^\ell$ give $2\nu$.
		
		Exponentiating and using
		$a_1a_2=a=(\lambda+\lambda^{-1})/2$, we therefore obtain
		\be
		\begin{aligned}
			a_1(z)
			&=
			\frac{\lambda(z)+\lambda(z)^{-1}}{2}
			\left(
			\frac{(z-\i\eta_2^\ell)(z+\i\eta_1^r)}
			{(z-\i\eta_1^r)(z+\i\eta_2^\ell)}
			\right)^{-1/4}
			\exp\left\{
			-\frac{1}{2\pi\i}
			\int_{\Sigma^r\setminus\Sigma^\ell}
			\nu(\xi)\frac{\d\xi}{\xi-z}
			\right\},\\
			a_2(z)
			&=
			\left(
			\frac{(z-\i\eta_2^\ell)(z+\i\eta_1^r)}
			{(z-\i\eta_1^r)(z+\i\eta_2^\ell)}
			\right)^{1/4}
			\exp\left\{
			\frac{1}{2\pi\i}
			\int_{\Sigma^r\setminus\Sigma^\ell}
			\nu(\xi)\frac{\d\xi}{\xi-z}
			\right\}.
		\end{aligned}
		\ee
		The fourth roots are reciprocal branches, with cuts along
		$\Sigma^\ell\cap\Sigma^r$, normalized to tend to $1$ at infinity.

		\begin{proposition}
			For $j=1,2$, as $z\to\i\eta_j^\ell$,
			\bse\label{e:stepasingularity}
			\be
			a_1(z)=\mathcal{O}\bigl(|z-\i\eta_j^\ell|^{-1/4}\bigr),
			\qquad
			a_2(z)=\mathcal{O}(1).
			\ee
			Similarly, as $z\to\i\eta_j^r$,
			\be
			a_1(z)=\mathcal{O}(1),
			\qquad
			a_2(z)=\mathcal{O}\bigl(|z-\i\eta_j^r|^{-1/4}\bigr).
			\ee
			\ese
		\end{proposition}
		
		\begin{proof}
			On $\Sigma^{r,+}\setminus\Sigma^{\ell,+}$,
			\be
			\nu(\xi)=
			\begin{cases}
				\displaystyle \frac{\i\pi}{2}
				+\mathcal{O}\bigl(|\xi-\i\eta_2^\ell|^{1/2}\bigr),
				& \xi\to\i\eta_2^\ell,\\[1mm]
				\displaystyle -\frac{\i\pi}{2}
				+\mathcal{O}\bigl(|\xi-\i\eta_2^r|^{1/2}\bigr),
				& \xi\to\i\eta_2^r.
			\end{cases}
			\ee
			Hence, with the upward orientation,
			\be
			\frac{1}{2\pi\i}
			\int_{\Sigma^r\setminus\Sigma^\ell}
			\nu(\xi)\frac{\d\xi}{\xi-z}
			=
			\begin{cases}
				\mathcal{O}(1), & z\to\i\eta_1^s,\\[1mm]
				-\dfrac14\log(z-\i\eta_2^s)+\mathcal{O}(1),
				& z\to\i\eta_2^s,
			\end{cases}
			\qquad s=\ell,r.
			\ee
			Substitution into the expression for $a_2$ gives
			\be
			\begin{aligned}
				\log a_2(z)
				=
				\begin{cases}
					\mathcal{O}(1), & z\to\i\eta_j^\ell,\\[1mm]
					-\dfrac14\log(z-\i\eta_j^r)+\mathcal{O}(1),
					& z\to\i\eta_j^r,
				\end{cases}
				\qquad j=1,2.
			\end{aligned}
			\ee
			Moreover, the explicit expression
			$a=(\lambda+\lambda^{-1})/2$ yields
			\be
			\log a(z)
			=-\frac14\log(z-\i\eta_j^s)+\mathcal{O}(1),
			\qquad z\to\i\eta_j^s,\quad j=1,2,\quad s=\ell,r.
			\ee
			Therefore,
			\be
			\log a_1(z)=\log a(z)-\log a_2(z)
			=
			\begin{cases}
				-\dfrac14\log(z-\i\eta_j^\ell)+\mathcal{O}(1),
				& z\to\i\eta_j^\ell,\\[1mm]
				\mathcal{O}(1), & z\to\i\eta_j^r,
			\end{cases}
			\qquad j=1,2.
			\ee
			Exponentiating proves the stated estimates.
		\end{proof}

		For the pure-step data, \eqref{e:r12} reduce to
		\be\label{e:steprnonoverlap}
		r_1(z)=\frac{1}{2a_1(z_-)a_1(z_+)}, \quad
		z\in\Sigma^{\ell,+}\setminus \Sigma^{r,+},\qquad \qquad
		r_2(z)=\frac{1}{2a_2(z_-)a_2(z_+)}, \quad
		z\in\Sigma^{r,+}\setminus\Sigma^{\ell,+}.
		\ee
		On the overlap, we have
		\be
		a_1(z_+)a_2(z_-)=-\i a(z),\qquad
		b_2(z_-)=-b_1(z_+),\qquad
		a(z)-b_1(z_+)=\lambda(z)^{-1}.
		\ee
		Hence
		\be
		a_1(z_+)a_2(z_-)-\i b_2(z_-)
		=a_1(z_+)a_2(z_-)+\i b_1(z_+)
		=-\i\bigl(a(z)-b_1(z_+)\bigr)
		=-\i\lambda(z)^{-1}.
		\ee
		Substituting into \eqref{e:defr1} and \eqref{e:defr2} gives
		\be\label{e:steproverlap}
		r_1(z)=\i\lambda(z)\frac{a_2(z_-)}{a_1(z_-)},
		\qquad
		r_2(z)=\i\lambda(z)\frac{a_1(z_+)}{a_2(z_+)},
		\qquad z\in\Sigma^{\ell,+}\cap \Sigma^{r,+}.
		\ee
		On the real line, $|a_2(z)|=1$ and $a_1(z)a_2(z)=a(z)$ yield
		\be
		\label{rho_step}
		\rho(z)
		=\frac{b(z)}{a_1(z)a_2^*(z)}
		=\frac{b(z)}{a(z)}a_2(z)^2
		=\frac{\lambda(z)^2-1}{\lambda(z)^2+1}a_2(z)^2,
		\qquad z\in\mathbb R.
		\ee
		
		\begin{proposition}
			For $j=1,2$, the reflection coefficients satisfy
			\be
			\begin{aligned}
				r_1(z)&=\mathcal{O}\bigl(|z-\i\eta_j^\ell|^{1/2}\bigr),
				&&z\to\i\eta_j^\ell,\quad z\in\Sigma^{\ell,+},\\
				r_2(z)&=\mathcal{O}\bigl(|z-\i\eta_j^r|^{1/2}\bigr),
				&&z\to\i\eta_j^r,\quad z\in \Sigma^{r,+}.
			\end{aligned}
			\ee
			In particular, they extend continuously by zero to the
			endpoints of their respective intervals.
		\end{proposition}
		
		\begin{proof}
			Plugging \eqref{e:stepasingularity} into \eqref{e:steprnonoverlap} implies
			\be
			r_1(z)=\mathcal{O}\bigl(|z-\i\eta_1^\ell|^{1/2}\bigr),
			z\to\i\eta_1^\ell,\qquad\qquad
			r_2(z)=\mathcal{O}\bigl(|z-\i\eta_2^r|^{1/2}\bigr),
			z\to\i\eta_2^r.
			\ee
			At the overlap endpoints,
			\be
			\lambda(z)=\mathcal{O}\bigl(|z-p|^{1/4}\bigr),
			\qquad z\to p,\qquad
			p\in\{\i\eta_2^\ell,\i\eta_1^r\}.
			\ee
			Consequently, \eqref{e:steproverlap} gives
			\be
			\begin{aligned}
				r_1(z)
				&=\i\lambda(z)\frac{a_2(z_-)}{a_1(z_-)}
				=\mathcal{O}\bigl(|z-\i\eta_2^\ell|^{1/2}\bigr),
				&&z\to\i\eta_2^\ell,\\
				r_2(z)
				&=\i\lambda(z)\frac{a_1(z_+)}{a_2(z_+)}
				=\mathcal{O}\bigl(|z-\i\eta_1^r|^{1/2}\bigr),
				&&z\to\i\eta_1^r.
			\end{aligned}
			\ee
			The corresponding estimates on the lower contours follow by Schwarz symmetry.
		\end{proof}

The inverse problem for this pure-step initial datum is therefore described by the Riemann--Hilbert problem \ref{RHP:Phi}, with the spectral bands and reflection coefficients specified explicitly above.	

%%%%%%%%%%%%%%%%%%%%%%%%%%%%%%%%%%%%%%%%%%%%%%%%%%%%%%%%%%%%%%%%%%%%%%%%%%%%%%%%%%%%%% 
\section{Time evolution}\label{inverse}
So far we have considered only the  potential at time $t=0$, namely  $u(x,t=0)$. A key feature of the inverse scattering transform is that if $u(x,t)$ evolves according to \eqref{e:nls}, then the time evolution of the scattering data is linear and trivial. 
Time evolution of the inverse scattering problem simply amounts to allowing $t\neq 0$ in $V(z;x,t)$, which gives us the time-evolved inverse scattering problem

Now we are ready to prove Theorem~\ref{theorem2} which we repeat here for the reader's convenience.
\InverseProblem*

\begin{proof}
The proof closely follows the argument of \cite{Zhou89}.  Writing the Riemann--Hilbert problem for $\Phi$ in the form
$\Phi(z_+)-\Phi(z_-)=\Phi(z_-)(V(z)-I)$,   with $V$ as in \eqref{e:Vtilde} one can use the Sokhotski-Plemelj formula to the contour $\hat{\Sigma}=\R\cup\Sigma^\ell\cup\Sigma^r$  to obtain 
\[
\Phi(z)=I+\dfrac{1}{2\pi \i }\int_{\hat{\Sigma}}\dfrac{\Phi(\xi_-)(V(\xi)-I)d\xi}{\xi-z}\,,
\]
and in particular 
\[
\Phi(z_-)=I+\dfrac{1}{2\pi \i }\int_{\hat{\Sigma}}\dfrac{\Phi(\xi_-)(V(\xi)-I)d\xi}{(\xi-z)_-}\,,
\]
where the subscript $"-"$ stands for the limiting value of the integral on the negative  side of the oriented contour $\hat{\Sigma}$.
Defining $\mu(\xi):=\Phi(\xi_-)$  it is clear that the solvability of the Riemann--Hilbert problem \ref{RHP:Phi}
 is equivalent to the solvability of the following  integral equation for the matrix function $\mu$:
\bse
\be
(I-\mathcal{C}_{V-I})\mu=I,\label{e:inhomo}
\ee
where  $\mathcal{C}_{V-I}$ is the  operator acting from  $ L^2(\hat{\Sigma}, \C^{2\times2})$ to itself defined as
\be
(\mathcal{C}_{V-I} f)(z)=\frac{1}{2\pi\i}\int_{\hat{\Sigma}} \frac{f(s)(V(s)-I)}{(s-z)_-}\d s, \quad f\in L^2(\hat{\Sigma}, \C^{2\times2}).
\ee
\ese
The operator $\mathcal{C}_{V-I}$  is a bounded linear operator from $ L^2(\hat{\Sigma}, \C^{2\times2})$ to itself.
 The solvability of the inhomogeneous problem \eqref{e:inhomo}  follows from the proof  that the homogeneous problem $(I-\mathcal{C}_{V-I})\nu=0$ has only the trivial solution. Suppose that there exists $\nu\in L^2(\hat{\Sigma}, \C^{2\times 2})$ such that $(I-\mathcal{C}_{V-I})\nu=0$. Then the function defined as
\be
\Phi_0(z;x,t)=\frac{1}{2\pi\i}\int_{\hat{\Sigma}}\frac{\nu(s)(V(s;x,t)-I)}{s-z}\d s\,,
\ee
solves the following Riemann--Hilbert problem:
	\begin{enumerate}
	\item $\Phi_0(z;x,t)$ is analytic in $\Complex\setminus \hat{\Sigma}$.
	\item $\Phi_0(z;x,t)=\mathcal{O}(z^{-1})$, as $z\to\infty$\,.
	\item  $\Phi_0(z_+)=\Phi_0(z_-)V(z)$ for $z\in\hat{\Sigma}$.
	\end{enumerate}
 Let $H(z)=\Phi_0(z)\Phi_0^*(z)^T$, where the superscript $T$ denotes the transpose. Then  clearly $H(z)=\mathcal{O}(z^{-2})$  as $z\to\infty$ and $H(z)$ satisfies the following symmetry:
\be\label{e:symmetryH}
H(z)=\sigma_2H^*(z)\sigma_2.
\ee
 According to Cauchy-Goursat theorem, 	we have 
\be
\int_{\Real} H_+(z)\d z +\int_{\Sigma^{+}} (H_+(z)-H_-(z)) \d z=0.
\ee
 We first look at the second integral. Since
 \begin{align*}
 \int_{\Sigma^{+}}(H_+(z)-H_-(z)) \d z&=\int_{\Sigma^{+}}(\Phi_{0+}(z)\Phi^*_{0+}(z)^T-\Phi_{0-}(z)\Phi_{0-}^*(z)^T )\d z \\
 &=\int_{\Sigma^{+}}\Phi_{0-}(z)\left(V(z)-(V^*(z)^T)^{-1}\right)\Phi_{0+}^*(z)^T \d z=0,
\end{align*}
because  $V(z)-(V^*(z)^T)^{-1}=0$. Next we  show that $\Phi_0\equiv 0$ iff $\int_{\Real} H_+(z)\d z=0$.
From the definition of $H(z)$, we first have
\be\label{e:hp1}
0=\int_{\Real}H_+(z)\d z=\int_{\Real} \Phi_{0-}(z)V(z)\overline{\Phi_{0-}(z)}^T \d z.
\ee
Using the symmetry of $H(z)$ \eqref{e:symmetryH} and the symmetry of $\Phi_0(z)$, we further have
\be\label{e:hp2}
0=\sigma_2\int_\Real \overline{\Phi_{0-}(z)V(z)\overline{\Phi_{0-}(z)}^T} \d z \sigma_2 =\int_\Real \Phi_{0-}(z)\sigma_{2} \overline{V}(z)\sigma_{2} \overline{\Phi_{0-}(z)}^T \d z.
\ee
\eqref{e:hp1} and \eqref{e:hp2} together give
\be
\frac{1}{2}\int_\Real \Phi_{0-}(z)\left(V(z)+\sigma_{2} \overline{V}(z)\sigma_{2} \right) \overline{\Phi_{0-}(z)}^T \d z=0.
\ee
Since $V(z)+\sigma_{2} \overline{V}(z)\sigma_{2} $ is positive definite, it follows that $\Phi_{0-}(z)$ is identically zero. Therefore, Zhou's vanishing lemma directly gives the solvability of the Riemann--Hilbert problem \ref{RHP:Phi}.

To prove uniqueness, let $\Phi$ and $\Psi$ be two solutions.
We first show that both matrices have determinant one.
Since $\det V=1$, the functions $\det\Phi$ and $\det\Psi$
have no jump across the contour. The endpoint condition gives
$\det\Phi(z),\det\Psi(z)=\mathcal{O}(|z-p|^{-1/2})$
at every $p\in\partial\Sigma^\ell\cup\partial\Sigma^r$.
The normalization at infinity and Liouville's theorem then give
$\det\Phi=\det\Psi=1$.
Consequently, $\Psi^{-1}=\operatorname{adj}\Psi$ has the same
fourth-root endpoints as $\Psi$. The quotient $\Phi\Psi^{-1}$ has no jump and satisfies
$\Phi(z)\Psi(z)^{-1}=\mathcal{O}(|z-p|^{-1/2})$ at each endpoint.
These singularities are again removable, so $\Phi\Psi^{-1}$
is entire. Since it tends to $I$ at infinity, Liouville's theorem
gives $\Phi\Psi^{-1}=I$, proving uniqueness.

 Consequently, we denote by $\mu$ the unique solution of the inhomogeneous singular integral equation \eqref{e:inhomo}. And the solution $\Phi$ of the Riemann--Hilbert problem is recovered by
\be
\Phi(z;x,t)=I+\frac{1}{2\pi\i}\int_{\hat{\Sigma}}\frac{\mu(s)(V(s;x,t)-I)}{s-z}\d s.
\ee

Finally, it is a standard Lax pair argument (see \cite{steplikemkdv}) to show that the function 
\be
\~\Phi(z;x,t)=\Phi(z;x,t)\e^{-\i(xz+tz^2)\sigma_3},
\ee
satisfies the Lax pair \eqref{e:NLSLP}, with the function $u(x,t)$ given by
\begin{align}
\begin{split}
u(x,t)&=2\i\lim_{z\to\infty} z\Phi_{12}(z;x,t)=-\frac{1}{\pi}\int_{\hat\Sigma}\Big(\mu(V-I)\Big)_{12}\d s \\
&=\underbrace{-\frac{1}{\pi}\int_{\Real}\mu_{11}(x,t;s) \rho^*(s) \e^{-2\theta(s;x,t)}\d s}_{u_\R(x,t)}\underbrace{-\frac{1}{\pi}\int_{(\Sigma^{\ell,+}\setminus\Sigma^{r,+})\cup(\Sigma^{\ell,+}\cap\Sigma^{r,+})}\Big(\mu(V-I)\Big)_{12}(s)\d s}_{u_\Sigma (x,t)}\,.
\end{split}
\end{align}
Finally, assuming that   $\rho(z)\in L^{2,2}(\Real)\cap L^{1,2}(\Real)$  and differentiating \eqref{e:inhomo} with respect to $x$ and using the bounded inverse $(I-\mathcal{C}_{V-I})^{-1}$, we obtain $\partial_x\mu$, $\partial_x^2\mu\in L^2({\hat{\Sigma}})$.
Here, to verify the $C^2$ regularity in $x$, it suffices to focus on the integral over $\R$, since the second integral is supported on a compact contour and is uniformly bounded together with its first two $x$-derivatives.
Taking the second derivative of \eqref{e:recons} with respect to $x$, we have 
%\be\label{e:2derivative}
%\partial_x^2 u_\R(x,t)= -\frac{1}{\pi}\int_{\Real}s^2\mu_{11}(s;x,t) \rho^*(s) \e^{-2\theta(s;x,t)}\d s + 2\int_{\R} s \rho^*(s) \e^{-2\theta(s;x,t)}\partial_x\mu_{11} \d s +\int_{\R} \rho^*(s) \e^{-2\theta(s;x,t)}\partial_x^2 \mu_{11}  \d s.
%\ee
\be\label{e:2derivative}
\partial_x^2 u_\R(x,t)= -\frac{1}{\pi}\int_{\Real}\left(\partial_x^2 \mu_{11}-4\i s\partial_x\mu_{11}-4 s^2\mu_{11}(s;x,t)\right) \rho^*(s) \e^{-2\theta(s;x,t)}\d s.
\ee
Since $|\e^{-2\theta(s;x,t)}|=1$ for $s\in\Real$, each item in \eqref{e:2derivative} is controlled by the integrability of $s^2\rho$ and the $L^2$-bound of $\partial_x^k\mu_{11}(s;x,t)$ ($k=1,2$). In particular, the assumption $\rho\in L^{1,2}(\Real)$ ensures $s^2\rho \in L^{1}(\Real)$. So the integrals in \eqref{e:2derivative} are absolutely convergent.  For the term containing $s^2\mu_{11}$, write
$\mu_{11}=1+(\mu_{11}-1)$ and use $s^2\rho\in L^1\cap L^2$
and $\mu_{11}-1\in L^2$.  Therefore, $u(\cdot,t)\in C^2(\R)$.
Similarly, differentiating once in $t$ gives $u(x,\cdot)\in C^1(\R)$ under the same hypotheses.
Finally, since  $\rho(z)\in L^{2,2}(\Real)\cap L^{1,2}(\Real)$, the solution $u(x,t)\in C^2(\R)\times C^1(\R)$.
The existence of a classical solution for the exact step initial data \eqref{e:stepic} follows from the fact that the integral along the real line 
can be deformed in  $\C^+$ and $\C^-$ so that the integral becomes exponentially small for $t>0$. 
Details are provided in \cite{steplikemkdv} for the modified KdV equation.

\end{proof}

%%%%%%%%%%%%%%%%%%%%%%%%%%%%%%%%%%%%%%%%%%%%%%%%%%%%%%%%%%%%%%%%%%%%%%%%%%%%%%%%%%%%%%
\appendix
\section*{Appendix}
\setcounter{section}1
\setcounter{subsection}0
\setcounter{equation}0
\setcounter{figure}0
\def\thesection{\Alph{section}}
\def\theequation{\Alph{section}.\arabic{equation}}
\def\thetheorem{\Alph{section}.\arabic{theorem}}
\def\thefigure{\Alph{section}.\arabic{figure}}
\addcontentsline{toc}{section}{Appendix}

\subsection{Derivation of the normalization constants $c$ in \eqref{e:defc} and $c_0$ in \eqref{c0}, and the constant $E$ in \eqref{e:defEN}}
\label{a:calculationcE}
To compute integrals involving the square root of a quartic polynomial with two pairs of complex conjugate roots
\[
	R(z)= \sqrt{(z-z_1)(z-\overline{z_1})(z-z_2)(z-\overline{z_2})}
\]	
we introduce the M\"obius transformation
\begin{equation}\label{mobius}
w=w(z)
=
\frac{z_1-\overline{z_2}}{z_1-\overline{z_1}}
\frac{z-\overline{z_1}}{z-\overline{z_2}}.
\end{equation}
This map takes the generalized circle, possibly a straight line, through the four branch points of $R$ to the real line such that 
\[
w(\overline{z_1})=0,\qquad
w(z_1)=1,\qquad
w(\overline{z_2})=\infty,
\]
and
\[
w_2:=w(z_2)
=
\frac{|z_1-\overline{z_2}|^2}{4\Im z_1\Im z_2}\geq 1.
\]
The holomorphic differential $R(z)^{-1} \d z$ must transform to a holomorphic differential $S(w)^{-1} \d w$ in the new coordinates.
\begin{equation}
	\frac{\d z}{R(z)} = \frac{\d w}{S(w)} \qquad \Rightarrow  \qquad  S(w)^2 = \left( \od{w}{z} \right)^2 R(z)^2 
\end{equation}
Computing
\be
w(z)-1= \frac{\ov{z_1}-\ov{z_2}}{z_1-\ov{z_1}} \frac{z-z_1}{z-\ov{z_2}}, \quad 
w(z)-w_2=
\frac{(z_1-\ov{z_2})(\ov{z_1}-\ov{z_2})}
{(z_1-\ov{z_1})(z_2-\ov{z_2})}
\frac{z-z_2}{z-\ov{z_2}}, \quad
 \frac{\d w}{\d z} =\frac{z_1-\ov{z_2}}{z_1-\ov{z_1}}
\frac{\ov{z_1}-\ov{z_2}}{(z-\ov{z_2})^2}, 
\ee
it follows that 
\[
S(w)^2=-4\Im z_1\Im z_2\,w(w-1)(w-w_2).
\]
We define:
\be
m:=\frac1{w_2}=
\frac{4\Im z_1\Im z_2}{|z_1-\ov{z_2}|^2}, \qquad  w_\infty:=w(\infty)=
\frac{z_1-\ov{z_2}}{z_1-\ov{z_1}},
\qquad
n:=\frac1{w_\infty}=
\frac{z_1-\ov{z_1}}{z_1-\ov{z_2}}.
\ee

The path from $\overline{z_1}$ to $z_1$ is the lift to the first
sheet of the inverse image of $[0,1]$ under \eqref{mobius},
oriented from $\overline{z_1}$ to $z_1$.
With the normalization $R(z)\sim z^2$ at infinity, the corresponding
branch satisfies
\[
S(w)=-2\i\sqrt{\eta_1\eta_2}\,
\sqrt{w(1-w)(w_2-w)},\qquad 0<w<1.
\]
Using our M\"obius map to change variables, we now calculate
\begin{align}\label{e:I0}
	\begin{split}
	I_0
	:=
	\int_{\ov{z_1}}^{z_1}\frac{\d z}{R(z)}
	&=
	\int_0^1\frac{\d w}{S(w)}=
	\frac{\i}{2\sqrt{\Im z_1\Im z_2}}
	\int_0^1
	\frac{\d w}{\sqrt{w(1-w)(w_2-w)}}\\
	&=
	\frac{\i}{2\sqrt{\Im z_1\Im z_2}}
	\frac{2}{\sqrt{w_2}}
	\int_0^1
	\frac{\d t}
	{\sqrt{(1-t^2)(1-t^2/w_2)}}=\frac{2\i}{|z_1-\ov{z_2}|}\mathrm{K}(m).
	\end{split}
\end{align}
The formula for $c$ in \eqref{e:defc} follows from recognizing $c = (-2 I_0)^{-1}$.

Next we compute $c_0$ defined by the first equality in \eqref{c0}. The inverse map of \eqref{mobius} is
\be
z=\ov{z_2}+
\frac{\ov{z_1}-\ov{z_2}}{1-nw}.
\ee
Then we have 
\begin{align}\label{e:I1}
	\begin{split}
	I_1
	&:=
	\int_{\ov{z_1}}^{z_1}\frac{z\,\d z}{R(z)}
	=
	\frac{\i}{2\sqrt{\Im z_1\Im z_2}}
	\int_0^1
	\left(
	\ov{z_2}+\frac{\ov{z_1}-\ov{z_2}}{1-nw}
	\right)
	\frac{\d w}{\sqrt{w(1-w)(w_2-w)}},\\
	&=
	\frac{\i}{\sqrt{w_2\Im z_1\Im z_2}}
	\left[
	\ov{z_2}
	\int_0^1
	\frac{\d t}{\sqrt{(1-t^2)(1-mt^2)}}+
	(\ov{z_1}-\ov{z_2})
	\int_0^1
	\frac{\d t}
	{(1-nt^2)\sqrt{(1-t^2)(1-mt^2)}}
	\right] \\
	&=
	\frac{2\i}{|z_1-\ov{z_2}|}
	\left[
	\ov{z_2}\mathrm{K}(m)+(\ov{z_1}-\ov{z_2})\Pi(n,m)
	\right].
	\end{split}
\end{align}
Similarly, we have that
\begin{align}\label{e:I2}
	\begin{split}
	I_2
	&:=
	\int_{\ov{z_1}}^{z_1}\frac{z^2\,\d z}{R(z)}
	=
	\frac{\i}{2\sqrt{\Im z_1\Im z_2}}
	\int_0^1
	\left[
	\ov{z_2}^{\,2}
	+
	\frac{2\ov{z_2}(\ov{z_1}-\ov{z_2})}{1-nw}
	+
	\frac{(\ov{z_1}-\ov{z_2})^2}{(1-nw)^2}
	\right]
	\frac{\d w}{\sqrt{w(1-w)(w_2-w)}}\\
	&=
	\frac{2\i}{|z_1-\ov{z_2}|}
	\left[
	\ov{z_2}^{\,2}\mathrm{K}(m)
	+
	2\ov{z_2}(\ov{z_1}-\ov{z_2})\Pi(n,m)
	+
	(\ov{z_1}-\ov{z_2})^2A(n,m)
	\right],
	\end{split}
\end{align}
where $A(n,m)=\Pi(n,m)+n\partial_n\Pi(n,m)$. Using the standard identity
\be
	\partial_n\Pi(n,m)
	=
	\frac{
		n\mathrm{E}(m)+(m-n)\mathrm{K}(m)+(n^2-m)\Pi(n,m)
	}
	{2(m-n)(n-1)n},
\ee
we obtain
\begin{align}\label{e:Anm}
	\begin{split}
	A(n,m)
	&=
	\Pi(n,m)+
	\frac{n\mathrm{E}(m)+(m-n)\mathrm{K}(m)+(n^2-m)\Pi(n,m)}
	{2(m-n)(n-1)}\\
	&=\frac{\mathrm{K}(m)}{2(n-1)}+
	\frac{n\mathrm{E}(m)}{2(n-1)(m-n)}+
	\frac{2mn-3m-n^2+2n}{2(n-1)(m-n)}\Pi(n,m).
	\end{split}
\end{align}
Plugging \eqref{e:I0}, \eqref{e:I1}, \eqref{e:I2} and \eqref{e:Anm} into the definition of $c_0$ gives the result in \eqref{c0}.

To calculate $E$, we first introduce the normalized meromorphic differential of the third kind
\be\label{e:dOmega3}
	\d\Omega_3=\frac{z+c_3}{R(z)}\,\d z,
\ee
where the constant $c_3$ is  determined by requiring $\oint_{\mathcal{A}}\d\Omega_3=0$. Using \eqref{e:I0} and \eqref{e:I1} we have:
\be\label{e:defc3}
	c_3=-\,\frac{\oint_{\A}\frac{z\,\d z}{R(z)}}
	{\oint_{\A}\frac{\d z}{R(z)}}=-\ov{z_{2}}-(\ov{z_{1}}-\ov{z_{2}})\,\frac{\Pi(n,m)}{\mathrm{K}(m)}.
\ee
Given the asymptotic expansion of $R(z)$ as $z\to\infty$,
\be\label{e:asymR}
R(z)=z^{2}\left(1-\Re(z_{1}+z_{2})\,z^{-1}+\left(\Re z_1\Re z_2+\frac{(\Im z_1)^2+(\Im z_2)^2}2\right)z^{-2}+\mathcal{O}(z^{-3})\right),
\ee
 and
using the local coordinate $\zeta=1/z$ at the two points $\infty^{\pm}$ over
$z=\infty$, one finds
\be\label{e:dOmega3asym}
	\d\Omega_3
	=\mp\left(\frac{1}{\zeta}+c_3+\Re(z_{1}+z_{2})+\mathcal{O}(\zeta)\right)\d\zeta
	\qquad\text{at }\infty^{\pm}.
\ee
Hence $\d\Omega_3$ is  holomorphic on $\mathcal{S}\setminus\{\infty^{+},\infty^{-}\}$ with
simple poles at $\infty^{+}$ and $\infty^{-}$ of residues $\mp 1$. 
Both $\d p$ and $\d\Omega_3$ are normalized so that their $\A$-periods vanish. The Riemann bilinear relation between them therefore gives
\be\label{e:RBR}
2\pi\i
\sum_{P\in\{\infty^+,\infty^-\}}
\Res_P\bigl(p\,\d\Omega_3\bigr)
=
\oint_{\mathcal A}\d p\oint_{\mathcal B}\d\Omega_3
-
\oint_{\mathcal B}\d p\oint_{\mathcal A}\d\Omega_3
=0.
\ee
We evaluate the two residues. From \eqref{e:pqasymp}, at the two points over infinity,
\be\label{e:pasymp}
	p= \pm \left( \frac{1}{\zeta}+E+o(1) \right) \quad \text{at }\infty^{\pm},
\ee
Combining \eqref{e:dOmega3asym} and \eqref{e:pasymp}, the residue is the same at both points:
\be\label{e:residues}
	\Res_{\infty^{+}}\bigl(p\,\d\Omega_3\bigr)
	=\Res_{\infty^{-}}\bigl(p\,\d\Omega_3\bigr)
	=-\bigl(E+c_3+\Re(z_{1}+z_{2})\bigr).
\ee
Substituting \eqref{e:defc3}, \eqref{e:residues} into \eqref{e:RBR},  we have
\be
E=-c_3-\Re(z_{1}+z_{2})=\ov{z_{2}}-\Re(z_{1}+z_{2})+(\ov{z_{1}}-\ov{z_{2}})\,\frac{\Pi(n,m)}{\mathrm{K}(m)}.
\ee

Finally, to calculate $N$, combining \eqref{dqR} and \eqref{e:asymR}, we have
\begin{align*}
N&=\lim_{z\to\infty}q(z)-z^2=\lim_{z\to\infty} (R(z)-z^2-v p(z))=-v E+\lim_{z\to\infty} (R(z)-z^2-vz)\\
&=-v E + \Re z_1\Re z_2+\frac{(\Im z_1)^2+(\Im z_2)^2}2,
\end{align*}
which is equivalent to the expression in \eqref{e:defEN}, but \eqref{e:defEN} is more natural for the elliptic travelling-wave phase, since it separates the Galilean contribution from the elliptic frequency $\omega_0$.

In either of the two special cases
(i) $\Re z_1=\Re z_2$ or (ii) $\Im z_1=\Im z_2$,
the map $(z,R)\mapsto(-v-z,R)$ preserves the spectral curve
and pulls $\d p$ back to $-\d p$. Consequently, we have  $\d p(-v-z)=-\d p(z)$,
which implies that 
\be\label{e:CE}
p(-v-z)+p(z)=C
\ee
for some constant $C$. Taking $z\to\infty$ and using \eqref{e:pqasymp}, we obtain
\be\label{e:C}
C=2E-v.
\ee
Therefore, it remains to determine $C$ in each of the two special cases.

\paragraph{Case (i).} We first consider the $\dn$ case with $z_1=-\frac{v}{2}+\i\eta_1$ and  $z_2=-\frac{v}{2}+\i\eta_2$.
Since $-v-z_2=\ov{z_2}$, setting $z=z_2$ in \eqref{e:CE} gives
\be
C=p(-v-z_2)+p(z_2)=p(\ov{z_2})-p(z_2)=0,
\ee
where we used $p(z_2)=0$ and the $\mathcal{A}$-normalization $\oint_\mathcal{A}\d p=0$.
Substituting this into \eqref{e:C} gives 
\be\label{e:Edn}
E=\frac{v}{2}.
\ee

\paragraph{Case (ii).} We next consider the horizontal case with 
$z_1=\xi_1+\i\eta$, and $
z_2=\xi_2+\i\eta$. In this case, $-v-z_2=\ov{z_1}$.  Plugging $z=z_2$ into \eqref{e:CE}, we have
\be
C= p(-v-z_2)=p(\ov{z_1})=\frac{\Omega_1}2,
\ee
where we used $p(z_2)=0$, the $\mathcal{A}$-normalization
\be
p(\ov{z_{1,+}})=p(z_{1,+})
=
\int_{z_2}^{z_1}\d p
=
\frac{1}{2}\oint_{\mathcal{B}}\d p
=
\frac{\Omega_1}{2}.
\ee

It follows from  \eqref{e:C} that
\be\label{e:Ecn}
E=\frac{v}{2}+\frac{\Omega_1}{4}.
\ee

Consequently,
\be\label{e:Ncndn}
N=-vE+\frac{v^2}{4}-\frac{\omega_0}{2}
=
\begin{cases}
	-\dfrac{v^2}{4}-\dfrac{\omega_0}{2},
	& \Re(z_1)=\Re(z_2),\\[2mm]
	-\dfrac{v^2}{4}-\dfrac{\omega_0}{2}
	-\dfrac{v\Omega_1}{4},
	& \Im(z_1)=\Im(z_2).
\end{cases}
\ee

%%%%%%%%%%%%%%%%%%%%%%%%%%%%%%%%%%%%%%%%%%%%%%%%%%%%%%%%%%%%%%%%%%%%%%%%%%%%%%%%%%%%%%%%%%%%%%%%%%%
\subsection{Proof of Proposition~\ref{proposition_elliptic}}\label{a:u}
\begin{proof}
We first assume $0<\eta_1<\eta_2$ and prove \eqref{e:|u|^2}.
The constants $\Omega_0,\Omega_1,\Omega_2,E,N$ are real.
Since $\tau\in\i\R^+$, we have
$\overline{\theta_3(z;\tau)}=\theta_3(\overline z;\tau)$.
In particular, $\theta_3(\Omega/(2\pi))$ is real. Therefore,
		   \[
		   |u_e(x,t)|^2=(\eta_2-\eta_1)^2 \frac{\theta_3(0)^2  |\theta_3(2A(\infty)+\frac{\Omega}{2\pi})|^2}{\theta_3(\frac{\Omega}{2\pi})^2 |\theta_3(2A(\infty))|^2}.
		   \]
	Since $\overline c=-c$ and
		$R(\overline z)=\overline{R(z)}$,
		the chosen Abel paths and $A(\overline{z_2})=1/2$ give
		\[
		A(\overline z)=\frac12-\overline{A(z)}.
		\]
		In particular, $\Re A(\infty)=1/4$, and we write
		\be
		2A(\infty)=\frac12+\i\beta,
		\qquad \beta=2\Im A(\infty)\in\R.
		\ee
		
		  We then use the  identities $|\theta_3(u+\i v)|^2=\theta_3(u+\i  v)\theta_3(u-\i v)$  for $u, v$ real  and 
		$\theta_3(u+v)\theta_3(u-v)\theta_4^2(0)=\theta_4^2(u)\theta_3^2(v)-\theta_1^2(u)\theta_2^2(v)$,  for any $u,v$ complex numbers, so that we can 
		expand $|u_e(x,t)|^2 $ in the form 
		\begin{align}\label{e:u2}
			\begin{split}
				|u_e(x,t)|^2&				
				=(\eta_2-\eta_1)^2\frac{\theta_3^2(0) \theta_3(\frac{\Omega}{2\pi}+\frac{1}{2}+\i \beta) \theta_3(\frac{\Omega}{2\pi}+\frac{1}{2}-\i \beta)}{\theta_3^2(\frac{\Omega}{2\pi}) 
				\theta_3(\frac{1}{2}+\i \beta) \theta_3(\frac{1}{2}-\i \beta)}\\
				&=(\eta_2-\eta_1)^2\frac{\theta_3^2(0)}{\theta_4^2(0)}\frac{\theta_4^2(\frac{\Omega}{2\pi}+\frac{1}{2})\theta_3^2(\i \beta)-\theta_1^2(\frac{\Omega}{2\pi}+\frac{1}{2})
					\theta_2^2(\i \beta)}{\theta_3^2(\frac{\Omega}{2\pi}) \theta_4^2(\i \beta)}\\
				&=(\eta_2-\eta_1)^2\frac{\theta_3^2(0)}{\theta_4^2(0)}\left[\frac{\theta_3^2(\i \beta)}{\theta_4^2(\i \beta)} -\frac{\theta_1^2(\frac{\Omega}{2\pi}+\frac{1}{2})
					\theta_2^2(\i \beta)}{\theta_4^2(\frac{\Omega}{2\pi}+\frac{1}{2}) \theta_4^2(\i \beta)}\right]\\
				&=(\eta_2-\eta_1)^2\frac{\theta_3^2(0)}{\theta_4^2(0)}\left[\frac{\theta_3^2(\i \beta)}{\theta_4^2(\i \beta)} -\frac{\theta_2^2(0)}{\theta_3^2(0)}\frac{
					\theta_2^2(\i \beta)}{\theta_4^2(\i \beta)}\mbox{sn}^2\left(\frac{\Omega \mathrm{K}(m)}{\pi}+\mathrm{K}(m)\right)\right]\,,
			\end{split}
		\end{align}
		where we have used the identity  $\mbox{sn}(2\mathrm{K}(m)z;m)=\dfrac{\theta_3(0)}{\theta_2(0)}\dfrac{\theta_1(z)}{\theta_4(z)}$.		It remains to compute the ratios $\frac{\theta_2^2(\i \beta)}{\theta_4^2(\i \beta)}$ and $\frac{\theta_3^2(\i \beta)}{\theta_4^2(\i \beta)}$ in \eqref{e:u2}. We first notice that 
		\be
		\frac{\theta_2^2(\i \beta)}{\theta_4^2(\i \beta)}=\frac{\theta_2^2(2y-\frac{1}{2})}{\t_4^2(2y-\frac{1}{2})}=\frac{\t_1^2(2y)}{\t_3^2(2y)}, \qquad \qquad y=A(\infty).
		\ee
		Subsequently, using the formulae
		\bse
		\begin{gather}
			\t_1(2y)\t_2(0)\t_3(0)\t_4(0)=2\t_1(y)\t_2(y)\t_3(y)\t_4(y),\\
			\t_3(2y)\t_3(0)\t_4^2(0)=\t_3^2(y)\t_4^2(y)-\t_1^2(y)\t_2^2(y), \label{e:t32y}
		\end{gather}
		\ese
		we have
		\be\label{e:t132y}
		\frac{\t_1^2(2y)}{\t_3^2(2y)}=\frac{4\t_4^2(0)}{\t_2^2(0)}\frac{\t_3^2(y)}{\t_1^2(y)}\frac{\t_4^2(y)}{\t_2^2(y)}\left(\frac{\t_3^2(y)}{\t_1^2(y)}\frac{\t_4^2(y)}{\t_2^2(y)}-1\right)^{-2}.
		\ee
		The calculations of the ratios $\frac{\t_3^2(y)}{\t_1^2(y)}$ and $\frac{\t_4^2(y)}{\t_2^2(y)}$ are more involved. To do so,
		we first consider the function 
		\be\label{e:f421}
		f_{42}(z)=\frac{\t_2^2(0)}{\t_4^2(0)}\frac{\t_4^2(\int_{z_2}^z \omega)}{\t_2^2(\int_{z_2}^z \omega)}.
		\ee
	    The quasi-periodicity of the theta functions makes $f_{42}$ a single-valued meromorphic function on $\mathcal S$. It has a double zero at $z_1$ and a double pole at	$\overline{z_2}$, with no other zeros or poles.	The right-hand side of \eqref{e:f422}, regarded as a function	on $\mathcal S$, has the same divisor.	Their quotient is therefore constant.	Since both functions equal $1$ at $z_2$, we obtain
		\be\label{e:f422}
		f_{42}(z)=\frac{z-z_1}{z-\overline{z_2}}\frac{z_2-\overline{z_2}}{z_2-z_1}.
		\ee
		As $z\to\infty$, comparing \eqref{e:f421} with \eqref{e:f422} directly shows
		\be\label{e:t42}
		\frac{\t_4^2(y)}{\t_2^2(y)}=\frac{\t_4^2(0)}{\t_2^2(0)}\frac{z_2-\overline{z_2}}{z_2-z_1}.
		\ee
		Similarly, we define 
		\be
		f_{31}(z)=\frac{\t_2^2(0)}{\t_4^2(0)}\frac{\t_3^2(\int_{z_2}^z \omega)}{\t_1^2(\int_{z_2}^z \omega)}.
		\ee
	The same divisor comparison gives
		\be	f_{31}(z)=\frac{z-\overline{z_1}}{z-z_2}\frac{\overline{z_2}-z_2}{\overline{z_2}-\overline{z_1}}.
		\ee
		Again letting $z\to\infty$ gives
		\be\label{e:t31}
		\frac{\t_3^2(y)}{\t_1^2(y)}=\frac{\t_4^2(0)}{\t_2^2(0)}\frac{\overline{z_2}-z_2}{\overline{z_2}-\overline{z_1}}.
		\ee
		Plugging \eqref{e:t42} and \eqref{e:t31} into \eqref{e:t132y}  and using the identities
		\[
		\dfrac{\t_2^4(0)}{\theta_3^4(0)}=m,\;\;\;\dfrac{\t_4^4(0)}{\theta_2^4(0)}=\dfrac{1-m}{m},\quad m=1-\left|\frac{z_1-z_2}{z_1-\overline{z_2}}\right|^2\,,
		\]
		 then gives 
		\begin{align}\label{e:expr2}
			\begin{split}
				\frac{\t_2^2(0)}{\t_4^2(0)}\frac{\theta_2^2(\i \beta)}{\theta_4^2(\i \beta)}&=\frac{\t_2^2(0)}{\t_4^2(0)}\frac{\t_1^2(2y)}{\t_3^2(2y)}=\frac{4\t_4^4(0)}{\t_2^4(0)}\left|\frac{z_2-\overline{z_2}}{z_2-z_1}\right|^2 \left(\frac{\t_4^4(0)}{\t_2^4(0)}\left|\frac{z_2-\overline{z_2}}{z_2-z_1}\right|^2-1\right)^{-2}\\
				&=4\frac{1-m}{m}\left|\frac{z_2-\overline{z_2}}{z_2-z_1}\right|^2\left(\frac{1-m}{m}\left|\frac{z_2-\overline{z_2}}{z_2-z_1}\right|^2-1\right)^{-2}
				=\frac{4\eta_1\eta_2}{(\eta_2-\eta_1)^2}\,.
			%	=\frac{2\eta_1\eta_2}{(\eta_1^2-\eta_2)^2}\sqrt{\frac{(\eta_1-\eta_2)^2+(\xi_1-\xi_2)^2}{\eta_1\eta_2}}.
			\end{split}
		\end{align}
		Moreover, for  $\frac{\theta_3^2(\i \beta)}{\theta_4^2(\i \beta)}$, we have 
		\be
		\frac{\theta_3^2(\i \beta)}{\theta_4^2(\i \beta)}=\frac{\theta_3^2(2y-\frac{1}{2})}{\theta_4^2(2y-\frac{1}{2})}=\frac{\t_4^2(2y)}{\t_3^2(2y)}.
		\ee
		Here we use \eqref{e:t32y} and 
		\be
		\t_4(2y)\t_4^3(0)=\t_3^4(y)-\t_2^4(y).
		\ee
		Then 
		\be\label{e:t432y}
		\frac{\t_4^2(2y)}{\t_3^2(2y)}=\frac{\t_3^2(0)}{\t_4^2(0)}\left(\frac{\frac{\t_3^2(y)}{\t_2^2(y)}-\frac{\t_2^2(y)}{\t_3^2(y)}}{\frac{\t_4^2(y)}{\t_2^2(y)}-\frac{\t_1^2(y)}{\t_3^2(y)}}
		\right)^2\,.
		\ee
		Similarly, we define rational functions $f_{32}(z)$ and  $f_{13}(z)$
		\bse
		\begin{gather}
			f_{32}(z)=\frac{\t_2^2(0)}{\t_3^2(0)}\frac{\t_3^2(\int_{z_2}^z \omega)}{\t_2^2(\int_{z_2}^z \omega)}=\frac{z-\overline{z_1}}{z-\overline{z_2}}\frac{z_2-\overline{z_2}}{z_2-\overline{z_1}},\\
			f_{13}(z)=\frac{\t_4^2(0)}{\t_2^2(0)}\frac{\t_1^2(\int_{z_2}^z \omega)}{\t_3^2(\int_{z_2}^z \omega)}=\frac{z-z_2}{z-\overline{z_1}}\frac{\overline{z_2}-\overline{z_1}}{\overline{z_2}-z_2},
		\end{gather}
		\ese
		and let $z\to\infty$ to obtain
		\be\label{e:t32t12}
		\frac{\t_3^2(y)}{\t_2^2(y)}=\frac{\t_3^2(0)}{\t_2^2(0)} \frac{z_2-\overline{z_2}}{z_2-\overline{z_1}},\qquad \frac{\t_1^2(y)}{\t_3^2(y)}=\frac{\t_2^2(0)}{\t_4^2(0)}\frac{\overline{z_2}-\overline{z_1}}{\overline{z_2}-z_2}.
		\ee
		Plugging \eqref{e:t42} and \eqref{e:t32t12} into \eqref{e:t432y}, we obtain
		\begin{align}\label{e:expr1}
			\begin{split}
				\frac{\t_3^2(0)}{\t_4^2(0)}\frac{\t_4^2(2y)}{\t_3^2(2y)}&=\frac{\t_3^4(0)}{\t_4^4(0)}\left(\frac{\frac{\t_3^2(0)}{\t_2^2(0)} \frac{z_2-\overline{z_2}}{z_2-\overline{z_1}}-\frac{\t_2^2(0)}{\t_3^2(0)} \frac{z_2-\overline{z_1}}{z_2-\overline{z_2}}}{\frac{\t_4^2(0)}{\t_2^2(0)}\frac{z_2-\overline{z_2}}{z_2-z_1}-\frac{\t_2^2(0)}{\t_4^2(0)}\frac{\overline{z_2}-\overline{z_1}}{\overline{z_2}-z_2}}
				\right)^2
				=\frac{1}{1-m}\left(
				\frac{\frac{1}{\sqrt{m}} \frac{z_2-\overline{z_2}}{z_2-\overline{z_1}}
					-
					\sqrt{m} \frac{z_2-\overline{z_1}}{z_2-\overline{z_2}}}
				{\frac{\sqrt{1-m}}{\sqrt{m}}\frac{z_2-\overline{z_2}}{z_2-z_1}
					-\frac{\sqrt{m}}{\sqrt{1-m}}\frac{\overline{z_2}-\overline{z_1}}{\overline{z_2}-z_2}}
				\right)^2=\frac{(\eta_1+\eta_2)^2}{(\eta_1-\eta_2)^2}\,.
			\end{split}
		\end{align}
		Plugging \eqref{e:expr1} and \eqref{e:expr2} into \eqref{e:u2}, and since
		\[
		\frac{\Omega\mathrm K(m)}{\pi}
		=-|z_1-\overline{z_2}|(x-x_0-vt)
		\]
		and
		$\sn^2(\mathrm K(m)-s;m)=\sn^2(\mathrm K(m)+s;m)$,
		we obtain \eqref{e:|u|^2}.
		The relation \eqref{period_u^2} can be obtained using the identity
		\[
		\int_0^{2\mathrm{K}(m)}\mbox{sn}^2(s,m)\d s=\frac{2}{m}(\mathrm{K}(m)-\mathrm{E}(m)).
		\]
		In the case $\Re(z_1)=\Re(z_2)$, we have $2A(\infty)=1/2$,
		so $u_e(x,t)\e^{-2\i(xE+tN)}$ is real by
		\eqref{elliptic_general}. This quantity equals
		$\eta_2-\eta_1>0$ at $\Omega=0$ and does not vanish
		by \eqref{e:|u|^2}. Hence it is positive by continuity.
		Using $\dn^2(u;m)+m\sn^2(u;m)=1$, we obtain
		\[
		u_e(x,t)=(\eta_1+\eta_2)\dn\left(
		|z_1-\overline{z_2}|(x-x_0-vt)+\mathrm{K}(m);m
		\right)\e^{2\i(xE+tN)}.
		\]

		In order to prove \eqref{elliptic_cn} in the case  $\Im(z_1)=\Im(z_2)=\eta$, we need to consider the expression \eqref{ellitpic_cn0}, namely 
		\[
			 u_e(x,t)=\Re(z_1-z_2)\Im(z_2)\frac{\theta_3(0) \theta_3(\frac{\Omega}{2\pi}+\frac{\tau+1}{2})}{\theta_3(\frac{\Omega}{2\pi})\theta_3'(\frac{\tau+1}{2})\dfrac{\i|z_1-\overline{z_2}|}{{ 4}\mathrm{K}(m)}}\e^{2\i(xE+tN)}\,.
		 \]
		To express it in terms of a Jacobi elliptic function we use the following formulae in \cite{Lawden}:
	\begin{align*}
		& \theta_1(z)=-\i\e^{\i\pi z+\frac{\i\pi}{4}\tau}\theta_3\left(z+\frac{1+\tau}{2}\right)\,,\qquad  \theta'_1(0)=\pi\theta_2(0)\theta_3(0)\theta_4(0)\,,\qquad \pi\theta_4(0)^2=2\mathrm{K}(m)\sqrt{1-m}\,,\\
		&  \theta_2(z)=\e^{\i\pi z+\frac{\i\pi}{4}\tau}\theta_3\left(z+\frac{\tau}{2}\right),\quad \mbox{cn}(2\mathrm{K}(m)z;m)=\dfrac{  \theta_3(\frac{1}{2})\theta_2(z)}{  \theta_2(0)\theta_3(z+\frac{1}{2})}\,, \quad 
		\theta_1(z)=-\theta_2(z+\frac{1}{2})\,,
	\end{align*}
	and then obtain
\begin{align*}
	u_e(x,t)&=-\frac{\Re(z_1-z_2)\eta}{\dfrac{\i|z_1-\overline{z_2}|}{{ 4}\mathrm{K}(m)}  \pi} \frac{ \theta_2(\frac{\Omega}{2\pi}+\frac{1}{2})}{\theta_3(\frac{\Omega}{2\pi})\theta_2(0)\theta_4(0)}\e^{2\i(xE+tN)}\e^{-\frac{\i\Omega}{2}}\\
	&=
	\frac{4 \i \mathrm{K}(m)\Re(z_1-z_2)\eta}{|z_1-\overline{z_2}|\pi\theta_4^2(0)}\mbox{cn}\left(\frac{\Omega  \mathrm{K}(m)}{\pi}+\mathrm{K}(m)\right)\e^{2\i(xE+tN)}\e^{-\frac{\i\Omega}{2}}\\
	&=2\i \e^{2\i(xE+tN)} \e^{-\frac{\i\Omega}{2}} \frac{\Re(z_1-z_2)\eta}{|z_1-\overline{z_2}|\sqrt{1-m}}\mbox{cn}\left(\frac{\Omega  \mathrm{K}(m)}{\pi}+\mathrm{K}(m)\right)
	=-2\i\eta\e^{2\i(xE+tN)}\e^{-\frac{\i\Omega}{2}}\mbox{cn}\left(\frac{\Omega  \mathrm{K}(m)}{\pi}+\mathrm{K}(m)\right).
\end{align*}
Using
\[
\frac{\Omega\mathrm K(m)}{\pi}
=-|z_1-\overline{z_2}|(x-x_0-vt),
\qquad
\cn(\mathrm K(m)-s;m)=-\cn(\mathrm K(m)+s;m),
\]
together with \eqref{e:Ecn} and \eqref{e:Ncndn}, gives
\[
u_e(x,t)=2\i\eta\,
\cn\left(|z_1-\overline{z_2}|(x-x_0-vt)+\mathrm K(m);m\right)
\e^{\i(vx-v^2t/2-\omega_0t-\Omega_0/2)},
\]
which proves \eqref{elliptic_cn}.
Taking absolute values also proves \eqref{e:|u|^2} in this case,
since $\cn^2(s;m)=1-\sn^2(s;m)$.
 
			\end{proof}

%%%%%%%%%%%%%%%%%%%%%%%%%%%%%%%%%%%%%%%%%%%%%%%%%%%%%%%%%%%%%%%%%%%%%%%%%%%%%%%%%%%%%%%%%%%%%%%%%%%
\subsection{Proof of Proposition~\ref{prop:verticaloverlap}}\label{a:verticaloverlap}
%%%%%%%%%%%%%%%%%%%%%%%%%%%%%%%%%%%%%%%%%%%%%%%%%%%%%%%%%%%%%%%%%%%%%%%%%%%%%%%%%%%%%%%%%%%%%%%%%%%

\begin{proof}
	Let $\gamma$ be a nontrivial open subarc common to both
	bands, chosen away from their endpoints and the zeros of
	$\d p^\ell$ and $\d p^r$.
	We first show that $\gamma$ is vertical. Since each band is a
	connected real-analytic arc, this will imply the same conclusion
	for the entire bands.
	
	For $s\in\{\ell,r\}$, define
	\be\label{e:defFs}
	F^s(z)=\left(\frac{\d p^s}{\d z}\right)^2
	=\frac{\left(z^2-\Re(z_1^s+z_2^s)z+c_0^s\right)^2}
	{\prod_{j=1}^2(z-z_j^s)(z-\ov{z_j^s})}.
	\ee
	By \eqref{p}, $F^s\in\R(z)$, where $\R(z)$ denotes the field of rational functions in $z$ with real coefficients. Moreover, $F^s(z)=1+\mathcal{O}(z^{-2})$ as $z\to\infty$, and all its finite poles have order at most one.
	
	Since $p^\ell$ and $p^r$ are real on $\gamma$, their
	differentials are real and nonzero along any nonzero
	tangent to $\gamma$. Hence
	\be
	\frac{F^r(z)}{F^\ell(z)}
	=\left(\frac{\d p^r}{\d p^\ell}\right)^2>0,
	\qquad z\in\gamma.
	\ee
	This quotient is nonconstant. Otherwise, its limit at
	infinity would give $F^r=F^\ell$. 
	Thus the field
	\be
	\mathcal K:=
	\{f\in\R(z):f\text{ is real on }\gamma
	\text{ wherever defined}\}
	\ee
	contains a nonconstant function. By L\"uroth's theorem
	over $\R$ (see \cite[Theorem 9.19]{Milne2022}), there exists $T\in\R(z)$, 
%	 that we  may  chose  of positive degree,  
which we may choose so that $T(\infty)=\infty$,
	 such that
	$\mathcal K=\R(T)$ and   $T$ is real on
	$\gamma$.  Hence $T(z)=T(\bar z)$ there.
	Since $\gamma\subset\C^+$, we have $z\neq\bar z$.
	Thus $\deg T\geq2$, because two different points have the same image. To determine the degree of $T$, we use the pole
	structure of $F^\ell$. Since both $p^\ell$ and $T$ are real on $\gamma$,
	away from the poles and critical points of $T$ we have
	\be
	\left(\frac{\d p^\ell}{\d T}\right)^2=\frac{F^\ell(z)}{T'(z)^2}
	>0,
	\qquad z\in\gamma.
	\ee
	This quotient is a
	real rational function and therefore belongs to
	$\mathcal K=\R(T)$. Consequently, there exists
	$Q\in\R(z)$ such that
	\be
	F^\ell(z)=Q(T(z))T'(z)^2,\quad z\in\C\,.
	\ee
	We now show that $T$ is a quadratic polynomial.  
Write
	$T(z)\sim az^n$ and $Q(z)\sim bz^m$ as
	$z\to\infty$, where $a,b\neq0$, $n\geq1$,
	and $m\in\mathbb Z$. Since $F^\ell(z)\to1$,
	comparison of leading powers gives
	\be
	n(m+2)=2.
	\ee
	In particular, $m+2>0$.	It remains to exclude finite poles of $T$.
	If $T$ had a pole of order $k\geq1$ at a
	finite point $z_0$, then
	$T(z)\sim c(z-z_0)^{-k}$ would imply
	\be
	F^\ell(z)\sim
	bc^{m+2}k^2(z-z_0)^{-(k(m+2)+2)}.
	\ee
	This would give a pole of order at least three,
	contradicting the fact that $F^\ell$ has at most
	simple finite poles. Thus $T$ has no finite poles and is a polynomial.
	Since $n=\deg T\geq2$, the relation $n(m+2)=2$
	gives $n=2$ and $m=-1$. Hence
	\be
	T(z)=Az^2+Bz+C,\qquad
	A,B,C\in\R,\quad A\neq0.
	\ee
	Finally, since 
	%$T$ is real on $\gamma\subset\C^+$,
	\be
	0=\Im T(z)=\Im z\,(2A\Re z+B)
		\ee
implies  $\Re z=\xi$  with   $\xi=-B/(2A)$, we have proved that  $\gamma$ is contained in a vertical line.
By real analyticity, this identity extends
along both connected bands and to their endpoints.
Thus both bands are vertical segments on the common
line $\Re z=\xi$, and their intersection is a single
segment of positive length.
Remark~\ref{prop:dn_vertical} identifies both
backgrounds as Galilean transforms   of $\dn$ waves,
with $v^\ell=v^r=-2\xi$.

\end{proof}

\subsection{Proof of Proposition \ref{propm}}\label{a:propm}

Our proofs of the list of properties of $W^s(x;z)$ recorded in Proposition~\ref{propm} rely on estimates of matrices and matrix-valued functions. Throughout, we use the following standard norms:
\[
  |A|_F := \left(  \operatorname{Tr}(A^\dagger A) \right)^{1/2}= \left( \sum_{i,j=1}^2 |a_{ij}|^2 \right)^{1/2}, \qquad A \in \C^{2\times 2}
\]
denotes the usual Frobenius norm; while  
\[
\|M\|_{L^\infty(\Omega)} := \operatorname*{ess\,sup}_{x\in\Omega}\,|M(x)|_{F},
\qquad M: \Omega \to \C^{2\times 2}.
\]
is the usual $L^\infty$ norm.

\begin{proof}[\textbf{Proof of properties \ref{analyticity}-\ref{jumps}.}]
Denote by $m_1^s(x;z)$ and $m_2^s(x;z)$ the first and second columns of $m^s(x;z)$, respectively. 
Below we prove the result for the first column $m^\ell_1(x;z)$. 
The proofs for $m_2^\ell(x;z)$, $m_1^r(x;z)$ and $m_2^r(x;z)$ are analogous. 
From \eqref{e:Intm} we can rewrite the integral equation for $m_1^\ell$ in the form
\begin{equation}\label{e:m1.operator}
(1 - \mathcal{K})m_1^\ell = e_1, \qquad e_1 = \begin{pmatrix}
    1 \\
    0
\end{pmatrix},
\end{equation}
where $\mathcal{K}$ is the integral operator 
\bse \label{e:m.kernel}
\begin{gather}
    (\mathcal{K} f)(x;z) := \int_{-\infty}^x K(x,y;z) f(y;z) \d y, \\%
    K(x,y;z) = \begin{bmatrix} 1 & 0 \\ 0 & \e^{2\i (p^\ell - E^\ell)(x-y)} \end{bmatrix} O^\ell(y;z)^{-1} \Delta U^\ell(y)\, O^\ell(y;z),
    \qquad 
    \Delta U^\ell(y) := U_0(y) - U_e^\ell(y).
\end{gather}
\ese

For $s\in\{\ell,r\}$, let $\Sigma_0^s$ denote the spectral gap associated with the elliptic background $u_e^s$.  Recall that $O(z;x)$, see \eqref{defO}, is unimodular, analytic in $\C \setminus (\Sigma^\ell \cup \Sigma_0^\ell)$, and uniformly bounded for $x\in \R$ and $z$ bounded away from the spectral endpoints $\partial \Sigma^\ell$ where it has $\frac{1}{4}$-root singularities. Also, from Lemma~\ref{dp_inequality}, $\Im p^\ell(z) >0$ for $z\in \C^+ \setminus \Sigma^{\ell,+}$ and $\Im p^\ell(z) = 0$ for $z \in \R $ and $z_\pm \in \Sigma^\ell$.
For fixed $\epsilon >0$, 
 let $D^+_\epsilon = (\C^+\setminus(\Sigma^{\ell,+} \cup \Sigma^{\ell,+}_0)) \setminus  \left( \bigcup_{p \in \partial \Sigma^{\ell,+}} B_\epsilon(p) \right)$,  
where $B_\epsilon(z_0)=\{z\in\C: |z-z_0|<\epsilon\}$.
Then 
\begin{equation}
    |K(x,y;z)|  \leq c_\epsilon \, |\Delta u^\ell(y) | , \qquad z \in D^+_\epsilon, \quad y<x, %\leq x_0\,,
\end{equation}
where $c_\epsilon$ is proportional to $\|O \|^2_{L^\infty(\R \times D_\epsilon^+)}$. 
% where $c_\eps(x_0)=\|O(\cdot,z)\|_{L^\infty(-\infty,x_0)}\|^2$. 
It follows that for any fixed $x_*\in \Real$ and for every $f\in L^{\infty}((-\infty,x_*)\times D^+_\epsilon)$, we have
\begin{equation}\label{e:Kestimate}
\| \mathcal{K} f\|_{L^{\infty}((-\infty, x_*)\times D_\epsilon^+)} \leq c_\epsilon \|\Delta u^\ell\|_{L^1(-\infty, x_*)} \|f\|_{L^\infty((-\infty, x_*)\times D_\epsilon^+)}.
\end{equation}
A straightforward induction argument then shows that 
\begin{align*}
	\|\mathcal{K}^{j}f\|_{L^{\infty}((-\infty,x_*)\times D_\epsilon^+)} \leq \frac{1}{j!} \left( c_\epsilon \|\Delta u^\ell\|_{L^1(-\infty,x_*)} \right)^j \|f\|_{L^{\infty}((-\infty,x_*)\times D_\epsilon^+)},
\end{align*}
and the resolvent operator exists and satisfies \[
    \|(1-\mathcal{K})^{-1}\|_{L^{\infty}((-\infty,x_*)\times D_\epsilon^+) \to L^{\infty}((-\infty,x_*)\times D_\epsilon^+) } 
    \leq \e^{ c_\epsilon \|\Delta u^\ell\|_{L^1(-\infty,x_*)} }.
\]
Consequently the Neumann series  
\begin{equation}
    \label{e:m1.nuemann}
    m_1^\ell(x;z) = \sum_{j=0}^\infty (\mathcal{K}^j e_1) (x;z),
\end{equation}
converges uniformly  for   $z\in D^+_\epsilon$  and defines the unique solution of \eqref{e:m1.operator} for each $x \in \R$. Moreover, this solution is analytic for $z \in \C^+ \setminus( \Sigma^\ell \cup \Sigma_0^\ell)$  with continuous boundary values on the real axis and on each side of the cuts, away from the spectral endpoints.  The proofs for the other columns are similar.

Now consider the boundary values $m^\ell(x;z_\pm)$ for  $z_\pm \in \operatorname{int}(\Sigma^\ell)$. From the above argument $m^\ell(x;z_\pm)$ is the unique solution of the Volterra equation \eqref{e:Intml}. 
For $z \in \Sigma^\ell$, using the jump relations $p^\ell(z_+) + p^\ell(z_-) =0$ and $O^\ell(x;z_+) = O^\ell(x;z_-)  \e^{2\i x E^\ell \sigma_3}(\i \sigma_1)$ we have
\begin{multline*}
    m^\ell(x;z_+) = I + \int_{-\infty}^x \e^{\i(x-y)(p^\ell(z_-)+ E^\ell) \sigma_3}  \sigma_1 \e^{-2\i y E^\ell \sigma_3} O^\ell(y;z_-)^{-1} \Delta U^\ell(y) O^\ell(y;z_-) \\
    \times \e^{2\i y E^\ell  \sigma_3} \sigma_1 m^\ell(y;z_+) \e^{-\i(x-y)(p^\ell(z_-)+ E^\ell) \sigma_3} \d y, \qquad z \in \Sigma^\ell.
\end{multline*}
Defining $\eta^\ell(x;z_+) := \e^{2\i x E^\ell \sigma_3} \sigma_1 m^\ell(x;z_+) \sigma_1 \e^{-2\i x E^\ell \sigma_3}$, the above equation is equivalent to 
\begin{multline*}
    \eta^\ell(x;z_+) = I + \int_{-\infty}^x \e^{-\i(x-y)(p^\ell(z_-) - E^\ell) \sigma_3} O^\ell(y;z_-)^{-1} \Delta U^\ell(y) O^\ell(y;z_-) \eta^\ell(y;z_+) \e^{\i(x-y)(p^\ell(z_-) - E^\ell) \sigma_3}\, \d y.
\end{multline*}
This is exactly the integral equation \eqref{e:Intml} satisfied by $m^\ell(x;z_-)$. Uniqueness implies that $\eta^\ell(x;z_+) = m^\ell(x;z_-)$ for $z \in \Sigma^\ell$. Using the relation \eqref{e:defm} between $W^\ell(x;z)$ and $m^\ell(x;z)$ together with the jump relations of $O^\ell(x;z)$ and $p^\ell(z)$ for $z \in \Sigma^\ell$ gives \eqref{jumpW}. This proves property~\ref{jumps}. 

A similar argument applied to the first column shows that
\[
m_1^\ell(x;z_+)
= \e^{\i\Omega^\ell}\e^{-\i\Omega^\ell\sigma_3}
m_1^\ell(x;z_-),
\qquad z\in\operatorname{int}(\Sigma_0^\ell)\cap\C^+,
\]
and through \eqref{e:defm} it follows that
$W_1^\ell(x;z_+)=W_1^\ell(x;z_-)$ for
$z\in\operatorname{int}(\Sigma_0^\ell)\cap\C^+$.
Morera's theorem then implies that $W_1^\ell(x;z)$ is analytic in $\C^+ \setminus \Sigma^{\ell,+}$ as claimed in property~\ref{analyticity}. 
\end{proof}
%%%%%%%%%%%%%%%%%%%%%%%%%%%%%%%%%%%%%%%%%%%%%%%%%%%%%%%%%%%%%%%%%%%%%%%%%%%%%%%%%%%%%%%%

%%%%%%%%%%%%%%%%%%%%%%%%%%%%%%%%%%%%%%%%%%%%%%%%%%%%%%%%%%%%%%%%%%%%%%%%%%%%%%%%%%%%%%%%
\begin{proof}[\textbf{Proof of property \ref{asymptotics}.}]
We prove the result for $m_1^\ell(x;z)$ here. 
The proofs for the other columns are analogous, and omitted for brevity. 
%We drop the superscript ``$\ell$'' for brevity.
Let $R^\ell = \max_{z \in \Sigma^\ell} |z|$.
For $\Delta u^\ell \in L^{1}(\R^-)$, property~\ref{analyticity} of Prop.~\ref{propm} shows that the function $m_1^\ell(x;z)$ defined by \eqref{e:m1.nuemann} satisfies \eqref{e:m1.operator} for each $z \in \overline{\C^+}$ with $|z| > R^\ell$. Recalling the relation \eqref{e:defm} relating $W_1^\ell$ to $m_1^\ell$ and observing that $p^\ell(z) - z - E^\ell$ and $O^\ell(x;z)$ are analytic functions for $|z| > R^\ell$, the existence of an expansion of
$W_1^\ell(x;z)\e^{\i(x-x_0^\ell)z}$  to order $z^{-(n-1)}$ follows immediately from the following lemma.

\begin{lemma}
\label{lem:iterates}
For $\Delta u^\ell \in \mathcal{W}^{n,1}(\R^-)$ the iterates $\mathcal{K}^j e_1$ in \eqref{e:m1.nuemann} admit expansions of the form
\begin{equation}\label{e:iterates}
    (\mathcal{K}^j e_1)(x;z) =  \begin{dcases} \sum_{p = \lceil j/2 \rceil}^{n-1} F_{j,p}(x) z^{-p} + \mathcal{O}(z^{-n}),  & j \leq 2(n-1), \\
   \mathcal{O}(z^{-n}) & j>2(n-1)\,,
   \end{dcases}
\end{equation}
with coefficients $F_{j,p} \in \mathcal{W}^{n-p,\infty}(\R^-)$. 
\end{lemma}
The Neumann series satisfies the exact identity
\[
m_1^\ell
=\sum_{j=0}^{2n-2}\mathcal{K}^j e_1
+(1-\mathcal{K})^{-1}\mathcal{K}^{2n-1}e_1.
\]
By Lemma~\ref{lem:iterates},
the last term is $\mathcal{O}(z^{-n})$ in $L^\infty(\R^-)$.
Summing \eqref{e:iterates} for $j=0,\ldots,2n-2$ and using
\eqref{e:defm} gives the expansion of
$W_1^\ell(x;z)\e^{\i(x-x_0^\ell)z}$ through order
$z^{-(n-1)}$ for fixed $x$.

The expansion coefficients  of $W_1^\ell$ can be directly computed from careful analysis of the Neumann iterates, but this is tedious. Once the existence of the  expansion to order $n$ is established, the coefficients are computed using the method presented in section~\ref{a:asymW}. Then our proof of property~\ref{asymptotics} is complete once Lemma~\ref{lem:iterates} is established. 
\end{proof}
\begin{proof}[Proof of Lemma~\ref{lem:iterates}]
The function $O^\ell(x;z)$ is unimodular, and $C^\infty$ as a function of $x$, and analytic for $|z| > R^\ell$. For $|z| > R^\ell$, both $O^\ell(x;z)$ and its inverse admit convergent Laurent expansions
\[
    O^\ell(x;z) = \sum_{k=0}^\infty O^\ell_k(x) z^{-k},
    \qquad
    O^\ell(x;z)^{-1} = \sum_{k=0}^\infty Q^\ell_k(x) z^{-k},
\]
where the coefficients $O^\ell_k(x)$ and $Q^\ell_k(x)$ are smooth functions bounded together with their $x$-derivatives expressible in terms of the asymptotic potential $u_e^\ell(x)$; the first term $O^\ell_1$ is given by \eqref{e:Oasymp}.
It follows that for any $n\in \N$ the kernel $K(x,y;z)$ of \eqref{e:m.kernel} can be expanded as
\[
% \begin{gathered}
%     K(x,y;z) = \begin{bmatrix} 1 & 0 \\ 0 & e^{i\phi(z,x-y)} \end{bmatrix} \left\{ \Delta U^\ell + \sum_{j=1}^\infty z^{-j}   A_j\left(y, \Delta u^\ell(y), \overline{ \Delta u^\ell(y) }\, \right)   \right\} \\
%     \phi(z,s) =2 s (p^\ell(z) - E^\ell)
% \end{gathered}
\begin{gathered}
    K(x,y;z) =     
       \begin{bmatrix} 1 & 0 \\ 0 & \e^{\i\phi(x-y,z)} \end{bmatrix} 
       \left\{ \Delta U^\ell(y) + 
       \sum_{j=1}^{n-1} 
       \frac{A_j(y)}{z^{j}}  + \frac{1}{z^n} R_n(x,y;z)
       \right\}, \\
   A_j(y) = \Delta u^\ell(y) \sum_{s=0}^j Q^\ell_s(y) \begin{bsmallmatrix} 0 & 1 \\ 0 & 0 \end{bsmallmatrix} O^\ell_{j-s}(y) - \overline{\Delta u^\ell(y)} \sum_{s=0}^j Q^\ell_s(y) \begin{bsmallmatrix} 0 & 0 \\ 1 & 0 \end{bsmallmatrix} O^\ell_{j-s}(y), \qquad j\geq 0  \\ 
    \phi(\xi,z) =2 \xi (p^\ell(z) - E^\ell)
\end{gathered}
\]
where the coefficients $A_j(y) \in \mathcal{W}^{n,1}(\R^-)$.

Consider the first iterate $\mathcal{K}e_1$, 
\[
    (\mathcal{K}e_1)(x;z) = \sum_{j=0}^{n-1} z^{-j} \int_{-\infty}^x
       \begin{bmatrix} 1 & 0 \\ 0 & \e^{\i\phi(x-y,z)} \end{bmatrix}  A_j(y) e_1 \d y  
       + \frac{1}{z^n} \int_{-\infty}^x   \begin{bmatrix} 1 & 0 \\ 0 & \e^{\i\phi(x-y,z)} \end{bmatrix} R_n(x,y;z)e_1 \d y\,.
\]
Repeated integration by parts in the second row gives
\begin{multline}
    (\mathcal{K}e_1)(x;z) = \sum_{j=0}^{n-1} z^{-j} \left[ \Etl \int_{-\infty}^x A_j(y) e_1 \d y - \Ebl \sum_{s=1}^{j} \left( \frac{z}{2\i(p^\ell(z) - E^\ell)} \right)^s \partial_x^{s-1} \left( A_{j-s}(x) e_1 \right)  \right] \\
    + z^{-n}\left\{
     \int_{-\infty}^x
      \begin{bmatrix}
    	1 & 0\\
    	0 & \e^{\i\phi(x-y,z)}
    \end{bmatrix}
    R_n(x,y;z) e_1  \d y 
    +
    \Ebl \sum_{s=1}^{n} \left( \frac{z}{2\i(p^\ell(z) - E^\ell)} \right)^s \Bigg[-\partial_x^{s-1} \left( A_{n-s}(x) e_1 \right) \right. \\
      +\int_{-\infty}^x\e^{\i\phi(x-y,z)}\partial_y^{s} \left( A_{n-s}(y) e_1 \right)  \d y \Bigg]\Bigg\} 
\end{multline}

We observe that: (1) $z/(p^\ell(z)-E^\ell)
=1+\mathcal{O}(z^{-2})$ is analytic at infinity;
(2) after expanding its powers, the coefficients of $z^{-j}$
belong to $\mathcal{W}^{n-j,\infty}(\R^-)$, $1\leq j<n$;
(3) the remainder is $\mathcal{O}(z^{-n})$ uniformly in
$x\in\R^-$, since $|R_n(x,y;z)|\leq C|\Delta u^\ell(y)|$
and $A_j\in\mathcal{W}^{n,1}(\R^-)$.

Now proceed by induction. Suppose that $(\mathcal{K}^m e_1)(x;z)$ admits an expansion of the form \eqref{e:iterates}, then by repeating the integration by parts argument one sees that if the leading coefficient $F_{m,\lceil{m/2}\rceil}(x)$ of $(\mathcal{K}^m e_1)(x;z)$ has a non-zero entry in its second row, then $(\mathcal{K}^{m+1} e_1)(x;z)$ decays at the same rate, but its leading coefficient $F_{m+1,\lceil{(m+1)/2}\rceil}(x)$ vanishes in its second row. 
If the leading coefficient of an iterate vanishes in its second
row, the off-diagonal form of $\Delta U^\ell$ and integration by
parts show that the coefficient at the same order in the next
iterate vanishes entirely. Thus, in either case,
$(\mathcal{K}^{m+2}e_1)(x;z)
=\mathcal{O}(z^{-\lceil m/2\rceil-1})$. 
 We note the smoothness assumption is always sufficient to integrate by parts to order $\mathcal{O}(z^{-n}).$

\end{proof}

%%%%%%%%%%%%%%%%%%%%%%%%%%%%%%%%%%%%%%

In order to prove the last condition of Proposition~\ref{propm}, it's convenient to work with an integral equation more directly related to the Jost functions $W^s(x;z)$ and which  isolates the singular behavior at the endpoints $z_j^s$ of $\Sigma^{s,+}$. 

Choose $\delta>0$ sufficiently small that the disks of radius $\delta$ centered at the four endpoints of $\Sigma^s$ are pairwise disjoint.
Define 
\begin{gather}\label{e:defWhat}
\begin{aligned}
	&\widehat W^s(x;z) := 
    \begin{dcases}
     \gamma^s(z) \e^{-\i x E^s \sigma_3}  W^s(x;z),
     & |z-z_2^s|<\delta, \\
      \gamma^s(z)^{-1} \e^{-\i x E^s \sigma_3} W^s(x;z),
     & |z-z_1^s|<\delta,
	\end{dcases} \\[.5em]
	&\widehat W_e^s(x;z) := 
    \begin{dcases}    
    \gamma^s(z) \e^{-\i x E^s \sigma_3} W_e^s(x;z) , 
    & |z-z_2^s|<\delta, \\
    \gamma^s(z)^{-1} \e^{-\i x E^s \sigma_3} W_e^s(x;z), 
    & |z-z_1^s|<\delta,
    \end{dcases} 
    \end{aligned}\\
    \nonumber
    \Delta \widehat  U^s(x) := 
	\begin{bmatrix}
		0 & \left( u_0(x) - u^s_e(x,0) \right) \e^{-2\i E^s x} \\
		- \overline{\left( u_0(x) - u^s_e(x,0) \right)} \e^{2\i E^s x} & 0 \\
	\end{bmatrix}.
\end{gather}
Then using \eqref{e:defW0}, \eqref{e:defm}, and the integral equations defining $m^s$, \eqref{e:Intm}, we find that 
$\widehat W^s$ satisfies the integral equation
\begin{gather}\label{What.int.eq}
	\widehat W^s(x;z) = \widehat W^s_e (x;z) + 
    (\widehat{\mathcal{K}}^{\,s} 	\widehat W^s)(x;z), \\
\widehat{\mathcal{K}}^{\,s} f(x;z)=   
    \int_{\infty_s}^x \widehat{K}^{s}(x,y;z) \Delta \widehat U^s(y) f(y;z) \, \d y\,,
\end{gather}
where
\begin{equation}\label{What.kernel}
	\widehat{K}^{s}(x,y;z) 
    = \widehat W_e^s(x;z) \widehat W_e^s(y;z)^{-1} 
    = \e^{-\i x E^s \sigma_3} W_e^s(x;z) W_e^s(y;z)^{-1} \e^{\i y E^s \sigma_3}. 
\end{equation}

\begin{lemma}\label{lem:kernel}
	$\widehat{K}^s(x,y;z)$ is an analytic function of $z$ for each value of $x,y \in \R$.
\end{lemma}	
\begin{proof}
	By construction the function $W_e^s(x;z)$ is analytic in $\C \setminus \Sigma^s$ and bounded away from the four branch points $z_1, \ov{z_1}, z_2, \ov{z_2}$ where it admits at worst $\frac{1}{4}$-root singularities. Along $\Sigma^s$ it satisfies a constant jump  \eqref{e:jumpW0} in both $x$ and $z$. The product $W_e^s(x;z) W_e^s(y;z)^{-1}$ is then continuous across $\Sigma^s$ and so can admit at worst isolated singularities at the four branch points. However, the growth condition at each endpoint guarantees the local growth rate of the product is at most $\frac{1}{2}$-root blow up; it follows that the branch point singularities are removable. 
\end{proof}

\begin{proof}[\textbf{Proof of Property \ref{singularity}.}]
	Suppose that $z \in \operatorname{int}(\Sigma^s)_\pm$, so that $\Im p^s(z) = 0$. 
	Here the $\pm$ subscript indicates that we must regard $z$ as a boundary value from either the left or right side of $\Sigma^s$ as the Jost functions are branched along these contours.  We want to study the behavior of the Jost functions $W^s(x;z)$ as $z \to z_k^s$, $k=1,2$. 
	
	Our starting point is the integral equation \eqref{What.int.eq}.
	We know from Lemma~\ref{lem:kernel} that the integral kernel is analytic in $z$. We now want to investigate its behavior in $x, y$ for $z$ near a branch point. 
	Expanding the kernel \eqref{What.kernel}  using \eqref{defO} and \eqref{e:defm} gives
	\begin{equation}\label{e:hatKs}
		\widehat{K}^s(x,y;z) = \frac{\gamma^s(z)^{-2}}{4} \widehat{K}^s_{1} (x,y;z) +\frac{1}{2} \widehat{K}^s_{0} (x,y;z) +\frac{\gamma^s(z)^{2}}{4} \sigma_3 \widehat{K}^s_{1} (x,y;z) \sigma_3,
	\end{equation}
	where
	\begin{subequations}\label{e:hatK}
		\begin{multline}
			\widehat{K}^s_{0} (x,y;z) = 
				\begin{bmatrix}
					H^s_{11}(x;z) H^s_{22}(y;z)    &   0 \\
					0  &   H^s_{21}(x;z) H^s_{12}(y;z)   
				\end{bmatrix} 
				\e^{-\i (x-y) p^s(z)}  \\
		+
				\begin{bmatrix}
					H^s_{12}(x;z) H^s_{21}(y;z) & 0 \\
					0 &  H^s_{22}(x;z) H^s_{11}(y;z)
				\end{bmatrix}
				\e^{ \i(x-y) p^s(z)
			},
		\end{multline}
		\begin{multline}\label{e:hatK1s}
			\widehat{K}^s_{1} (x,y;z) = 
				\begin{bmatrix*}[r]
					H^s_{11}(x;z) \\ -H^s_{21}(x;z)
				\end{bmatrix*}
				\begin{bmatrix*}[r] 
					H^s_{22}(y;z)    & H^s_{12}(y;z) 
				\end{bmatrix*} 
				\e^{-\i (x-y) p^s(z)}  \\
-
				\begin{bmatrix*}[r]
					H_{12}^s(x;z) \\ -H_{22}^s(x;z) 
				\end{bmatrix*} 
				\begin{bmatrix} 
					H_{21}^s(y;z) & H_{11}^s(y;z)
				\end{bmatrix}
				\e^{ \i(x-y) p^s(z)
			}.
		\end{multline}
		\end{subequations}
		We can further rewrite $\widehat{K}_1^s(x,y;z)$ in the form 
		\begin{multline}\label{e:hatK1}
			\widehat{K}^s_{1} (x,y;z) =
			-2\i \begin{bmatrix*}[r]
				H^s_{11}(x;z) \\ -H^s_{21}(x;z)
			\end{bmatrix*}
			\begin{bmatrix} 
				H_{21}^s(y;z) & H_{11}^s(y;z)
			\end{bmatrix}
			\sin( (x-y) p^s(z)) \\
			+  \begin{bmatrix*}[r]
				H^s_{11}(x;z) \\ -H^s_{21}(x;z)
			\end{bmatrix*}
			\begin{bmatrix} 
				(H^s_{22}-H^s_{21})(y;z)    & (H^s_{12}-H^s_{11})(y;z) 
			\end{bmatrix} 
			\e^{-\i (x-y) p^s(z)}  \\
			- \begin{bmatrix} 
				(H^s_{12}-H^s_{11})(x;z) \\  -(H^s_{22}-H^s_{21})(x;z)
			\end{bmatrix}  
			\begin{bmatrix} 
				H_{21}^s(y;z) & H_{11}^s(y;z)
			\end{bmatrix}
			\e^{ \i(x-y) p^s(z)}\,.
		\end{multline}
		Now consider the behavior of the kernel as $z \to z_2^s$ with $z \in \Sigma^{s,+}$. 
		The functions $H_{jk}(x;z)$ are uniformly bounded in $x$ for $z$ in sufficiently small neighborhoods of each branch point. 
		The quasi-momentum $p^s(z)$ satisfies $p^s(z) = \bigo{(z-z_2^s)^{1/2}}$ as $z \to z_2^s$.  
		Likewise, the Abel map satisfies $A^s(z) = \bigo{(z-z_2^s)^{1/2}}$ for $z \to z_2^s$. 
		Since $A^s(z_2^s)=0$ and $\theta_3$ is even, \eqref{e:Hij} gives:
		\[
		H^s_{12}(x;z_2^s)=H_{11}^s(x;z_2^s),\quad H_{22}^s(x;z_2^s)=H_{21}^s(x;z_2^s).
		\]
		Expanding the theta quotients in the local variable $A^s(z)$, we obtain
		\begin{subequations}\label{H.z2}
			\begin{align}
				&\left| H^s_{12\pm}(x; z) - H^s_{11\pm}(x;z) \right| \leq C |z-z_2^s|^{1/2} 
				&&z \to z_2^s, \\
				&\left| H^s_{22\pm}(x; z) - H^s_{21\pm}(x;z) \right| \leq C |z-z_2^s|^{1/2}
				&&z \to z_2^s, 
			\end{align}
		\end{subequations}
		for some constant $C > 0$ which is independent of $x$. 
		Fixing a value $x_* \in \R$,  for boundary values $z\in\Sigma^{s,+}$ sufficiently close
		to $z_2^s$, we have, 
		\begin{equation}
			\begin{aligned}
				& |\widehat{K}^\ell(x,y;z) | \leq  C(x_* - y + 1),  && y < x< x_*,\\
				& | \widehat{K}^r(x,y;z) | \leq  C(y-x_*  + 1), && x_* < x < y,
			\end{aligned}
		\end{equation} 
		where the constant $C$ is independent of $x,y, z$. 
	
		For the complex-domain limits, \eqref{e:hatKs}--\eqref{e:hatK1s},
		together with $\widehat K_1^\ell(x,y;z_j^\ell)=0$, give
		\[
		\left|\e^{\i(x-y)p^\ell(z)}\widehat K^\ell(x,y;z)\right|
		\leq C(1+x-y),\qquad y\leq x,
		\]
		near either upper endpoint.
		Apply \eqref{What.int.eq} to
		$\e^{\i(x-x_0^\ell)p^\ell(z)}\widehat W_1^\ell(x;z)$,
		whose background term is uniformly bounded.
		The same Volterra estimate and dominated convergence give
		the endpoint limits; for fixed $x$, the scalar factor can
		then be removed.
		For $W_2^r$, use the factor
		$\e^{-\i(x-x_0^r)p^r(z)}$.
        The solution of \eqref{What.int.eq} can then be constructed as a convergent Neumann series $$\widehat{W}^s(x;z)=\sum_{n=0}^\infty \left( \widehat{\mathcal{K}}^{s} \right)^n \widehat{W}_e^s (x;z)$$ for each $z \in \Sigma^s$. The estimates in \eqref{e:singW} follow from an argument similar to that between \eqref{e:Kestimate} and \eqref{e:m1.nuemann}. Finally, \eqref{e:W.endpoints} follows from using dominated convergence to write         \begin{equation}\label{e:endpoint.integral.eq}
            \lim_{z \to z^s_j} \widehat{W}^s(x;z) = \widehat{W}^s_e (x;z_j^s)+
             \int_{\infty_s}^x \widehat{K}^s(x,y; z^s_j) \Delta\widehat{U}^s(y)
             \widehat{W}^s(y;z^s_j) \d y = \widehat{W}^s(x;z^s_j),  \qquad j =1,2.
        \end{equation}

        For the H\"older estimate, assume in addition that
        $\Delta u^\ell\in L^{1,2}(\R^-)$ and
        $\Delta u^r\in L^{1,2}(\R^+)$.
        We prove the estimate near $p=z_1^\ell$. Choose $\delta>0$ sufficiently small so that $\overline{B_\delta(p)}$ contains no other endpoint of $\Sigma^{\ell,+}$. Throughout the proof, $z,w\in \Sigma^{\ell,+}\cap\overline{B_\delta(p)}$ are taken from the same side of the contour, and the corresponding boundary value subscripts are omitted.
        To show the H\"older continuity, we need to estimate
        \be\label{e:diffhatW}
        \begin{aligned}
        	\widehat W^\ell(x;z)-\widehat W^\ell(x;w)
        	=
        &	\widehat W_e^\ell(x;z)-\widehat W_e^\ell(x;w)
        	+
        	\int_{-\infty}^{x}
        	\widehat K^\ell(x,y;z)\Delta\widehat U^\ell(y)
        	\left(
        	\widehat W^\ell(y;z)-\widehat W^\ell(y;w)
        	\right)\,\d y
        	\\
        	&+
        	\int_{-\infty}^{x}
        	\left(
        	\widehat K^\ell(x,y;z)-\widehat K^\ell(x,y;w)
        	\right)
        	\Delta\widehat U^\ell(y)
        	\widehat W^\ell(y;w)\,\d y.
        \end{aligned}
        \ee
        We first estimate the background difference and the kernel difference. We write
        \[
        \begin{aligned}
        	\left|\widehat W_e^\ell(x;z)-\right.&\left.\widehat W_e^\ell(x;w) \right|
        	\leq
        	\left|\e^{-\i xE^\ell\sigma_3}
        	\left[
        	\frac{O^\ell(z;x,0)}{\gamma^\ell(z)}
        	-
        	\frac{O^\ell(w;x,0)}{\gamma^\ell(w)}
        	\right]
        	\e^{\i xE^\ell\sigma_3}
        	\e^{-\i\left((x-x_0^\ell)p^\ell(z)+x_0^\ell E^\ell\right)\sigma_3}\right|
        	\\
        	&+
        	\left|\gamma^\ell(w)^{-1}
        	\e^{-\i xE^\ell\sigma_3}
        	O^\ell(w;x,0)
        	\e^{\i xE^\ell\sigma_3}
        	\left[
        	\e^{-\i\left((x-x_0^\ell)p^\ell(z)+x_0^\ell E^\ell\right)\sigma_3}
        	-
        	\e^{-\i\left((x-x_0^\ell)p^\ell(w)+x_0^\ell E^\ell\right)\sigma_3}
        	\right]\right|
        	\\
        	&
        	\leq
        C(1+|x|)|z-w|^{1/2}.
        \end{aligned}
        \]
        We next estimate the kernel difference. By \eqref{e:hatKs},
        \be\label{e:diffhatK}
        \begin{aligned}
        	\widehat K^\ell(x,y;z)-\widehat K^\ell(x,y;w)
        	&=
        	\frac14
        	\left[
        	\frac{\widehat K_1^\ell(x,y;z)}{\gamma^\ell(z)^2}
        	-
        	\frac{\widehat K_1^\ell(x,y;w)}{\gamma^\ell(w)^2}
        	\right]+
        	\frac12
        	\left[
        	\widehat K_0^\ell(x,y;z)
        	-
        	\widehat K_0^\ell(x,y;w)
        	\right]
        	\\
        	&+
        	\frac14\sigma_3
        	\left[
        	\gamma^\ell(z)^2\widehat K_1^\ell(x,y;z)
        	-
        	\gamma^\ell(w)^2\widehat K_1^\ell(x,y;w)
        	\right]
        	\sigma_3.
        \end{aligned}
        \ee
        We estimate the three differences separately.

        The  explicit
        expressions of $\frac{\widehat K_1^\ell(x,y;z)}{\gamma^\ell(z)^2}$ and
        $\widehat K_0^\ell(x,y;z)$ in \eqref{e:hatK} show that their Lipschitz constants in the local variable $(z-p)^{1/2}$ are bounded by $C(1+|x-y|)$. Hence
        \[
        	\left|
        	\frac{\widehat K_1^\ell(x,y;z)}{\gamma^\ell(z)^2}
        	-
        	\frac{\widehat K_1^\ell(x,y;w)}{\gamma^\ell(w)^2}
        	\right|\leq
        	C(1+|x-y|)
        	\left|z-w\right|^{1/2},
        \]
        and
        \[
        	\left|
        	\widehat K_0^\ell(x,y;z)
        	-
        	\widehat K_0^\ell(x,y;w)
        	\right|\leq
        	C(1+|x-y|)
        	\left|z-w	\right|^{1/2}.
        \]
        
        It remains to estimate the last term in \eqref{e:diffhatK}. Since the complete kernel is analytic at $p=z_1^\ell$, while
        $\gamma^\ell(z)^2=\mathcal{O}((z-p)^{-1/2})$, we have
        \[
        \widehat K_1^\ell(x,y;p)=0.
        \]
  Since $p=z_1^\ell$, 
  $(z-p)^{1/2}\gamma^\ell(z)^2$
  is analytic and nonzero near $p$. Hence
  \be\label{e:estimate2}
  \begin{aligned}
  	\gamma^\ell(z)^2\widehat K_1^\ell(x,y;z)
  	-
  	\gamma^\ell(w)^2\widehat K_1^\ell(x,y;w)
  	&=
  	(z-p)^{1/2}\gamma^\ell(z)^2
  	\left[
  	\frac{\widehat K_1^\ell(x,y;z)}{(z-p)^{1/2}}
  	-
  	\frac{\widehat K_1^\ell(x,y;w)}{(w-p)^{1/2}}
  	\right]
  	\\
  	&+
  	\left[
  	(z-p)^{1/2}\gamma^\ell(z)^2
  	-
  	(w-p)^{1/2}\gamma^\ell(w)^2
  	\right]
  	\frac{\widehat K_1^\ell(x,y;w)}{(w-p)^{1/2}}.
  \end{aligned}
  \ee
  
We first estimate 
  $\frac{\widehat K_1^\ell(x,y;z)}{(z-p)^{1/2}}
  -
  \frac{\widehat K_1^\ell(x,y;w)}{(w-p)^{1/2}}$.
  By \eqref{e:hatK1s}, write
  \[
  	\widehat K_1^\ell(x,y;z)
  	=
  	\underbrace{
  		\begin{bmatrix}
  			H_{11}^\ell(x;z)\\
  			-H_{21}^\ell(x;z)
  		\end{bmatrix}
  		\begin{bmatrix}
  			H_{22}^\ell(y;z) &H_{12}^\ell(y;z)
  		\end{bmatrix}
  	}_{=:M_-^\ell(x,y;z)}
  	\e^{-\i(x-y)p^\ell(z)}-
  	\underbrace{
  		\begin{bmatrix}
  			H_{12}^\ell(x;z)\\
  			-H_{22}^\ell(x;z)
  		\end{bmatrix}
  		\begin{bmatrix}
  			H_{21}^\ell(y;z)&H_{11}^\ell(y;z)
  		\end{bmatrix}
  	}_{=:M_+^\ell(x,y;z)}
  	\e^{\i(x-y)p^\ell(z)}.
  \]
  Since $\widehat K_1^\ell(x,y;p)=0$,
  \[
  M_-^\ell(x,y;p)\e^{-\i(x-y)p^\ell(p)}
  =
  M_+^\ell(x,y;p)\e^{\i(x-y)p^\ell(p)}.
  \]
  Consequently,
  \[
  \begin{aligned}
  	\widehat K_1^\ell(x,y;z)
  	={}&
  	\left(M_-^\ell(x,y;z)-M_-^\ell(x,y;p)\right)
  	\e^{-\i(x-y)p^\ell(z)}-
  	\left(M_+^\ell(x,y;z)-M_+^\ell(x,y;p)\right)
  	\e^{\i(x-y)p^\ell(z)}
  	\\
  	&-
  	2\i M_-^\ell(x,y;p)\e^{-\i(x-y)p^\ell(p)}
  	\sin\left((x-y)\left(p^\ell(z)-p^\ell(p)\right)\right).
  \end{aligned}
  \]
  Applying this identity at $z$ and $w$, dividing by the corresponding
  square roots, and subtracting gives
  \[
  \begin{aligned}
  	&
  	\frac{\widehat K_1^\ell(x,y;z)}{(z-p)^{1/2}}
  	-
  	\frac{\widehat K_1^\ell(x,y;w)}{(w-p)^{1/2}}
  	\\
  	={}&
  	\left[
  	\frac{M_-^\ell(x,y;z)-M_-^\ell(x,y;p)}{(z-p)^{1/2}}
  	\e^{-\i(x-y)p^\ell(z)}
  	-
  	\frac{M_-^\ell(x,y;w)-M_-^\ell(x,y;p)}{(w-p)^{1/2}}
  	\e^{-\i(x-y)p^\ell(w)}
  	\right]
  	\\
  	&-
  	\left[
  	\frac{M_+^\ell(x,y;z)-M_+^\ell(x,y;p)}{(z-p)^{1/2}}
  	\e^{\i(x-y)p^\ell(z)}
  	-
  	\frac{M_+^\ell(x,y;w)-M_+^\ell(x,y;p)}{(w-p)^{1/2}}
  	\e^{\i(x-y)p^\ell(w)}
  	\right]
  	\\
  	&-
  	2\i M_-^\ell(x,y;p)\e^{-\i(x-y)p^\ell(p)}
  	\left[
  	\frac{\sin\left((x-y)(p^\ell(z)-p^\ell(p))\right)}{(z-p)^{1/2}}
  	-
  	\frac{\sin\left((x-y)(p^\ell(w)-p^\ell(p))\right)}{(w-p)^{1/2}}
  	\right].
  \end{aligned}
  \]
  
The first two terms are easy to estimate. We have
  \[
  	\left|
  	\frac{M_\pm^\ell(x,y;z)-M_\pm^\ell(x,y;p)}{(z-p)^{1/2}}
  	\e^{\pm\i(x-y)p^\ell(z)}
  	-
  	\frac{M_\pm^\ell(x,y;w)-M_\pm^\ell(x,y;p)}{(w-p)^{1/2}}
  	\e^{\pm\i(x-y)p^\ell(w)}
  	\right|\leq
  	C(1+|x-y|)
  	\left|z-w\right|^{1/2}.
  \]
  
  For the last term, write
  \[
  \frac{\sin\left((x-y)(p^\ell(z)-p^\ell(p))\right)}{(z-p)^{1/2}}
  =
  (x-y)
  \underbrace{
  	\frac{\sin\left((x-y)(p^\ell(z)-p^\ell(p))\right)}
  	{(x-y)(p^\ell(z)-p^\ell(p))}
  }_{=:A(z)}
  \underbrace{
  	\frac{p^\ell(z)-p^\ell(p)}{(z-p)^{1/2}}
  }_{=:B(z)}.
  \]
  Hence
  \be\label{e:diffsin}
  \begin{aligned}
  	&
  	\frac{\sin\left((x-y)(p^\ell(z)-p^\ell(p))\right)}{(z-p)^{1/2}}
  	-
  	\frac{\sin\left((x-y)(p^\ell(w)-p^\ell(p))\right)}{(w-p)^{1/2}}
  	\\
  	&\qquad=
  	(x-y)A(z)\left(B(z)-B(w)\right)
  	+
  	(x-y)B(w)\left(A(z)-A(w)\right).
  \end{aligned}
  \ee
  The branch-point expansion of $p^\ell$ gives
  \be\label{e:diffB}
  |B(z)-B(w)|\leq C|z-w|,
  \qquad
  |B(w)|\leq C.
  \ee
  Since the boundary values of $p^\ell$ are real on the band,
  $|A(z)|\leq1$.
  Boundedness and uniform Lipschitz continuity of $\sin s/s$
  on $\R$, with its continuous value at $s=0$, give
  \be\label{e:diffA}
  \begin{aligned}
  	|A(z)-A(w)|
  	\leq
  	C\min\{1, |x-y||z-w|^{1/2}\}.
  \end{aligned}
  \ee
  By \eqref{e:diffsin} and \eqref{e:diffB}, the difference is
  bounded by
  $C|x-y||z-w|+C|x-y|^2|z-w|^{1/2}$, and hence
  \be\label{e:diffsin1}
  \begin{aligned}
  	\left|
  	\frac{\sin\left((x-y)(p^\ell(z)-p^\ell(p))\right)}{(z-p)^{1/2}}
  	-
  	\frac{\sin\left((x-y)(p^\ell(w)-p^\ell(p))\right)}{(w-p)^{1/2}}
  	\right|\leq 
  	C|x-y|(1+|x-y|) 
  	\left|z-w\right|^{1/2}.
  \end{aligned}
  \ee
  Combining \eqref{e:diffB}--\eqref{e:diffsin1} gives
  \be\label{e:diffK1}
  \begin{aligned}
  	\left|
  	\frac{\widehat K_1^\ell(x,y;z)}{(z-p)^{1/2}}
  	-
  	\frac{\widehat K_1^\ell(x,y;w)}{(w-p)^{1/2}}
  	\right|\leq
  	C 
  	(1+|x-y|)^2
  	\left|z-w\right|^{1/2}.
  \end{aligned}
  \ee
  The same decomposition, without taking the difference between $z$
  and $w$, gives
  \be\label{e:boundK1}
  \left|
  \frac{\widehat K_1^\ell(x,y;w)}{(w-p)^{1/2}}
  \right|
  \leq
  C(1+|x-y|).
  \ee
  
  Since $(z-p)^{1/2}\gamma^\ell(z)^2$ is analytic near $p$,
  \be\label{e:diffgamma}
  \left|
  (z-p)^{1/2}\gamma^\ell(z)^2
  -
  (w-p)^{1/2}\gamma^\ell(w)^2
  \right|
  \leq C|z-w|,
\qquad\qquad
  \left|(z-p)^{1/2}\gamma^\ell(z)^2\right|\leq C.
  \ee
  Substituting \eqref{e:diffK1}--\eqref{e:diffgamma} into \eqref{e:estimate2}  yields
  \[
  \begin{aligned}
  	\left|
  	\gamma^\ell(z)^2\widehat K_1^\ell(x,y;z)
  	-
  	\gamma^\ell(w)^2\widehat K_1^\ell(x,y;w)
  	\right|\leq
  	C
  	(1+|x-y|)^2
  	|z-w|^{1/2}.
  \end{aligned}
  \]

  	In the norm $\sup_{x\leq x_*}\frac{|f(x)|}{(1+x_*-x)}$,
  	the same Volterra estimate as above gives a uniformly
  	bounded resolvent under the $L^{1,2}$ assumption.
  	Applying it to \eqref{e:diffhatW} and using the preceding
  	estimates, we obtain, for every fixed $x\in\R$,
  	 \[
  	 \left|
  	 \widehat W^\ell(x;z)-\widehat W^\ell(x;w)
  	 \right|
  	 \leq C|z-w|^{1/2}.
  	 \]
  	 The cases $z=p$ or $w=p$ follow by continuous extension.

\end{proof}

\subsection{Asymptotics of $W(x,t;z)$ as $z\to\infty$}\label{a:asymW}
Fix $t$ and suppress it in the notation for $W$ and $\Gamma$.
Let $W(x;z)$ be a normalized Schwarz-symmetric solution of the
Zakharov--Shabat linear problem
\[
W_x={ \mathcal L}(u,z)  W\,,
\]
with  ${ \mathcal L}(u,z) = -\i z\sigma_3 + U(x,t)$,  and $ U(x,t)=\begin{pmatrix}
	0 & u(x,t)\\
	-\overline{u(x,t)} & 0
\end{pmatrix}.$

We define the matrix $\Gamma(x; z)$ by

\begin{equation}
	\label{eq:psi_1}
	W(x;z)=\Gamma(x;z){\rm e}^{-\i(x-x_0)(p(z)-E)\sigma_{3}}\,.
\end{equation}
To compute the coefficients, we consider the formal expansion
\begin{equation} \label{eq:Gammak}
	\Gamma(x;z)=\sum_{k=0}^\infty\dfrac{\Gamma_k(x)}{z^k},\quad 
	\Gamma_{k}(x) =
	\begin{bmatrix}
		\alpha_{k}(x)  & \beta_{k}(x)\\
		-\ov \beta_{k}(x)  & \ov\alpha_{k}(x)
	\end{bmatrix},\;k\geq 0,\quad \alpha_0=1,\;\;\beta_0=0\,.
\end{equation}
The background matrix $O(x;z)$ is analytic at infinity and therefore admits a convergent Laurent expansion there.
The symmetry of the coefficients follows from
$\ov{W(x;\ov z)}=\sigma_2W(x;z)\sigma_2$ and
$\ov{p(\ov z)-E}=p(z)-E$ near infinity. This implies
\[
\begin{pmatrix}
\ov{W_{11}(\ov z)}&\ov{W_{12}(\ov z)}\\
\ov{W_{21}(\ov z)}&\ov{W_{22}(\ov z)}\\
\end{pmatrix}
=\begin{pmatrix}
W_{22}(z)&-W_{21}(z)\\
-W_{12}(z)&W_{11}(z)
\end{pmatrix}.
\]
We consider the ratio
\begin{align}
	\label{eq:proof_5.2} 
	(\pa_x  W) W^{-1}=  -\i\left(z+\sum_{k=1}^\infty\dfrac{\pi_k}{z^k}\right)\left( \1 + \sum_{k=1}^{\infty}\frac{\Gamma_{k} }{z^{k}}\right)\sigma_3 \left( \1 + \sum_{k=1}^{\infty}\frac{\tilde{\Gamma}_{k} }{z^{k}}\right) + \left(\sum_{k=1}^{\infty}\frac{\pa_x\Gamma_{k} }{z^{k}}\right)\left( \1 + \sum_{k=1}^{\infty}\frac{\tilde{\Gamma}_{k} }{z^{k}}\right)\,,
\end{align}
where  $p(z)-E=z+\sum_{k=1}^\infty\dfrac{\pi_k}{z^k}$     
and $\tilde{\Gamma}_{k} $ is still a $2 \times 2$ matrix given by the recursive formula
\be
\label{eq:rec-tilde}
\tilde{\Gamma}_{k}= -\Gamma_{k} -\sum_{j=1}^{k-1}\Gamma_{j}\tilde{\Gamma}_{k-j}.
\ee
Expanding  the Laurent series  in equation ~\eqref{eq:proof_5.2} we obtain 
\begin{equation}
	\begin{split}
		\label{eq:Lau-ser}
		&  (\pa_x  W) W^{-1}\sim -\i z \sigma_3 -\i[\Gamma_1 ,\sigma_3]  \\
		& + \sum_{k=1}^{\infty}z^{-k}\left[ \pa_x \Gamma_{k} + \sum_{l=1}^{k-1} \pa_{x}\Gamma_{l}\tilde{\Gamma}_{k-l} -\i \left([\Gamma_{k+1},\sigma_3] + \sum_{j=1}^{k}[\Gamma_j,\sigma_3]\tilde{\Gamma}_{k+1-j}\right)\right]\\
		&-\i \sum_{k=1}^{\infty}\dfrac{\pi_k}{z^{k}}\sigma_3 -\i \sum_{k=2}^{\infty}z^{-k}\sum_{n=1}^{k-1} \pi_{k-n} \left[ \left([\Gamma_{n},\sigma_3] + \sum_{j=1}^{n-1}[\Gamma_j,\sigma_3]\tilde{\Gamma}_{n-j}\right)\right]\,.
	\end{split}
\end{equation}

Since $W_xW^{-1}=-\i z\sigma_3+U(x,t)$, all coefficients of negative powers of $z$ in \eqref{eq:Lau-ser} vanish.
From 	the ZS linear equation  we have $-\i[\Gamma_1,\sigma_3] =U(x,t)=\begin{pmatrix}
	0 & u(x,t)\\
	-\overline{u(x,t)} & 0
\end{pmatrix}$  so that $\beta_1=\frac{u}{2\i}$.

From the coefficient $z^{-1}$ we obtain 
\be
	\label{eq:lemma-5.1}
	\i\pa_{x} \Gamma_{1}=-  [\Gamma_{2},\sigma_{3}] + [\Gamma_{1},\sigma_3]\Gamma_{1} - \pi_{1}\sigma_3, 
		\ee
		that implies
\[
\i\pa_x \alpha_1=2\beta_1 \ov \beta_1- \pi_1=\dfrac{|u(x,t)|^2}{2}-\pi_1,
\]
	and
	\begin{equation}
	\label{beta_rec}
\i\pa_x \beta_1=2\beta_2-2\beta_1 \ov \alpha_1.
\end{equation}
From the coefficient $z^{-2}$ we obtain 
\begin{align}
		k=2 \qquad & \pa_x \Gamma_{2} + \pa_x \Gamma_1 \tilde{\Gamma}_1 -\i\left[[\Gamma_3,\sigma_3]+[\Gamma_1,\sigma_3]\tilde{\Gamma}_2 +[\Gamma_2,\sigma_3]\tilde{\Gamma}_1 +\pi_2\sigma_3+\pi_1[\Gamma_{1},\sigma_3]  \right]= 0\,.
	\end{align}
	From the equation ~\eqref{eq:rec-tilde} we get that $\tilde{\Gamma}_1= - \Gamma_1$ and $\tilde{\Gamma}_2= -\Gamma_2 + (\Gamma_1)^2$, so we have that 
	\begin{align}    \pa_x \Gamma_2 =\i([\Gamma_3,\sigma_3] - [\Gamma_1,\sigma_3]\Gamma_2+\pi_2\sigma_3+
\pi_1\Gamma_{1}\sigma_3).    \end{align}
		The above equation implies
        \bse
\begin{gather}
	\label{alpha_rec2}
	 \pa_x \alpha_2 = -2\i \beta_1 \ov \beta_2 +\i\sum_{j=1}^2\pi_j\alpha_{2-j}, \\
     \label{beta_rec2}
	 \pa_x \beta_2 = -2\i \beta_3  +2\i \beta_1\ov{\alpha_2}-\i\pi_1 \beta_1.  \end{gather} 
    \ese

Similarly, at $k=3$, we have
\be
\partial_x\Gamma_3=\i\left([\Gamma_4,\sigma_3]-[\Gamma_1,\sigma_3]\Gamma_3
+\pi_3\sigma_3+\pi_1\Gamma_2\sigma_3+\pi_2\Gamma_1\sigma_3\right),
\ee
which implies
\begin{align}
\pa_x \alpha_3&=-2\i \beta_1 \ov{\beta_3}+\i\sum_{j=1}^3\pi_j\alpha_{3-j},\\
\pa_x \beta_3&=-2\i \beta_4+2\i \beta_1\ov{\alpha_3}-\i\pi_2\beta_1-\i\pi_1\beta_2.
\end{align}

For $u\in C^1(\R)$, these recurrences define the formal coefficients
$\alpha_j$ and $\beta_j$, $j=0,1,2$, once the integration constants
in the diagonal coefficients have been fixed.
For the background and Jost solutions, these constants are fixed
by their respective normalizations.

\subsection{Expansion of $b(z)$ as $z\to\infty$}\label{s:asymptoticsb}
By definition we have 
\[
b(z)=\det[W^r_1(x;z),W^\ell_1(x;z)]=W^r_{11}(x;z)W^\ell_{21}(x;z)-W^r_{21}(x;z)W^\ell_{11}(x;z)\,.
\]
Assume that $u_0-u_e^\ell\in\mathcal{W}^{4,1}(\R^-)$ and $u_0-u_e^r\in\mathcal{W}^{4,1}(\R^+)$. For $s\in\{\ell,r\}$, we use the normalized expansions
\[
W^s(x;z)
\e^{\i(x-x_0^s)(p^s(z)-E^s)\sigma_3}
=\sum_{k=0}^{3}\frac{\Gamma_k^s(x)}{z^k}
+\mathcal{O}(z^{-4}),
\qquad z\in\R,\quad |z|\to\infty,
\]
where $\Gamma_0^s=I$ and the coefficients have the form
\eqref{eq:Gammak}.
Since $b(z)$ is independent of $x$, we evaluate it at $x=x_0^\ell$.
All coefficients in the following expansion are evaluated at this point.
Using $\beta_1^\ell=\beta_1^r=u_0/(2\i)$, we expand the determinant
through $z^{-3}$ with an $\mathcal{O}(z^{-4})$ remainder:
\be\label{expansionb}
\begin{split}
&b(z)\e^{\i(x_0^\ell-x_0^r)z}=\frac{\ov \beta_1^r-\ov \beta_1^\ell}{z}+\frac{1}{z^2}(\ov\beta_1^r\alpha_1^\ell-\ov \beta_1^\ell\alpha_1^r+\ov\beta_2^r-\ov\beta_2^\ell-\i(x_0^\ell-x_0^r)\pi_1^r(\ov \beta_1^r-\ov \beta_1^\ell))\\
&+\frac{1}{z^3}\Bigg(\ov{\beta}_1^\ell(\alpha_2^\ell-\alpha_2^r)+(\alpha_1^\ell\ov{\beta}_2^r-\alpha_1^r\ov{\beta}_2^\ell)+(\ov{\beta}_3^r-\ov{\beta}_3^\ell)
+\i (x_0^\ell-x_0^r)(\alpha_1^r\ov{\beta_1^\ell}+\ov{\beta}_2^\ell-\alpha_1^\ell\ov{\beta}_1^r-\ov{\beta}_2^r)\pi_1^r\Bigg.\\
&+\frac12(x_0^\ell-x_0^r)^2
(\ov{\beta}_1^\ell-\ov{\beta}_1^r)(\pi_1^r)^2
+\i(x_0^\ell-x_0^r)
(\ov{\beta}_1^\ell-\ov{\beta}_1^r)\pi_2^r
\Bigg)
+O(z^{-4}).
\end{split}
\ee
Since $\pi_1^\ell$ and $\pi_1^r$ are real, the diagonal
recurrence gives
\[
\partial_x(\ov\alpha_1^\ell-\ov\alpha_1^r)
=-\i(\pi_1^\ell-\pi_1^r).
\]
Together with \eqref{beta_rec} and \eqref{beta_rec2}, this yields
\bse\label{expansions}
\begin{align}
&\beta_2^\ell-\beta_2^r=\frac{\i}{2} \pa_x( \beta^\ell_{1} - \beta^r_{1})+ \beta^\ell_1(\ov\alpha_{1}^\ell-\ov\alpha_{1}^r),\\
&\beta_3^\ell-\beta_3^r=\frac{\i}{2}\partial_x(\beta_2^\ell-\beta_2^r)+\beta_1^\ell(\ov \alpha_2^\ell-\ov \alpha_2^r)-\frac{1}{2}\beta_1^\ell(\pi_1^\ell-\pi_1^r)=(\beta_2^\ell-\beta_1^\ell\ov\alpha_1^\ell)
(\ov\alpha_1^\ell-\ov\alpha_1^r)
+\beta_1^\ell(\ov\alpha_2^\ell-\ov\alpha_2^r).
\end{align}
\ese
Substituting \eqref{expansions} into \eqref{expansionb},
the coefficients of $z^{-1}$, $z^{-2}$ and $z^{-3}$ vanish.
Therefore,
\be\label{b-asym}
	b(z)\e^{\i(x_0^\ell-x_0^r)z}
	=\mathcal{O}(z^{-4}),
	\qquad z\in\R,\quad |z|\to\infty.
\ee

\subsection{Proof of Proposition \ref{prop:a.representation}}\label{app:a.representation}

Define the function 
\begin{equation}\label{e:defF}
	F(z) = \begin{dcases}
		a(z) \e^{-\i (x_0^\ell-x_0^r) z} & z \in \C^+ \setminus \Sigma^{+} \\
		a^*(z)^{-1} \e^{-\i (x_0^\ell-x_0^r) z} & z \in \C^- \setminus \Sigma^{-} \\
	\end{dcases}
\end{equation}
Using the properties \eqref{e:scat.asympt}, \eqref{jumpS}
and \eqref{e:ajump} of $a(z)$, it follows immediately that
\begin{enumerate}
	\item $F(z)$ is analytic in $\C \setminus (\R \cup \Sigma_{\mathrm{tot}})$ and satisfies $F^*(z) = F(z)^{-1}$.
	\item $F(z) = 1 + \bigo{z^{-1}}$ as $z \to \infty$.
	\item For $z \in \R \cup \Sigma_{\mathrm{tot}}$, the boundary values of $F(z)$ satisfy the jump relations:
	\begin{align}
	\label{F.ratio.jump}
		\frac{F_+(z)}{F_-(z)}  &= \begin{dcases}
			-\frac{b_1(z_-)}{b_1(z_+)} & z \in  \Sigma^{r,+} \setminus \Sigma^{\ell,+} \\
			-\frac{b_2(z_-)}{b_2(z_+)} & z \in \Sigma^{\ell,+} \setminus  \Sigma^{r,+} \\
			1-|b(z)|^2 & z \in \R \\
			\left(-\frac{b_1^*(z_-)}{b_1^*(z_+)}\right)^{-1} & z \in \Sigma^{r,-} \setminus \Sigma^{\ell,-} \\
			\left(-\frac{b_2^*(z_-)}{b_2^*(z_+)} \right)^{-1} & z \in \Sigma^{\ell,-} \setminus \Sigma^{r,-}
		\end{dcases} \\
	\label{F.prod.jump}
		F_+(z) F_-(z) &= \begin{dcases}
			(1 - b_1(z_+) b_2(z_+) ) \e^{-2\i x_0^\ell(z+E^\ell) + 2\i x_0^r (z+E^r)} & z \in \Sigma^{\ell,+} \cap  \Sigma^{r,+} \\
			(1 - b_1^*(z_+) b_2^*(z_+) )^{-1} \e^{-2\i x_0^\ell(z+E^\ell) + 2\i x_0^r (z+E^r)} & z \in \Sigma^{\ell,-} \cap \Sigma^{r,-}  
		\end{dcases}
	\end{align}
	\item For each $p \in \partial \Sigma_{\pm}$, $F(z) (z - p)^{\pm 1/4} = \bigo{1}$ as $z \to p$. It is bounded and nonzero at all other points in $\C$. 
\end{enumerate}	

By Remark~\ref{rmk:windinga}, $a$ has zero winding number
around each independent closed contour in
$\C^+\setminus\Sigma^+$.
 Since $F(z)$ in \eqref{e:defF} has the same winding number as $a(z)$, we can choose a single-valued analytic
logarithm of $F(z)$ satisfying
\be
\log F(z)=\mathcal{O}(z^{-1}), \qquad z\to\infty.
\ee

\textit{Case 1: $\Sigma^\ell \cap \Sigma^r$ is empty or discrete.}
Using the  jumps relations   for $\log a(z)$    in \eqref{nu}, we obtain from  Plemelj formula 
\begin{multline}\label{tildeF.def}
F(z) := \exp \left[ \frac{1}{2\pi \i } \left( 
  \int_{\Sigma^{r,+}} \log \left(- \frac{ b_1(\xi_-)}{b_1(\xi_+)} \right) \frac{\d \xi}{\xi-z}
+\int_{\Sigma^{\ell,+}} \log \left(- \frac{ b_2(\xi_-)}{b_2(\xi_+)} \right) \frac{\d \xi}{\xi-z}
+\int_\R \log( 1 - |b(\xi)|^2) \frac{ \d \xi}{\xi-z} 
\right. \right. \\ \left. \left.
 -\int_{\Sigma^{r,-}} \log \left(- \frac{ b_1^*(\xi_-)}{b_1^*(\xi_+)} \right) \frac{\d \xi}{\xi-z} 
 -\int_{\Sigma^{\ell,-}} \log \left(- \frac{ b_2^*(\xi_-)}{b_2^*(\xi_+)} \right) \frac{\d \xi}{\xi-z}
\right) \right].
\end{multline}

\textit{Case 2: $\Sigma^{\ell,+} \cap  \Sigma^{r,+}$ is an arc.}
By Proposition~\ref{prop:verticaloverlap}, the intersection
is a single nontrivial segment. $\Sigma^{\ell,+} \cap  \Sigma^{r,+}$ begins at the initial point $z_1^s$ of the segment $\Sigma^{s,+}$ and terminates at the terminal endpoint $z_2^{s'}$ of the segment  $\Sigma^{s',+}$ for $ s, s'  \in  \{ \ell ,r \}$.

Define
\[
\Lambda(z) =   \left(\prod_{p \in \partial( \Sigma^{\ell,+} \cap  \Sigma^{r,+}) } (z-p)(z-\overline p) \right)^{1/2} , \qquad z \in \C \setminus (\Sigma^r \cap \Sigma^\ell) 
\]
analytic in $\C \setminus (\Sigma^r \cap \Sigma^\ell)$, normalized such that %$w(z) = 1 + \bigo{z^{-1}}$ and 
$z^{-2} \Lambda(z) = 1 + \bigo{z^{-1}}$ as $z \to \infty$. Observe that %$w^*(z) = w(z)^{-1}$ while 
$\Lambda^*(z) = \Lambda(z)$.  Let us consider 
\be\label{e:defG}
G(z)=\frac{\log F(z)+\i x_0^\ell(z+E^\ell)
	-\i x_0^r(z+E^r)}{\Lambda(z)}.
\ee
Since $\Lambda(z)=z^2+\mathcal{O}(z)$ as $z\to\infty$,
we have $G(z)=\mathcal{O}(z^{-1})$. Moreover, using the
jump relations of $F$, with the logarithms in the definition
of $\nu$ in \eqref{nu} chosen consistently with $\log F$, we obtain
\be
G_+(z)-G_-(z)=
\begin{cases}
	\displaystyle\frac{\nu(z)}{\Lambda(z)},
	& z\in(\R\cup\Sigma_{tot})
	\setminus(\Sigma^\ell\cap\Sigma^r),\\[2mm]
	\displaystyle\frac{\nu(z)}{\Lambda_+(z)},
	& z\in\Sigma^\ell\cap\Sigma^r.
\end{cases}
\ee
Applying the Cauchy
integral formula to $G$, we obtain
\be\label{e:CauchyG}
G(z)=\frac{1}{2\pi\i}\left(
\int_{\Sigma^\ell\cap\Sigma^r}
\frac{\nu(\xi)}{\Lambda_+(\xi)}
\frac{\d\xi}{\xi-z}
+
\int_{(\R\cup\Sigma_{tot})
	\setminus(\Sigma^\ell\cap\Sigma^r)}
\frac{\nu(\xi)}{\Lambda(\xi)}
\frac{\d\xi}{\xi-z}
\right)
\ee
and therefore
\begin{equation}
\label{tildeF.def2} 
F(z) = \exp \left[ G(z) \Lambda(z)\right]\e^{-\i x_0^\ell(z+E^\ell) + \i x_0^r (z+E^r)}\,.
\end{equation}
 Standard Cauchy estimates give at most quarter root singularities to $F$ at the end points of all the paths we integrate.

%%%%%%%%%%%%%%%%%%%%%%%%%%%%%%%%%%%%%%%%%%%%%%%%%%%%%%%%%%

\subsection{Proof of Proposition \ref{prop:a.factorization}}\label{a:proof3.7}
\begin{proof}
	By Assumption~\ref{assume} and $a=a_1a_2$, both factors
	are nonzero off the cuts. The analyticity and normalization
	of $\Phi$ follow from \eqref{renormalization},
	\eqref{e:asym} and RHP~\ref{RHP:M}, while its Schwarz
	symmetry follows from the reflected definition in the
	lower half-plane. Moreover,
	\[
	\Phi(z;x,0)=
	\left[
	\frac{W_1^\ell(x;z)}{a_1(z)},
	\frac{W_2^r(x;z)}{a_2(z)}
	\right]\e^{\i(x-x_0^r)z\sigma_3},
	\qquad z\in\C^+\setminus\Sigma^+.
	\]
	The endpoint estimates in Proposition~\ref{propm} and
	the bounds on $a_j^{-1}$ give
	$\Phi(z;x,0)=\mathcal O(|z-p|^{-1/4})$
	at every upper spectral endpoint.
	The lower endpoints follow by Schwarz symmetry.
	It remains to verify the jumps. The jump relation for $\Phi(z;x,0)$ is $\Phi_+(z;x,0)=\Phi_-(z;x,0)\tilde{J}(z;x)$, where in the upper half-plane, it is given by
	\be\label{tildeJ}
			\tilde{J}(z;x)=a_2(z_-)^{-\sigma_3} J(z;x)a_2(z_+)^{\sigma_3}, \qquad z\in\Sigma^{\ell,+}\cup\Sigma^{r,+},
	\ee
	where $J$ is given by \eqref{jumpM}. Next we will simplify the jump in \eqref{tildeJ}.
	
	We first consider the overlapping case when $z\in\Sigma^{r,+}\cap\Sigma^{\ell,+}$. By $\frac{a_1(z_+)a_2(z_-)}{a_1(z_-)a_2(z_+)}=-\frac{b_1(z_+)}{b_1(z_-)}$ in \eqref{e:jumpa12}, we have
	\begin{align}\label{eq:diagov}
		R_1(z_+)\frac{a_2(z_+)}{a_2(z_-)}=-R_1(z_-)\frac{a_2(z_-)}{a_2(z_+)}.
	\end{align}
	Thus the two diagonal entries of \eqref{tildeJ} coincide. 
Using $\det S(z_-)=1$ and the relations
\eqref{jumpS} and \eqref{e:jumpa12}, we obtain
\be\label{eq:ov}
\left(a_1(z_+)a_2(z_-)\right)^2
-\left(\i b_1(z_+)\e^{2\i x_0^rE^r}\right)^2
=
\frac{a_1(z_+)a_2(z_-)}
{a_1(z_-)a_2(z_+)}
\e^{-2\i x_0^\ell E^\ell+2\i x_0^rE^r}
\ne0.
\ee
With the relations in  \eqref{e:jumpa12} and \eqref{jumpS}, we have 
\begin{align}\label{eq:r1r2-1}
	r_1(z)r_2(z)=\frac{a_1(z_+)a_2(z_-)}{a_1(z_-)a_2(z_+)}\frac{\e^{-2\ii x_0^\ell E^\ell}\e^{2\ii x_0^rE^r}}{\left(a_1(z_+)a_2(z_-)+\ii b_1(z_+)\e^{2\ii x_0^rE^r}\right)^2}. 
\end{align}
Hence the common denominator in \eqref{e:r12} is nonzero.
	Substitution into \eqref{eq:r1r2-1}  then yields 
	\begin{align}\label{eq:r12ov}
		r_1(z)r_2(z)=\frac{a_1(z_+)a_2(z_-)-\ii b_1(z_+)\e^{2\ii x_0^rE^r}}{a_1(z_+)a_2(z_-)+\ii b_1(z_+)\e^{2\ii x_0^rE^r}}.
	\end{align}
	In particular,
	\[
	1+r_1(z)r_2(z)
	=
	\frac{2a_1(z_+)a_2(z_-)}
	{a_1(z_+)a_2(z_-)+\i b_1(z_+)\e^{2\i x_0^rE^r}}
	\ne0.
	\]
Equations \eqref{eq:diagov} and \eqref{eq:r12ov} give
\[
\tilde{J}_{11}(z)=\tilde{J}_{22}(z)=\frac{1-r_1(z)r_2(z)}{1+r_1(z)r_2(z)}.
\]
Moreover, \eqref{e:r12} and \eqref{eq:r12ov} yield 
	\be\label{eq:12ov}
			\tilde{J}_{12}(z)=\frac{\ii\e^{2\ii x_0^rE^r+2\ii x_0^rz}}{a_2(z_-)a_2(z_+)}\e^{-2\ii xz}
			=\frac{2\i r_2(z)} {1+r_1(z)r_2(z)}\e^{-2\i xz}.
		\ee
Since $a_2(z_\pm)$ are finite and nonzero, this entry is nonzero in the interior of the overlap. Since the transformed jump has determinant one,
\[
\tilde{J}_{21}(z)=\frac{\tilde{J}_{11}(z)\tilde{J}_{22}(z)-1}{\tilde{J}_{12}(z)}
=\frac{2\i r_1(z)}{1+r_1(z)r_2(z)}\e^{2\i xz}.
\]
Thus, on $\Sigma^{r,+}\cap\Sigma^{\ell,+}$, the transformed jump coincides with the corresponding jump matrix $V(z;x,0)$ in \eqref{e:Vtilde}.

	%%%%%%%%%%%%%%%%%%%%%%%%%%%%%%%%%%%%%%%%%%%%%%%%%%%%%%%%%%%%
	On $\Sigma^{\ell,+}\setminus\Sigma^{r,+}$, the function $a_2$
	has no jump, and \eqref{a+symmetry} gives
	\[
	a_1(z_+)a_2(z_-)
	=a(z_+)
	=-\i b_2(z_-)\e^{-2\i x_0^\ell E^\ell}.
	\]
	Substituting into \eqref{e:defr1}, and arguing similarly on
	$\Sigma^{r,+}\setminus\Sigma^{\ell,+}$, gives
	\begin{equation}\label{e:r1}
		r_1(z)
		=\frac{\e^{-2\i x_0^r z-2\i x_0^\ell E^\ell}}
		{2a_1(z_-)a_1(z_+)},
		\qquad
		z\in\Sigma^{\ell,+}\setminus\Sigma^{r,+},
	\end{equation}
	and
	\begin{equation}\label{e:r2}
		% Keep the existing label of (A.101).
		r_2(z)
		=\frac{\e^{2\i x_0^r z+2\i x_0^r E^r}}
		{2a_2(z_-)a_2(z_+)},
		\qquad
		z\in\Sigma^{r,+}\setminus\Sigma^{\ell,+}.
	\end{equation}
	Substituting \eqref{e:r1} and \eqref{e:r2} into \eqref{tildeJ}, and using \eqref{a+symmetry} and \eqref{eq:reflections}, yields the corresponding triangular jumps $V(z;x,0)$ in \eqref{e:Vtilde} on $\Sigma^{\ell,+}\setminus\Sigma^{r,+}$ and $\Sigma^{r,+}\setminus\Sigma^{\ell,+}$, respectively.

On the real axis, \eqref{renormalization} gives
\[
\widetilde J(z;x)
=(a_2^*(z))^{\sigma_3}J(z;x)a_2(z)^{\sigma_3}.
\]
Using $a=a_1a_2$ and the definitions of $R$ and $\rho$,
we obtain
\[
\rho(z)=\frac{a_2(z)}{a_2^*(z)}
R(z)\e^{-2\i x_0^r z}.
\]
Since $|a_2(z)|=1$ on $\R$, we have
$a_2^*(z)a_2(z)=1$ and $|\rho(z)|=|R(z)|$.
Substitution of the real-axis jump in \eqref{jumpM}
therefore gives $V(z;x,0)$.
The lower-band jumps follow by Schwarz symmetry.
		
\end{proof}

\subsection{Proof of Proposition~\ref{prop:afactor}}\label{a:proof3.8}
\begin{proof}
	Multiplying the defining expressions for $a_1$ and $a_2$ gives $a_1a_2=a$ by \eqref{a.rep.disjoint} and \eqref{a.rep.overlap}. We now verify the jump relations \eqref{e:jumpa12}.
	For convenience, we define
	\begin{equation}\label{e:chi}
	\chi(z):=\exp\left\{
	\i(x_0^\ell-x_0^r)z
	+\frac{1}{2\pi \i}
	\int_{\mathbb R\cup\Sigma^\ell\cup\Sigma^r}
	\frac{\widetilde\nu(\xi)}{\xi-z}\d\xi
	\right\}.
	\end{equation}
	Consequently,
	\begin{equation}\label{chi=a12}
 \frac{a_1(z)}{a_2(z)}=\chi(z).
	\end{equation}

	On the non-overlapping parts of the upper bands, the logarithmic jump of $a$ is $\nu$, as established in Appendix~\ref{app:a.representation}, while the Plemelj formula gives
	\[
	\frac{\chi(z_+)}{\chi(z_-)}=\e^{\widetilde\nu(z)}.
	\]
	Since $\widetilde\nu=\nu$ on $\Sigma^{\ell,+}\setminus \Sigma^{r,+}$ and $\widetilde\nu=-\nu$ on $\Sigma^{r,+}\setminus\Sigma^{\ell,+}$, the Plemelj formula applied directly to the defining exponentials yields
	\[
	\begin{aligned}
		\frac{a_1(z_+)}{a_1(z_-)}
		&=\e^{\nu(z)},
		& a_2(z_+)&=a_2(z_-),
		&& z\in\Sigma^{\ell,+}\setminus \Sigma^{r,+}; \qquad
		a_1(z_+)=a_1(z_-),
		& \frac{a_2(z_+)}{a_2(z_-)}
		&=\e^{\nu(z)},
		&& z\in\Sigma^{r,+}\setminus\Sigma^{\ell,+}.
	\end{aligned}
	\]
	By the definition of $\nu$, these are the first two jump relations in \eqref{e:jumpa12}. 
	
	On the overlap, the same Plemelj formula and the definition of $\widetilde\nu$ give
	\[
	\frac{a_1(z_+)a_2(z_-)}{a_1(z_-)a_2(z_+)}
	=\frac{\chi(z_+)}{\chi(z_-)}
	=\e^{\widetilde\nu(z)}
	=-\frac{b_1(z_+)}{b_1(z_-)},
	\qquad z\in\Sigma^{\ell,+}\cap \Sigma^{r,+}.
	\]
	Analytic continuation across the complementary bands follows
	from the local cancellation argument in
	Appendix~\ref{a:proof3.9}, Case~1.
	The boundary continuity and endpoint bounds required in
	Proposition~\ref{prop:a.factorization} are established
	in the same appendix.
	
Next, the Cauchy integral in \eqref{e:chi} is
$\mathcal O(z^{-1})$ as $z\to\infty$.
By the defining formulas in Proposition~\ref{prop:afactor},
with the prescribed branch normalization in the overlapping
case, the full exponents defining $a_1$ and $a_2$ are
$\i(x_0^\ell-x_0^r)z+\mathcal O(z^{-1})$
and $\mathcal O(z^{-1})$, respectively. Hence
\be
a_1(z)=\e^{\i(x_0^\ell-x_0^r)z}
\left(1+\mathcal O(z^{-1})\right),\qquad
a_2(z)=1+\mathcal O(z^{-1}),\qquad z\to\infty.
\ee
This proves \eqref{e:asym}.

	Finally, we show that $a_1$ and $a_2$ satisfy \eqref{e:a12mod}.
	The symmetry of $\widetilde\nu$ and the contour orientations imply from \eqref{e:chi} that
	$\chi(z_-)=(\overline{\chi(z_+)})^{-1}$ for $z\in\R$.
	Hence the Plemelj formula gives
	\[
	|\chi(z_+)|^2=\frac{\chi(z_+)}{\chi(z_-)}=\e^{\widetilde\nu(z)}=1-|b(z)|^2=|a(z)|^2,\qquad z\in\R,
	\]
	where the last equality follows from $\det S(z)=1$.
	Taking absolute values in $a_1a_2=a$ and	\eqref{chi=a12} yields
	\[
	|a_1(z)|\,|a_2(z)|=|a(z)|,\qquad \frac{|a_1(z)|}{|a_2(z)|}=|\chi(z_+)|=|a(z)|,\qquad z\in\R.
	\]
	These identities imply $|a_1(z)|=|a(z)|$ and $|a_2(z)|=1$ on $\R$, proving the claimed modulus identities.
\end{proof}

\subsection{Proof of Proposition~\ref{p:singularitya12}}\label{a:proof3.9}
\begin{proof}
	We consider the non-overlapping and overlapping arc configurations
	with distinct left and right spectral endpoints.

	First, away from all spectral endpoints, Proposition~\ref{propm}
	and \eqref{e:scoefs} give local Lipschitz continuity of
	$\widetilde\nu$. Its Cauchy transform therefore has continuous
	boundary values there. Together with Theorem~\ref{theorem1},
	\eqref{e:chi} and \eqref{chi=a12} give the required continuity of
	$a_1$ and $a_2$ away from the endpoints.
	It remains to determine their behavior at those endpoints.
	
		The normalized endpoint limits of $a$ are nonzero.
	Indeed, \eqref{e:W.endpoints},
	\eqref{e:jostendpointholder} and local Lipschitz
	continuity of the other Jost column show that a zero
	normalized limit would imply
	$a(z_\pm)=\mathcal O(|z-p|^{1/4})\to0$,
	contrary to Assumption~\ref{assume}.
	Hence $a(z)^{-1}=\mathcal O(|z-p|^{1/4})$
	near every upper spectral endpoint.

	\medskip
	\noindent\textbf{Case 1: an endpoint outside the other band.}
	Let $p=z_j^\ell\notin\Sigma^{r,+}$.
	We first show that $a_2$ is analytic and nonzero near $p$, and the estimate for $a_1$ will then follow from $a_1a_2=a$.
	If the bands overlap elsewhere, $\Lambda$ is analytic and nonzero
	near $p$, and $\widetilde\nu=\nu$ locally. Thus the local part of
	the exponent in \eqref{e:defa2} has kernel
	\[
	\left(\frac{\Lambda(z)}{\Lambda(\xi)}\nu(\xi)
	-\widetilde\nu(\xi)\right)\frac{1}{\xi-z}
	=\frac{\Lambda(z)-\Lambda(\xi)}
	{\Lambda(\xi)(\xi-z)}\nu(\xi).
	\]
	The apparent singularity at $z=\xi$ is removable. Since $\nu$ is
	locally integrable, this local integral is analytic near $p$;
	the remaining terms are analytic there as well.
	For disjoint bands the same conclusion follows directly from
	the representation in Proposition~\ref{prop:afactor}.
	Consequently,
	\[
	a_2(p)\ne0,\qquad a_2(z)^{\pm1}=\mathcal O(1),\qquad
	a_1(z)=\frac{a(z)}{a_2(z)}
	=\mathcal O(|z-p|^{-1/4}),
	\]
	where the last estimate follows from Theorem~\ref{theorem1}.
	At $p=z_j^r\notin\Sigma^{\ell,+}$, use
	$\widetilde\nu=-\nu$ in \eqref{e:defa1}, and  the roles of
	$a_1$ and $a_2$ are exchanged.
	
	\medskip
	\noindent\textbf{Case 2: an endpoint inside the other band.}
	Let $p=z_2^\ell\in\operatorname{int}(\Sigma^{r,+})$.
	We first prove that $a_2$ has a finite, nonzero limit on each side
	of $p$. The estimate for $a_1$ will then follow from $a_1a_2=a$.
	
	Fix $x$ and choose $\delta>0$ so that $B_\delta(p)$ contains no
	other spectral endpoint. Set
	$I=\Sigma^{r,+}\cap\overline{B_\delta(p)}$, with its inherited
	orientation, and let $q$ be its terminal endpoint.
	The overlap ends at $p$, and the right-only portion runs from $p$
	to $q$. 
	By \eqref{e:chi} and \eqref{chi=a12},
	\be\label{a2^2}
	a_2(z)^2=\frac{a(z)}{\chi(z)}=a(z)\exp\left\{-\i(x_0^\ell-x_0^r)z-\frac{1}{2\pi\i}\int_{\R\cup\Sigma_{\mathrm{tot}}}\frac{\widetilde\nu(s)}{s-z}\,\d s\right\}.
	\ee
To determine the limits of $a_2^2$ at $p$, we separate the jump of the density. Set
	\be\label{e:f}
	f(s)=
	\begin{cases}
		\widetilde\nu(s),&s\in I\cap\Sigma^{\ell,+},\\[1mm]
		\widetilde\nu(s)-\dfrac{\i\pi}{2},
		&s\in I\setminus\Sigma^{\ell,+}.
	\end{cases}
	\ee
	The  H\"older regularity
	of the normalized factors shows that $f$ is H\"older continuous
	with exponent $1/2$ through $p$.
	
	By the definition of $f$ in \eqref{e:f},
	\be\label{integralv}
	-\frac{1}{4\pi\i}\int_I
	\frac{\widetilde\nu(s)}{s-z}\,\d s
	=
	-\frac{1}{4\pi\i}\int_I\frac{f(s)}{s-z}\,\d s
	+\frac18\log(p-z)-\frac18\log(q-z).
	\ee
Substituting \eqref{integralv} into \eqref{a2^2}, and using the simple zero of $(\gamma^\ell)^4$ at $p$, gives
	\[
	a_2(z)^2=\gamma^\ell(z)a(z)\exp\left\{-\frac{1}{2\pi\i}\int_I\frac{f(s)}{s-z}\,\d s+2h(z)\right\},
	\]
	where $h$ is analytic near $p$ on each side.
	
	Since $f$ is H\"older continuous near the interior point $p$ of $I$,
	the Plemelj--Privalov theorem gives continuous extensions of its
	Cauchy integral to $p$ from each side. 
	Together with the finite, nonzero limits of $\gamma^\ell a$, this proves that $a_2^2$ has finite, nonzero limits on each side of $p$. Since $a_2$ is analytic on each side, it agrees with one fixed local square root of $a_2^2$ near $p$. Hence
	\[
	a_2(z)\to a_2(p_\pm)\ne0
	\]
	from the corresponding side, with matching limits along the
	overlapping and right-only portions.
	
	Finally, $a_1=a/a_2$ and \eqref{e:def.gamma} give
	\[
	a_1(z)=\mathcal O\bigl(|z-p|^{-1/4}\bigr),
	\qquad a_2(z)^{\pm1}=\mathcal O(1).
	\]
	
	At $p=z_1^\ell\in\operatorname{int}\Sigma^{r,+}$,
	the overlap starts at $p$. We subtract $\i\pi/2$ from
	$\widetilde\nu$ on the portion of $I$ after $p$ and
	replace $\gamma^\ell$ by $(\gamma^\ell)^{-1}$.
	The same argument applies.
	At $p=z_j^r\in\operatorname{int}\Sigma^{\ell,+}$,
	we use $a_1^2=a\chi$ and the density step $-\i\pi/2$, and
	the roles of $a_1,a_2$ and $\ell,r$ are exchanged.
	
	Since $a_1a_2=a$ and the factor corresponding to the
	other background has a finite nonzero limit,
	these estimates also give
	$a_j(z)^{-1}=\mathcal O(1)$, $j=1,2$,
	at every upper spectral endpoint.

\end{proof}

\section*{Acknowledgments}
  We thank Marco Bertola for insightful  discussions during the preparation of this manuscript. 
  We also thank Xiaodong Zhu for valuable discussions and comments. His related work \cite{Zhu} was prepared in parallel with the present manuscript.  

 T. Grava and Z. Zhang acknowledge the support of PRIN 2022 (2022TEB52W)  "The charm of integrability: from nonlinear waves to random matrices"-– Next Generation EU grant – PNRR Investimento M.4C.2.1.1 - CUP: G53D23001880006; the GNFM-INDAM group and the  research project Mathematical Methods in NonLinear Physics (MMNLP), Gruppo 4-Fisica Teorica of INFN. R. Jenkins acknowledges the support of the National Science Foundation grant DMS-2307142 and of the Simons Foundation grant MPS-TSM-853620. 7. X.F. Zhang acknowledges the support from the China Scholarship Council.
% \st{ and the Postgraduate Research Innovation Program of Jiangsu Province Grant No. KYCX24-2670}.

%%%%%%%%%%%%%%%%%%%%%%%%%%%%%%%%%%%%%%%%%%%%%%%%%%%%%%%%%%%%%%%%%%%%%%%%%%%%%%%%%%%%%%%%%%%%%%%%%%%%%%%%%%%%%%%%%%%%%%%%%%%%%%%%%%%%%%%
\medskip
\let\em=\it

\makeatletter
\def\@biblabel#1{#1.}
\def\doibase{http://dx.doi.org/}
\def\reftitle#1{``#1''}
\def\booktitle#1{\textit{#1}}
\def\href#1#2{#2}
\makeatother
\small

%%%%%%%%%%%%%%%%%%%%%%%%%%%%%%%%%%%%%%%%%%%%%%%%%%%%%%%%%%%%%%%%%%%%%%%%%%%%%%%%%%%%%%

\end{document}